\documentclass[reqno,11pt]{amsart}
\usepackage{a4wide}
\usepackage{amssymb}
\usepackage{color}
\usepackage{tikz}
\usetikzlibrary{decorations.pathreplacing}
\usepackage{hyperref}

\newcommand{\CR}{\chi_R}

\newcommand{\CL}{\chi_L}
\newcommand{\ZL}{\zeta}
\newcommand{\cR}{\mathcal R}

\newcommand{\loc}{\textnormal{sol}}

\newcommand{\NN}{\mathbb{N}}
\newcommand{\RR}{\mathbb{R}}
\newcommand{\cA}{\mathcal{A}}

\newcommand{\cC}{\mathcal{C}}

\newcommand{\cE}{\mathcal{E}}
\newcommand{\cF}{\mathcal{F}}
\newcommand{\cI}{\mathcal{I}}
\newcommand{\cL}{\mathcal{L}}
\newcommand{\cLt}{\widetilde{\mathcal{L}}}
\newcommand{\cN}{\mathcal{N}}

\newcommand{\cRR}{\mathcal{R}}
\newcommand{\cS}{\mathcal{S}}

\newcommand{\cY}{\mathcal{Y}}
\newcommand{\cZ}{\mathcal{Z}}
\newcommand{\ONE}{\mathbf{1}}
\newcommand{\lesssimD}{\lesssim_\delta}
\newcommand{\kpu}{\kappa_1}
\newcommand{\kpd}{\kappa_2}

\DeclareMathOperator{\sech}{sech}

\DeclareMathOperator{\SPAN}{span}

\newtheorem{theorem}{Theorem}[section]
\newtheorem{lemma}{Lemma}[section]
\newtheorem{proposition}{Proposition}[section]

\theoremstyle{definition}

\theoremstyle{remark}
\newtheorem{remark}{Remark}[section]
\numberwithin{equation}{section}
\begin{document}
\title[Continuum of finite point blowup rates for critical gKdV]
{Continuum of finite point blowup rates for the critical generalized Korteweg-de Vries equation}

\author[Y. Martel]{Yvan Martel}
\address{Laboratoire de mathématiques de Versailles,
UVSQ, Université Paris-Saclay, CNRS, and Institut Universitaire de France,
45 avenue des États-Unis,
78035 Versailles Cedex, France}
\email{yvan.martel@uvsq.fr}

\author[D. Pilod]{Didier Pilod}
\address{Department of Mathematics, University of Bergen, Postbox 7800, 5020 Bergen, Norway }
\email{Didier.Pilod@uib.no}


\begin{abstract}
The generalized Korteweg-de Vries equation (gKdV) is obtained by changing the nonlinearity of the famous (quadratic) Korteweg-de Vries equation. We consider the quintic power nonlinearity, 
for which the gKdV equation is critical for the $L^2$ norm. 
Kenig, Ponce and Vega proved in the 1990's that the associated Cauchy problem is locally well-posed both in the critical space $L^2$ and in the energy space $H^1$.

On the one hand, as a consequence of the local posedness result and of a sharp Gagliardo-Nirenberg inequality proved by Weinstein, any finite energy solution with a mass strictly less than that of the ground state exists globally in time. 

On the other hand, in the early 2000's, Merle and the first author constructed solutions with masses slightly above that of the ground state which blow up in finite time by concentrating a solitary wave. About a decade later, the blow-up dynamics of the critical gKdV equation was revisited in a series of papers by Merle, Rapha\"el and the first author, leading to a classification of finite energy solutions with  masses close to that of the ground state and satisfying in addition a smallness condition in some $L^2$ weighted space. In particular, a notion of stable blowup emerged, with the $L^2$-norm of the gradient of the solution blowing up at the rate $t^{-1}$ as the time $t>0$ approaches the blowup time (taken by convention at $t=0$). In complement, the same authors constructed exotic blowup solutions with various rates of blowup for initial data that do not satisfy the smallness condition in weighted norm. It is worth noting that in these works, blowup occurs through the bubbling of a solitary wave going at infinity in space at the blow-up time;
this behavior is known to be related to blowup rates $t^{-\nu}$ with exponents $\nu\geq \frac12$.

Recently, the authors of the present paper constructed the first example of a finite energy solution that blows up in finite time at a finite point in space, 
as a consequence of a particular blowup rate $t^{-2/5}$, reopening the question of 
the possible blowup rates for the critical gKdV equation.

In the present article, we introduce a more general point of view on the finite point blowup
for this equation.
For any $\nu\in(\frac 37,\frac12)$,
we prove the existence of a finite energy solution $u$ of
the mass critical generalized Korteweg-de Vries equation
on the time interval $(0,T_0]$,
for some $T_0>0$, which blows up at the time $t=0$ and at the point $x=0$  with the rate 
$\|\partial_x u (t)\|_{L^2} \sim t^{-\nu}$.
By construction, this blowup rate is associated to an $H^1$ blowup residue, obtained by passing to the
limit in the solution $u(t)$ as $t\downarrow 0$ after subtracting the singular bubble, 
of the form $r_\alpha(x)= C x^{\alpha -\frac 12}$ for $x>0$ close to the blowup point, where $\alpha=\frac{3\nu-1}{2-4\nu}$.
Note that $\nu=\frac25$ is equivalent to $\alpha=\frac12$, while
the condition $\nu\in(\frac37,\frac12)$ is equivalent to $\alpha>1$.
Thus, this range of $\nu$ corresponds exactly to
the range of $\alpha$ for which the residue $r_\alpha$ belongs to $H^1$.
For this reason, these are the only possible
finite point blowup rates $\nu$ accessible by our construction method.

Finally, we present some open problems regarding the blowup phenomenon for the mass critical
gKdV equation.
\end{abstract}

\maketitle

\section{Introduction}
\subsection{Problem setting}
We consider the generalized Korteweg-de Vries equation
(gKdV) with the critical quintic power nonlinearity 
\begin{equation}\label{eq:KV}
\partial_t u+\partial_x\big(\partial_x^2u+u^5\big)=0,\quad (t,x) \in \RR\times \RR.
\end{equation}
For a solution $u$ of \eqref{eq:KV},  the mass $M(u)$ and the energy $E(u)$ are formally conserved, where
\begin{align*}
M(u)&=\int_{\RR} u^2 \, dx , \\
E(u)&=\frac12 \int_{\RR} (\partial_xu)^2 \, dx-\frac16 \int_{\RR} u^6 \, dx .
\end{align*}
The Cauchy problem for~\eqref{eq:KV} is locally well-posed in the energy space $H^1(\RR)$ by the results obtained in~\cite{KPV}.
For any $u_0 \in H^1(\RR)$, there exists a unique (in a certain class) maximal solution $u$ of~\eqref{eq:KV} in $\cC([0,T), H^1(\RR))$ satisfying $u(0,\cdot)=u_0$.
Moreover, by \cite[Corollary 1.4]{KPV2}, if the maximal time of existence $T>0$ is finite then 
\[
\liminf_{t \uparrow T} (T-t)^{\frac 13}\|\partial_x u(t)\|_{L^2}>0.
\]
We say that the solution blows up in finite time.
In this article, we are interested by the blowup phenomenon.

We recall some other basic facts concerning~\eqref{eq:KV}, as the symmetries of the
equation and the existence of traveling waves.
If $u$ is a solution of~\eqref{eq:KV} then for any $(\lambda,\sigma) \in(0,+\infty) \times \RR$, the function
\begin{equation} \label{scaling:inv}
v(t,x):=\pm \frac 1{\lambda^{\frac12}} u\left(\frac t{\lambda^3}, \frac{x-\sigma}\lambda \right)
\end{equation}
and the function $w(t,x)=u(-t,-x)$ are also solutions of~\eqref{eq:KV}.
Moreover, the function
\[
Q(x)=\left(3\sech^2(2x)\right)^{1/4}
\]
is known to be the unique positive even solution in $H^1(\RR)$ of the equation
\begin{equation}\label{eq:Q5}
Q''+Q^5 = Q,\quad x\in \RR.
\end{equation}
Then, the function $u(t,x)=Q(x-t)$ is a traveling wave solution of \eqref{eq:KV},
also called solitary wave or soliton.
More generally, for any $(\lambda,\sigma) \in(0,+\infty) \times \RR$, the function
\begin{equation}\label{def:sol} 
u(t,x)=\frac 1{\lambda^{\frac12}}Q\left(\frac 1{\lambda} \left( x - \frac t{\lambda^{2}} - \sigma\right)\right)
\end{equation}
is a solution of~\eqref{eq:KV}. 
While the traveling waves given by \eqref{def:sol} are global and bounded solutions, 
there is a close relation between solitons and blowup, as we recall below.

By the sharp Gagliardo-Nirenberg inequality proved in~\cite{We82}
\begin{equation} \label{sharpGN}
\frac 13 \int_{\RR} \phi^6 \le \left(\frac{\int_{\RR} \phi^2}{\int_{\RR} Q^2} \right)^2 \int_{\RR} (\partial_x\phi)^2, \quad \forall \, \phi \in H^1(\RR),
\end{equation}
and the conservation of mass and energy, for
any initial data $u_0 \in H^1(\RR)$ satisfying
\begin{equation}\label{eq:su}
\|u_0\|_{L^2} < \|Q\|_{L^2},
\end{equation}
the corresponding solution of~\eqref{eq:KV} is global and bounded in $H^1(\RR)$.

A natural question after the local well-posedness result  \cite{KPV} and 
the global well-posedness result \cite{We82} in the mass subcritical case \eqref{eq:su} was 
the determination and the description of possible blowup solutions.
This problem has been addressed roughly in three main periods.
Most of the results cited below concern critical or slightly supercritical mass initial data
\begin{equation}\label{mass:close}
\|Q\|_{L^2} \leq  \|u_0\|_{L^2} < (1+\delta) \|Q\|_{L^2},\quad
\mbox{for $0<\delta\ll 1$.}
\end{equation}

In the first period,
using a rigidity property of the gKdV flow around the family of solitons established in \cite{MaMejmpa}, the first proof of blowup in finite or infinite time was given in \cite{Me3} 
for any negative energy $H^1$ initial data satisfying \eqref{mass:close}.
A direct link was also established between the blowup phenomenon and the solitary wave in \cite{MM1}, where it was shown that
any blowup solution $u$ of \eqref{eq:KV} satisfying \eqref{mass:close} has the following bubbling behavior
\begin{equation}\label{bubbling}
\lambda^{\frac 12}(t) u (t, \lambda(t) x + \sigma(t)) \rightharpoonup \pm Q \quad 
\mbox{and} \quad \lambda(t)\to 0 \quad \mbox{as $t\uparrow T$ in $H^1$ weak,}
\end{equation}
where $0<T\leq \infty$ is the (possibly infinite) blowup time.
The above weak convergence shows that the soliton $Q$ is the universal profile for the blowup phenomenon, but also allows various behaviors for the 
parameters $\lambda$ and $\sigma$. At that stage, it was unknown whether blow up could occur in finite or infinite time and what was the precise asymptotic behavior of the parameters $\lambda$ and
$\sigma$ close to the blow-up time.
However, it was proved that
\begin{equation}\label{sig:lam}
\lim_{t\uparrow T} \lambda^2(t) \sigma'(t) =1 .
\end{equation}
Then, a finite time blowup result was proved in \cite{MM2} for initial data with 
spatial decay and negative energy.

In the second period, more precise blowup results were proved, mostly relying on two main new ingredients.
The first ingredient is a refined blowup profile $Q_b=Q+bP_b$, replacing the
simple profile based only on $Q$, where $b$ is a small parameter and $P_b$ is related to the
scaling direction. The introduction of this ingredient was inspired by the previous treatment
of the log-log blowup for the mass critical nonlinear Schrödinger equation, see \cite{MR07}
for a synthetic presentation.
The second ingredient, more specific to the gKdV equation, is a functional 
with combines monotonicity properties of localized variants of the mass and energy for the gKdV equation in the spirit of Kato's early works
\cite{Kato} and a virial-type property in the neighborhood of solitons proved in \cite{MaMejmpa}.
 
More precisely, it was proved in \cite{MMR1} that
there exists a subset of initial data, included and open for $\|\cdot\|_{H^1}$ in the set
\[
\cA = \left\{ u_0 = Q+\varepsilon_0 : \varepsilon_0 \in H^1,\ \|\varepsilon_0\|_{H^1} <\delta_0 \mbox{ and } \int_{x>0} x^{10} \varepsilon_0^2 dx <1\right\}
\]
where $\delta_0>0$ is small,
leading to blow-up solutions of \eqref{eq:KV} such that
\[
\lim_{t\uparrow T} \left(T-t\right) \|\partial_x u(t)\|_{L^2} = C(u_0).
\]
Because of the openness property, we shall call this phenomenon  stable blowup. 
The asymptotic behavior of these solutions at the 
blow-up time $T$ is typical of the bubbling phenomenon. Indeed, there exists a function $r_\star\in H^1$ such that
\begin{equation}\label{uSS}
\lim_{t\uparrow T} \left\{ u(t,\cdot) - \frac 1{\lambda^{\frac 12}(t)}
Q\left(\frac{\cdot-\sigma(t)}{\lambda(t)} \right)\right\} = r_\star \quad \mbox{in $L^2$}
\end{equation}
for some functions $\lambda$, $\sigma$ satisfying 
\[
\lim_{t\uparrow T} \frac{\lambda(t)}{T-t} = \ell
\quad \mbox{and} \quad \lim_{t\uparrow T} \left(T-t\right)\sigma(t) = \frac 1{\ell^2}\,.
\]
The articles \cite{MMR1,MaMeNaRa} provided a classification of all the possible long time
behaviors for initial data in $\cA$, proving in particular that $(T-t)^{-1}$ is
the only possible blowup rate in such case.

Then, \cite{MMR2} settled the case of blowup at the threshold mass $\|u_0\|_{L^2}=\|Q\|_{L^2}$
showing the existence of a unique (up to invariances) blow-up solution $S(t)$ of
\eqref{eq:KV} with the mass of $Q$, blowing up in finite time with the rate $\|S(t)\|_{H^1} \sim C(T-t)^{-1}$ and concentrating a soliton at the position $\sigma(t) \sim C (T-t)^{-1}$ as $t\uparrow T$
(here, by time translation invariance, the value of $T$ is only a matter of convention).
A precise description of the time and space asymptotic behavior of $S$ in~\cite{CoM1}
has also permitted to construct multi-point bubbling in~\cite{CoM2}.
Such multi-bubble solutions provide examples of blowup with an arbitrary number of blow-up bubbles,
all propagating simultaneously to~$\infty$ as $t\uparrow T$.

Because of the restriction $u_0\in \cA$,
the above stable blow-up theory does not describe all blowup solutions with \eqref{mass:close}.
From~\cite{MMR3}, there exists a large class of exotic finite time blow-up solutions of \eqref{eq:KV}, close to the family of solitons, with blow-up rates of the form $\|\partial_x u(t)\|_{L^2} \sim (T-t)^{-\nu}$ for any $\nu\geq \frac {11}{13}$.
Such blow-up rates are generated by the nonlinear interactions of the bubbling soliton with an explicit slowly decaying tail added to the initial data. Because of the tail, these $H^1$ solutions fail to belong to $\cA$ and the classification \cite{MMR1,MaMeNaRa} does not apply.
The value $\frac {11}{13}$ is not optimal in~\cite{MMR3}, see below.
Moreover, \cite{MMR3} also constructs a large family of infinite time blowup solutions.

The third period started with a surprising blowup rate
$\|u(t)\|_{H^1} \approx (T-t)^{-\frac25}$ constructed in \cite{MP24} and
which led to the first example of finite point blowup for \eqref{eq:KV}.
Indeed, recall that for a single bubble blowup solution, by the relation \eqref{sig:lam} between the position $\sigma$ of the bubble and its blowup rate $\lambda$ (assuming it is of the form 
$\lambda(t)\approx(T-t)^\nu$)
the rate $\nu = \frac 12$ is critical:
if $\nu\geq \frac 12$ then the soliton position goes at $\infty$ since $\lim_{T}\sigma= \infty$, while for $\frac 13 < \nu < \frac 12$, the soliton position converges to a finite point.
The existence of bubbling solutions with $\nu= \frac 13$ was ruled out in~\cite{MM1}, but
the range $\frac 13 < \nu < \frac 12$ was left open by all the previously cited works.
As stated by the authors in \cite{MP24}, this result has reopened the question of 
exotic blowup, in particular the question of determining all the possible blowup rates
$\nu$ within the range $(\frac13,\frac{11}{13})$.
The recent article \cite{NM25} proves the existence (with some form of stability) of blowup solutions with any rate $(T-t)^{-\nu}$ for $\nu\in (\frac12,1)$,
thus completing the picture for power blowup rates at an infinite point.

In the present paper, we construct blowup solutions
with the rate $(T-t)^{-\nu}$,
for all $\nu\in (\frac37,\frac12)$.
For the remaining intervals
$(\frac13,\frac25)$ and $(\frac25,\frac37]$, we refer to the discussion in 
Section \ref{S:1.3}.

Figure 1 recapitulates the  values $\nu>0$ for which blowup solutions $u$ of \eqref{eq:KV} with
the blowup rate $\lim_{t\uparrow T} (T-t)^\nu \|\partial_x u \|_{L^2} = C$ for $C>0$, have been constructed.

\medskip 

\hskip -0.5cm
\begin{tikzpicture}[scale=1]
\draw[arrows=->,line width=.4pt,thick](-4,1.5)--(-4,1);
\draw[arrows=->,line width=.4pt,thick](2.1,0.6)--(1.6,0.1);
\draw (3.5,0.7) node{{\footnotesize ``stable'' blowup \cite{MMR1}}};
\draw [dashed] (-7.8,0) -- (-5.3,0);
\draw [thin](-5.13,0) -- (-3.56,0);
\draw [very thick] (0.3,0) -- (6.5,0);
\draw [very thick] (-3.43,0) -- (-2.56,0);
\draw [very thick] (-2.43,0) -- (0.3,0);
\draw (-3.5,0) circle (2pt);
\draw (-2.5,0) circle (2pt);
\draw [dashed, very thick] (6.5,0) -- (7.5,0);
\draw (-5.2,0) circle (2pt);
\filldraw (1.5,0) circle (2pt);
\filldraw (-4,0) circle (2pt);
\draw (-7.8,0.1) -- (-7.8,-0.1);
\draw (-7.8,0.6) node{{\footnotesize $0$}};
\draw (-2.5,0.6) node{{\footnotesize $\frac 12$}};
\draw (-4,0.6) node{{\footnotesize $\frac 25$}};
\draw (-3.5,0.6) node{{\footnotesize $\frac 37$}};
\draw (-5.2,0.6) node{{\footnotesize $\frac 13$}};
\draw (1.5,0.6) node{{\footnotesize $1$}};
\draw (7,0.6) node{{\footnotesize $\nu\to\infty$}};
\draw (-4,1.8) node{{\footnotesize \cite{MP24}}};
\draw [decorate,decoration={brace,amplitude=10pt}] (-3.5,1.2) -- (-2.5,1.2) node [black,midway,xshift=0.00cm,yshift=0.6cm]
{\footnotesize Thm \ref{th:01}};
\draw [decorate,decoration={brace,amplitude=10pt}] (-2.5,1.2) -- (7.5,1.2) node [black,midway,xshift=0.00cm,yshift=0.6cm]
{\footnotesize \cite{MMR3} and \cite{NM25}};
\draw [decorate,decoration={brace,amplitude=10pt}] (7.5,-0.2) -- (-2.5,-0.2) node [black,midway,xshift=0.00cm,yshift=-0.7cm]
{\footnotesize Blow-up point at $\infty$};
\draw [decorate,decoration={brace,amplitude=10pt}] (-2.5,-0.2) -- (-5.2,-0.2) node [black,midway,xshift=0.00cm,yshift=-0.7cm]
{\footnotesize Finite blow-up point};
\draw [decorate,decoration={brace,amplitude=10pt}] (-7.8,1.2) -- (-5.2,1.2) node [black,midway,xshift=0.00cm,yshift=0.6cm]
{\footnotesize Ruled out \cite{KPV2}};
\end{tikzpicture}
\begin{center}
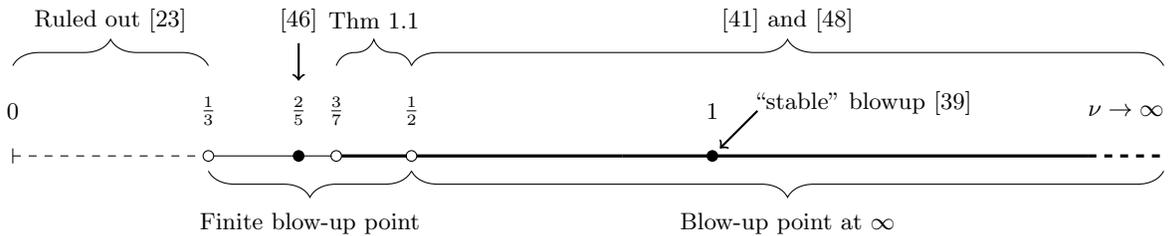
Figure 1. Blow-up rates $\nu$ in $H^1$ for finite time blowup of the critical gKdV equation
\end{center}

\medbreak

The results gathered in Figure 1 indicate that
all $\nu>\frac 12$ are possible (gathering results in \cite{MMR3} and \cite{NM25})
and all $\nu\in (\frac 37,\frac12)$ are possible (this is Theorem \ref{th:01} of the present article).
Moreover, the critical value $\nu=\frac 12$ is a work in progress announced in \cite{NM25}.
The isolated point $\nu=\frac25$ was obtained previously in \cite{MP24} and
may seem surprising.
We provide a formal explanation in Remark \ref{rk:al} below.

\subsection{Main result}

By convention, and for simplicity, we fix the blowup point at $x=0$, the blowup time at $t=0$ and 
we reverse the sense of time.

\begin{theorem}\label{th:01}
Let $\nu \in (\frac37,\frac12)$ and $\alpha=\frac{3\nu-1}{2-4\nu}$.
For any $\delta>0$, there exist $T_0>0$ and a solution $u\in \cC((0,T_0]:H^1(\RR))$   of \eqref{eq:KV} of the form
\begin{equation}\label{eq:th}
u(t,x)=
\frac1{t^{\frac\nu2}}Q\biggl(\frac{x-\sigma(t)}{t^\nu}\biggr) +r(t,x),
\end{equation}
where $\lim_{t\downarrow 0} \sigma(t)=0$ and  
\begin{equation}\label{eq:ra}
\sup_{t\in(0,T_0]} \|r(t)\|_{L^2}\lesssim \delta^\alpha,\quad
\lim_{t\downarrow 0} t^\nu \|\partial_x r(t)\|_{L^2} = 0.
\end{equation}
In particular, the solution $u(t)$ blows up at time $0$ with
\[
\lim_{t\downarrow 0} t^\nu \|\partial_x u(t)\|_{L^2} = \|Q'\|_{L^2}.
\]
Moreover,
\begin{equation}\label{eq:lr}
\lim_{t\downarrow 0} r(t)=r_\alpha \quad \mbox{in $L^2(\RR)$}
\end{equation}
where $r_\alpha$ belongs to $H^1(\RR)$ and satisfies
\begin{equation}\label{eq:dr}
r_\alpha(x)=\begin{cases} 
0\quad & \mbox{for $x<0$,}\\
c_\alpha x^{\alpha-\frac 12} \quad &\mbox{for $0<x<\delta$,}
\end{cases}
\end{equation}
where $c_\alpha=(4\alpha+3)^{-\frac14(2\alpha+1)}(2\alpha+1)\int Q$.
\end{theorem}

\begin{remark}\label{rk:al}
Let $r_\alpha$ be as in \eqref{eq:dr}, with compact support (to fix ideas). Observe that
$\alpha=\frac 12$ for $\nu = \frac 25$ (which is the case treated in \cite{MP24}),
while $\alpha\in (1,+\infty)$ is equivalent to $\nu \in (\frac 37,\frac 12)$.
Note also that the condition $\alpha>1$
corresponds exactly to the values of $\alpha$ for which the function $r_\alpha$ belongs to $H^1$
(we only consider the issue of the behavior at $x=0$).
As can be seen from the proof of Theorem~\ref{th:01}, this function is used in the present article as
a forcing term to produce the blowup rate $t^{-\nu}$.
Therefore, we observe that the construction of the blowup rate $t^{-\nu}$
by our method
is possible for the full range of parameters $\alpha$ for which the blowup residue $r_\alpha$
belongs to $H^1$.
\end{remark}

\subsection{Open problems}\label{S:1.3}
We present some open problems for the interested readers.
\begin{enumerate}
\item 
Does there exist a blowup residue $r_\alpha$ in $H^1$ for any solution of \eqref{eq:KV} 
blowing up in finite time at a finite point
(thus, $\nu\in (\frac13,\frac12)$)?
This question is open even for the solution constructed in \cite{MP24} for $\nu=\frac25$.
See related results in \cite{MR07} and later in \cite{MN24} for the mass critical nonlinear Schrödinger equation.
\item Do the solutions of Theorem \ref{th:01} enjoy any form of stability?
\item Are $\nu=\frac25$ and $\nu>\frac37$ the only possible power blowup rates for a solution in $H^1$?
\item Construct other kind of blowup rates, including for example oscillations between two admissible powers $\nu_1$
and $\nu_2$.
See related results in \cite{GNT} for a parabolic model and in \cite{DHKS,GNT} for critical wave models.
\item If the answer to question (3) is positive, are all the concentration rates $\nu\in (\frac 13,\frac 37)\setminus\{\frac25\}$  possible, but for infinite energy solutions belonging to 
$H^\kappa$ for some $\kappa=\kappa(\nu)\in (0,1)$?
We use the terminology concentration rate instead of blow-up rate since the solution
would not belong to $H^1$. 
See  \cite{LX25} for such an example for the critical cubic $2$D Zakharov-Kuznetsov equation 
(surprisingly, for a minimal mass blowup solution).
\end{enumerate}

\subsection{Sketch of the proof}
The exotic blowup rates described in Theorem \ref{th:01} are obtained using the interaction
between two components of different nature of the solution: first, the function $r_\alpha$, which satisfies 
$r_\alpha(x)=x^{\alpha-\frac12}$ for $x>0$ close to $0$, and which will eventually form the prescribed $H^1$ blowup residue
as $t\downarrow 0$;
second a concentrating blowup bubble based on the soliton $Q$, of the following general form
\[
\frac 1{\lambda^{\frac 12}(t)} Q \left(\frac{x-\sigma(t)}{\lambda(t)}\right),\quad
\mbox{where} \quad \mbox{$\lambda(t)\to 0$ and $\sigma(t)\to 0$  as $t \downarrow 0$.}
\]
These two components have different scales, 
since $r_\alpha(x)$ is a fixed function, while
the bubble is concentrating at a rate $\lambda(t)$ (going to $0$ faster than the self-similar rate $t^{1/3}$),
and thus they interact weakly.
In particular, close to $(t,x)=(0,0)$,
the background function $r_\alpha$ is not really affected by the presence of the singular bubble.
Moreover, the general shape of the bubble is only slightly perturbed. 
However, and this is the key feature to obtain exotic blowups,
the rate of concentration $\lambda(t)$ of the bubble is a quite sensitive parameter,
which, suitably perturbed by the nonlinear interactions between the soliton and
the function $r_\alpha$, can lead to exotic blowup rates and, as a consequence,
to blowup at a finite point. We do not know if this phenomenon has any form of stability.
Indeed, the solutions of Theorem \ref{th:01} are obtained by a compactness argument, passing to the limit in a sequence of solutions
constructed backwards in time. We refer to \cite{Me3,RaSz}
for pioneering works using this strategy in the blowup context, which leads
to solutions without any form of stability, unlike the solutions constructed in~\cite{MMR1,MMR3,MP24,NM25}.

\smallskip

Let us describe the arguments into more details, while also discussing the organization of the proof.

\smallskip

First, we discuss the issues related to the introduction of the function $r_\alpha(x)$.
Of course, it is not by itself a time-independent solution of the gKdV equation
and using it as an ansatz would lead to unmanageable error terms.
In Section \ref{S:2}, we rather introduce the solution $\Theta(t,x)$ of the gKdV equation 
(up to multiplication by a constant) with initial data $r_\alpha(x)$ at the expected blowup time $t=0$.
To compute the nonlinear interactions of the solution $\Theta$ with the bubble, it is important to describe its behavior with some precision and at the regularity level $\cC^4$, for small time and in some space region around
$x=\sigma(t)$ where the center of the soliton is located (see Proposition \ref{PR:qq}).
The proofs of such local estimates are reminiscent of the original proof of the Kato 
smoothing effect in \cite{Kato}.
However, the local description of the solution $\Theta$ in high Sobolev norms (up to $H^5$)
with an initial data of rather low regularity (in the case where $\alpha$ is close to $1$)
requires a sharp induction argument on localized energy quantities
defined with two weight functions of different spatial localizations. 
Proving such estimates is a key new technical ingredient of this article.
In this direction, we also refer to \cite{ILP} for general results of the propagation of regularity and decay for the generalized KdV equations.

The earlier construction \cite{MP24}  corresponds to the special case $\alpha=\frac12$ and 
thus $\Theta \equiv 1$ around the blowup point.
In particular, this step was not required.
In \cite{MMR3,MP20}, the analogue estimates are for large $x$, which is related to more
standard estimates in weighted norms for gKdV.

\smallskip

Second, we briefly describe the content of Section \ref{S:3} concerning the approximate blowup profile.
We have said before that the shape of the bubble is only weakly perturbed by the interaction with
$\Theta$; after rescaling, its main part will be the function $Q$. However,
taking into account small interactions is necessary to obtain
an approximate solution at a sufficient order and to produce the desired blowup regime, and
for this, we need a refined approximate profile which has a different form compared to \cite{MP24}.
To define and justify heuristically the blowup profile, we introduce rescaled variables $(s,y)$, 
$s<0$ and $y\in \RR$, setting
\begin{equation} \label{u:w}
U(t,x) = \frac 1{\lambda^{\frac 12}(s)} w(s,y) ,\quad \frac {ds}{dt}= \frac 1{\lambda^3},\quad y = \frac {x-\sigma(s)}{\lambda(s)}.
\end{equation}
Note that in view of the expected estimates on $\lambda$ and $\sigma$, as $t\in (0,t_0]$, we have $s\in (-\infty,s_0]$.
The exact time-dependent parameters $\lambda>0$, $\sigma\in \RR$ are to be determined and the function $w$ satisfies the rescaled equation
\begin{equation} \label{eq:re}
\partial_s w + \partial_y ( \partial_y^2 w - w + w^5) 
- \frac{\lambda_s}{\lambda}\left( \frac w 2 + y \partial_y w \right)
- \left( \frac{\sigma_s}{\lambda} - 1 \right) \partial_y w =0.
\end{equation}
An approximate solution $W$ of \eqref{eq:re} is defined of the form 
\[
W=V_0+V,\quad 
V=V_1\theta+V_2\theta^2+V_2^*\partial_y\theta+V_3\theta^3,
\]
where $V_0$ is a simple (time-dependent) cut-off of the function $Q$ and where the function 
$\theta(s,y)= \lambda^\frac12(s)\Theta(t(s),\lambda(s)y+\sigma(s))$ is the rescaled
version of $\Theta(t,x)$.
In Proposition \ref{pr:bp}, we define the functions $V_1$, $V_2$, $V_2^*$ and $V_3$
(these functions depend mainly on $y$, with a time-dependent cut-off)
so that $W$ is an approximate solution of \eqref{eq:re} at a sufficient order.
The function $V_1$ is a time-dependent cut-off of a function $A_1=A_1(y)$ defined as the unique
odd function such that $(-A_1''+A_1-5Q^4A_1)' = \frac Q2+y Q'$. 
As observed in \cite{MP24}, the introduction of such a function $A_1$ is related to the existence of a resonance for the linearized operator around the soliton.
The oddness of $A_1$ is essential to obtain a crucial cancellation of the crossed terms involving the even function $Q$ and $A_1$ when computing the energy of the approximate solution.
Ensuring the existence of suitable functions $V_2$, $V_2^*$ and $V_3$ 
requires additional scaling terms which in turn modify the scaling
equation. 
Indeed, it turns out that at the main order, the following equation is verified
\[
\frac{\lambda_s}{\lambda} \approx (2\alpha+1) \lambda^\frac12\sigma^{\alpha-\frac12} .
\]
Of course, the term on the right-hand side is related to 
the value of $\theta(t,x)\approx \lambda^\frac12 r_\alpha(x)$ at the point $x=\sigma$
where the bubble is located.
This is the mechanism by which the nonlinear interactions between $\Theta$ and the bubble 
induce the desired blowup rate.

In Subsection \ref{S:310}, we justify formally that the equation above and $\frac{\sigma_s}{\lambda} \approx 1$
imply the  blowup regime described in Theorem \ref{th:01}. The full proof in Section \ref{S:5} requires bootstrap estimates, in particular
on the difference between the exact  solution and the approximate solution.

\smallskip

In Section \ref{S:4}, we introduce a further refinement of the approximate solution, 
of the form $W_b= W+bP_b$, where $b$ is a small parameter needed to adjust the blowup profile
in a certain direction. The function $P_b$ is a time-dependent cut-off of a given function $P(y)\not\in L^2$.
Such a refinement $bP_b$ have been introduced in \cite{MMR1} and
used in the articles on the blowup phenomenon for \eqref{eq:KV} since then.
We also set 
\[
w(s,y) = W_{b}(s,y) + \varepsilon(s,y), 
\]
where $w$ is the solution of \eqref{eq:re} and where the small function $\varepsilon$
satifies three orthogonality conditions related to the tuning of the three ``free'' parameters
$(\lambda,\sigma,b)$
(the main order of the parameters $\lambda$ and $\sigma$ is prescribed by the interactions,
but they can be adjusted at a higher order to ensure the orthonality conditions).

Then, Section \ref{S:5} is devoted to the introduction of precise bootstrap estimates on all the components
of the solution. Moreover, using the equations of $(\lambda,\sigma,b)$, we also close
all the estimates, except the ones on the function $\varepsilon$.

In Section \ref{S:6}, we use a variant of an energy functional on $\varepsilon$ that was introduced in \cite{MMR1} (and used since then
in the articles related to blowup of \eqref{eq:KV})
to control $\varepsilon$ in the soliton region. 
Note that in this part of the proof, it seems essential to estimate the error term $\varepsilon$ backwards in time, since in the defocusing framework, a delicate scaling term has a favorable sign.
Our approach thus combines the idea introduced in \cite{RaSz} to construct blowup solutions backwards in time using favorable energy estimates with the local virial-energy estimates specific
to the critical gKdV equation introduced in \cite{MMR1}. 
As said before, no form of stability is to be expected using such approach.

Finally, in Section \ref{S:7}, we pass to the limit in the sequence of solutions and we obtain the solution
described in Theorem \ref{th:01}. We also justify the convergence of the function $r(t,x)$ to the
blowup residue $r_\alpha(x)$, by proving that the global $L^2$ norm of $\varepsilon$ 
converges to $0$ at the blowup time.
Such a strong convergence in $L^2$ could not be obtained in \cite{MP24}
since trying to reproduce the argument would have required a discontinuous blowup residue
$r_{1/2}(x)=1$ for $x>0$ small, and $r_{1/2}(x)=0$ for $x<0$.
By our method, residue that are not in $H^1$ lead to infinite energy solutions.

\subsection{Related results for other critical equations}

We cite here some references on related results for the critical gKdV equation.
We start with the formal and numerical blowup study \cite{Amodio}.
Then, for general recent results on the long time behavior of solutions of \eqref{eq:KV},
we refer to \cite {Do,DG22,KKSV,Pi16}.
On the slightly supercritical problem, \emph{i.e.}, the gKdV equation with power nonlinearity
with $p=5^+$, see \cite{Koch,La16}. 
For other related questions, like the continuation after the blowup time
or the case of saturated nonlinearity, we refer to \cite{Lanstar,Lan}.
Finally, in the spirit of \cite{MMR3}, flattening solutions were constructed in \cite{MP20}, 
\emph{i.e.} global solutions close to a soliton blowing down at $t\to+\infty$,
for any  admissible rate.

For the mass critical nonlinear Schrödinger equation, the co-existence of at least two different blowup rates, namely the 
log-log blowup rate (see \cite{SulemSulem}, \cite{P} and \cite{MeRa04,MeRa05,MeRa05bis,MeRa06,MR07,Ra}) and
the conformal rate (see \cite{BW,KS2,MRS,RaSz}),
leads to interesting problems concerning all the possible blowup rates and blowup residues.
We also refer to \cite{MaRa} at a rate slightly above the conformal rate obtained
by collision of several blowup bubbles at the blowup point.
See also the general references \cite{Cabook,SulemSulem,We86}.

The question of blowup for the mass critical (cubic) generalized Benjamin-Ono (gBO) is widely open.
So far, only the minimal mass solution was constructed in \cite{MP17}.
See \cite{KMR} for a first blowup study for another
nonlocal equation similar to the (gBO) equation, but close to gKdV.
We also refer to \cite{RWY} for numerical experiments for (gBO).

The blowup problem for the mass critical 2-dimensional (cubic) Zakharov-Kutnetzov equation (ZK)
is closely related to the one for gKdV. Several advances on this problem have been obtained
recently \cite{BGMY,GLY1,GLY2,FHRY,LX25}.

The notion of exotic blowup, and in particular the construction of 
families of solutions with a continuum of blowup rates is also a quite interesting
question for energy critical problems, mainly the heat flow, the wave maps and the nonlinear wave equation.
Without trying to give an exhaustive list, we
refer the reader to pioneering works \cite{GNT,KS14,KST1,KST2} on this question and 
to \cite{BK,DHKS,HR,JJ,JL1,JLR,KN,KS2,KS,Ro}.

We point out that the results in the articles \cite{JJ,JLR,MN24}
are close in spirit to the present article,
prescribing a blowup residue related to an expected blowup rate.

\subsection{Notation}
The linearized operator $\cL$, whose main properties are recalled in Lemma \ref{le:li} below,
is defined by
\begin{equation}\label{eq:cL} 
\cL = -\partial_x^2 + 1 - 5 Q^4.
\end{equation}
We recall the definitions of the parameters $\nu$ and $\alpha$ introduced in the statement of 
Theorem~\ref{th:01} and the definition of another useful parameter $\beta$
\begin{equation}\label{eq:ab}
\alpha>1,\quad \nu=\frac{2\alpha+1}{4\alpha+3}\in\left(\frac37,\frac12\right),\quad \beta=\frac1{4\alpha+3}\in \left(0,\frac17\right),\quad
1-\beta=2\nu.
\end{equation}
For $j\in \NN$, let
\begin{equation} \label{eq:La}
\Lambda_jf=\frac{1-j}2f+yf',\quad \Lambda = \Lambda_0.
\end{equation}
Let
\begin{equation}\label{eq:m0}
m_0 = \frac 14 \int Q.
\end{equation}
The following weight function $\omega$ will be used throughout this article
\[ 
\omega(y) = e^{-\frac{|y|}{10}},
\]
as well as the weighted $L^2$-norm 
\begin{equation} \label{def:L2sol} 
\|f\|_{L^2_\loc}=\left(\int_{\mathbb R} f(y)^2 e^{-\frac{|y|}{10}}dy \right)^{\frac12} .
\end{equation}
Now, we introduce several standard smooth cut-off functions.
Let $\chi:\mathbb R \to [0,1]$ be a smooth function such that
\begin{equation} \label{eq:C}
\chi \equiv 0 \ \text{on} \ (-\infty,1], \ \chi \equiv 1 \ \text{on} \ [2,+\infty) \ \text{and} \ \chi'\ge 0 \ \text{on} \ \mathbb R .
\end{equation}
Let  $\CL:\RR\to[0,1]$ be a smooth function such that 
\begin{equation}\label{eq:CL}
\mbox{
$\CL\equiv 0$ on $(-\infty,-1]$,
$\CL\equiv 1$ on $[-\frac12,+\infty)$
and $\CL' \geq 0$ on $\RR$.}
\end{equation}
Let $\CR:\RR\to[0,1]$ be a smooth function such that
\begin{equation}\label{eq:CR}
\mbox{$\CR\equiv 1$ on $(-\infty,\frac12]$, 
$\CR\equiv 0 $ on $[1,+\infty)$
and $\CR' \leq 0$ on $\RR$.}
\end{equation}
Let $\cY$ be the set of functions $\phi:\RR\to\RR$ of class $\cC^{\infty}(\RR)$ such that
\begin{equation*}
\mbox{$\forall \, m \in \NN$, $\exists C_m > 0$,
$r_m \ge 0$ such that
$|\phi^{(m)}(y)| \le C_m(1+|y|)^{r_k}e^{-|y|},$
$\forall \, y \in \RR$.}
\end{equation*}
Fix an odd function $z_0:\RR\to \RR$ of class $\cC^\infty$ such that 
$z_0(y)=1$ for $y>1$, and define
\[
\cZ_0=\cY+\SPAN(z_0).
\]
For any $m= 0,1$, fix a function $z_m^-:\RR\to[0,\infty)$ of class $\cC^\infty$ such that
$z_m^-(y)=y^m$ for $m \leq -1$ and 
$z_m^-(y)=0$ for $y>0$ and define 
\begin{equation}\label{eq:Zk}
\cZ_m^- = \cY + \SPAN(z_0^-,\ldots,z_m^-).
\end{equation}
We observe that
\begin{equation}\label{eq:zz}
\Lambda_1\cZ_0\subset\cY, \quad
\Lambda_1\cZ_0^-\subset\cY,\quad
\Lambda_{3}\cZ_1^-\subset \cZ_0^-.
\end{equation}

\section{The blowup residue}\label{S:2}
In this section, we prepare a suitable ``background function'' $\Theta(t,x)$, 
whose aim is to force the exotic blow-up regime at the point $(t,x)=(0,0)$.
For $x>0$ close to $0$, this function is simply $x^{\alpha-\frac12}$ at the main order.
However, since we require the function $\Theta(t,x)$ to be a solution of the gKdV equation, 
its exact behavior is not easily described, even close to $(t,x)=(0,0)$.

With the notation defined in \eqref{eq:ab}, set
\begin{equation}\label{eq:mb}
\mu(t) = t^\nu,\quad
\rho(t)=t^\beta.
\end{equation}
Let $\delta>0$ small.
Define
\begin{equation}\label{eq:T0}
\Theta_0(x) =
\begin{cases}
x^{\alpha-\frac 12} \CR\big(\delta^{-1}x\big) & \mbox{ for $x>0$,} \\ 
0 & \mbox{ for $x<0$}.
\end{cases}
\end{equation}
Since $\alpha>1$, one has $\|\Theta_0\|_{H^1}\lesssim \delta^{\alpha-1}$
and so for $\delta>0$ sufficiently small, it follows from \cite{KPV} that the solution $\Theta$ of
\begin{equation}\label{eq:TT}
\begin{cases}
\partial_t \Theta+\partial_x (\partial_x^2 \Theta + m_0^4 \Theta^5)=0\\
\Theta(0)=\Theta_0
\end{cases}
\end{equation}
is global and satisfies
\begin{equation} \label{eq:T9}
\sup_{t\in\RR}\|\Theta(t)\|_{H^1} \lesssim \delta^{\alpha-1}\lesssim 1 .
\end{equation}
The need to have $\Theta_0\in H^1(\RR)$ justifies the restriction $\alpha>1$ and so 
the interval $\nu\in(\frac37,\frac12)$ for the possible blowup rates constructed in the present article.
See Remark \ref{rk:al}.

Let
\begin{equation} \label{eq:Td}
\Theta_\sharp(t,x)=\Theta_0(x)+t\chi( t^{-\frac76\beta} x) \Theta_1(x) \quad \text{where} \quad
\Theta_1 =-\Theta_0''' -m_0^4(\Theta_0^5)' 
\end{equation}
and
\begin{equation}\label{def:q}
q(t,x) = \Theta(t,x) - \Theta_\sharp(t,x).
\end{equation}
The cut-off function $\chi$ is involved in the definition of $\Theta_\sharp$ to avoid a possible singularity
at $x=0$ due to $\Theta_1$ (depending on the value of $\alpha$).
We prove estimates on the $L^{\infty}$-norm of $q(t)$  and its spatial derivatives up to order $4$, 
for $t\in [0,1]$ and on spatial regions on the right of the curve $x=\frac 12\rho(t)$
(where the blowing up soliton will be located in the framework of Theorem~\ref{th:01}).
Such estimates mean that $\Theta_\sharp$ is an approximate solution for
the equation of $\Theta$ at some order.
Note that, at the cost of a more complicated expression, it would possible to improve $\Theta_\sharp$ as an approximate solution to
\eqref{eq:KV}, which would allow us to reduce further the size of the error terms obtained in the next proposition.

\begin{proposition}\label{PR:qq}
For all $t\in \RR$, the function $q$ satisfies $\|q(t)\|_{H^1} \lesssim \delta^{\alpha-1}$. 
Moreover, there exists $t_0\in(0,1]$ such that for all $k\in\{0,1,2,3,4\}$, $t\in (0,t_0]$,
\begin{equation}\label{eq:dq}
\|\partial_x^k q(t,x)\|_{L^\infty(x>\frac 12 \rho(t))} \lesssim 
\begin{cases}
t^{2+\frac{2\alpha-2k-13}{2(4\alpha+3)}} & 1<\alpha<k+6,\\
t^{2-\frac{1}{2(4\alpha+3)}} |\log t|^\frac14 & \alpha=k+6,\\
t^{2+\frac{\alpha - k-7 }{2(4\alpha+3)}}  & k+6<\alpha<k+7,\\
t^2 |\log t|^\frac14 & \alpha=k+7,\\
t^2 & \alpha>k+7.
\end{cases}
 \end{equation}
\end{proposition}
The rest of this section is devoted to the proof of Proposition \ref{PR:qq}.

\subsection{A differential inequality}
\begin{lemma}\label{LE:ax}
Let ${\kpu}>0$, $\kpd>0$ and $C \geq 0$.
Let $f$ be a nonnegative continuous function on $(0,1]$
such that $\int_0^1 f <+\infty$.
Assume that $G:[0,1]\to [0,+\infty)$, continuous on $[0,1]$ and 
of class $\mathcal C^1$ on $(0,1]$, satisfies
 the differential inequality on $(0,1]$
\begin{equation}\label{eq:DG}
G'(t)\leq 
 C t^{\kpu-1} G(t)+f(t)  \sqrt{G(t)}+C t^{\kpd-1} ,\quad G(0)=0.
\end{equation}
Then, for all $t\in [0,1]$,
\[
G(t) \lesssim   \left( \int_0^t f\right)^2 +t^{\kpd}.
\]
\end{lemma}
\begin{proof}[Proof of Lemma \ref{LE:ax}]
Let $k(t) = \exp(- \frac C{\kpu} t^{\kpu})$ so that
for all $t\in (0,1]$,
$k'(t) = - C t^{\kpu-1} k(t)$
and $\exp(-\frac C{\kpu})\leq k(t)\leq 1$.
We set
\[
H(t) = \sqrt{k(t)G(t)+ t^{\kpd}}.
\]
Then, for $t\in (0,1]$,
\[
H'=\frac{kG'+k'G+\kpd t^{\kpd-1}}{2\sqrt{kG+t^{\kpd}}}
\lesssim \frac{f\sqrt{G}+t^{\kpd-1}}{\sqrt{kG+t^{\kpd}}}
 \lesssim   f+  t^{\frac 12\kpd-1}.
\]
Since  $G(0)=0$, we have $H(0)=0$
and integrating the estimate above on $(0,t)$,
we obtain, for all $t\in [0,1]$,
$H(t) \lesssim \int_0^t f + t^{\frac 12\kpd}$.
The result then follows from  $G(t) \lesssim H^2(t)$.
\end{proof}

\subsection{Preliminary estimates}
We start the proof of Proposition \ref{PR:qq} by giving
some estimates on the approximate solution $\Theta_\sharp$ and the function $q$.
\begin{lemma}\label{le:eq}
\begin{enumerate}
\item[(i)] The functions $\Theta_0$, $\Theta_1$ and $\Theta_\sharp(t)$ are supported on the interval $[0,\delta]$.
\item[(ii)] \emph{Estimates on $[0,\delta]$.} For all $j\in \{0,\ldots,6\}$, $x\in[0,\delta]$, $t\in [0,1]$ satisfying $x\ge 2 t^{\frac76\beta}$,
\begin{equation}\label{eq:eT}
|\partial_x^j \Theta_0|+|\partial_x^j \Theta_\sharp|\lesssim x^{\alpha-\frac12-j},\quad
 |\partial_x^j \partial_t\Theta_\sharp|+|\partial_x^j \Theta_1|\lesssim x^{\alpha-\frac 72-j}.
\end{equation}
\item[(iii)] \emph{Equation of $q$.}
It holds $\|q(t)\|_{H^1} \lesssim \delta^{\alpha-1}$.
Moreover, for $(t,x)\in(0,1]\times\RR$,
\begin{equation}\label{eq:qq}
\partial_t q+\partial_x (\partial_x^2 q + m_0^4 ((q+\Theta_\sharp)^5-
\Theta_\sharp^5))=F_\sharp+R_\sharp
\end{equation}
where  
\begin{equation}\label{eq:Fd}
F_\sharp = 
-t\chi(t^{-\frac76\beta}x) \Theta_1''' - m_0^4\partial_x (\Theta_\sharp^5-\Theta_0^5)
\end{equation}
and where $R_\sharp(t,x)=0$ for all $(t,x)\in[0,1]\times\RR$ such that $x\geq 2t^{\frac76\beta}$.
\end{enumerate}
\end{lemma}
\begin{proof}[Proof of Lemma \ref{le:eq}]
Firstly, (i) is clear from the definitions.

We now concentrate on the interval $[0,\delta]$ for the space variable $x$.
For $j\geq0$ and $0<x<\delta$, by the definition of $\Theta_0$, we see that
\begin{equation*}
|\partial_x^j \Theta_0| \lesssim 
\sum_{l=0}^jx^{\alpha-\frac12-l}\delta^{-j+l}
|\CR^{(j-l)}(\delta^{-1} x)|
\lesssim x^{\alpha-\frac12-j}.
\end{equation*}
Thus, by the definition of $\Theta_1$, for $j\geq 0$ and $0<x<\delta$,
$|\partial_x^j \Theta_1| 
\lesssim x^{\alpha-\frac72-j}$.
In particular, for $0<x<\delta$ and $t\in [0,1]$,
\[
\Theta_\sharp^2\lesssim x^{2\alpha-1}+t^2 \chi^2( t^{-\frac76\beta} x) x^{2\alpha-7}
\lesssim x^{2\alpha-1}+t^{2-7\beta} x^{2\alpha-1}
\lesssim x^{2\alpha-1},
\]
using $0<\beta\leq \frac 17$.
Similarly, for $0<x<\delta$ and $t\in [0,1]$,
\[
(\partial_x\Theta_\sharp)^2\lesssim x^{2\alpha-3}+t^2 \chi^2( t^{-\frac76\beta} x) x^{2\alpha-9}
+ t^{2-\frac73\beta} (\chi')^2(t^{-\frac76\beta} x) x^{2\alpha-7} 
\lesssim x^{2\alpha-3}.
\]
In particular, we obtain 
\[
\|q(t)\|_{H^1} \lesssim \|\Theta(t)\|_{H^1} + \|\Theta_\sharp(t)\|_{H^1}
 \lesssim \delta^{\alpha-1}.
\]
More generally, for all $j\geq 0$, $0 <x<\delta$ and $t\in [0,1]$,
\begin{align*}
|\partial_x^j \Theta_\sharp| 
& \lesssim |\partial_x^j \Theta_0|
+\sum_{l=0}^jt^{1-\frac76\beta l} |\chi^{(l)}(t^{-\frac76\beta}x)| |\partial_x^{j-l} \Theta_1|\\
& \lesssim x^{\alpha-\frac12-j} + x^{\alpha-\frac12-j}
\sum_{l=0}^j t^{1-\frac76\beta l} |\chi^{(l)}(t^{-\frac76}x)| x^{l-3}.
\end{align*}
For $0\leq l\leq 3$ and $t^{\frac76\beta}\leq x$, one has
$t^{1-\frac 76\beta l} x^{l-3}\lesssim t^{1-\frac72 \beta}\lesssim 1$.
For $4\leq l\leq 6$ and $t^{\frac76\beta}\leq x\leq \delta$, one has
$t^{1-\frac 76\beta l} x^{l-3}\lesssim t^{1-\frac76 \beta l}\lesssim 1$
since $\beta\leq \frac 17$ and $l\leq 6$, so that
$\frac 76 \beta l\leq 1$.
Thus, for all $0\leq j\leq 6$, $0< x< \delta$ and $t\in [0,1]$, 
one has $|\partial_x^j \Theta_\sharp|\lesssim x^{\alpha-\frac12-j}$.
The estimate for $\partial_x^j\partial_t\Theta_\sharp$ follows similarly observing that for $x\ge 2 t^{\frac76\beta}$, $\partial_t\Theta_\sharp=\Theta_1$.

Finally, we derive the equation of $q$ from the equation of $\Theta$ and the definition of $\Theta_\sharp$
\begin{align*}
0 & = \partial_t \Theta + \partial_x (\partial_x^2 \Theta
+m_0^4\Theta^5)\\
& = \partial_t q + \partial_x^3 q
+m_0^4 \partial_x((q+\Theta_\sharp)^5-\Theta_\sharp^5)
+\partial_t \Theta_\sharp +\partial_x^3\Theta_\sharp
+m_0^4\partial_x(\Theta_\sharp^5)
\\
& = \partial_t q + \partial_x^3 q
+m_0^4 \partial_x((q+\Theta_\sharp)^5-\Theta_\sharp^5)
+ \chi(t^{-\frac76\beta}x)\Theta_1
-\frac 76 \beta t^{-\frac 76\beta}x \chi'(t^{-\frac 76\beta}x)\Theta_1\\
&\quad+\Theta_0'''+t\chi(t^{-\frac76\beta}x) \Theta_1'''
+3t^{1-\frac 76\beta}\chi'(t^{-\frac 76\beta}x)\Theta_1''
+3t^{1-\frac 73\beta}\chi''(t^{-\frac 76\beta}x)\Theta_1'
+t^{1-\frac 72\beta}\chi'''(t^{-\frac 76\beta}x)\Theta_1\\
&\quad 
+m_0^4 (\Theta_0^5)'
+m_0^4\partial_x (\Theta_\sharp^5-\Theta_0^5)
\end{align*}
and so, using the definition of $\Theta_1$, the function $q$ satisfies equation \eqref{eq:qq} where $F_\sharp$ is defined in~\eqref{eq:Fd} and where the error term $R_\sharp$ has the following expression
\begin{align*}
R_\sharp
& = 
-(\Theta_0'''+m_0^4(\Theta_0^5)') (1-\chi(t^{-\frac76\beta}x))
+\frac76 \beta t^{-\frac76\beta}x \chi'(t^{-\frac76\beta}x)\Theta_1\\
&\quad -3t^{1-\frac76\beta}\chi'(t^{-\frac76\beta}x)\Theta_1''
-3t^{1-\frac73\beta}\chi''(t^{-\frac76\beta}x)\Theta_1'
-t^{1-\frac72\beta}\chi'''(t^{-\frac76\beta}x)\Theta_1.
\end{align*}
As stated in the lemma, we check from the definition of $\chi$ that $R_\sharp=0$ for $x>2t^{\frac76\beta}$.
\end{proof}

\subsection{Estimates for the first derivatives}
Our next goal is to prove estimates on $q$
and its first two derivatives.
We define some auxiliary functions.
Let $A>1$ to be chosen sufficiently large 
and 
\begin{equation}\label{eq:dp}
\varphi(x) = \int_{-\infty}^x \sech\left(\frac{y}A\right) dy
=2A\arctan\left(\exp\left(\frac{x}A\right)\right).
\end{equation}
For $k=0,1,2$, set
\begin{align}
& \chi_k(t,x)=\chi^{(11)^k}(z_0(t,x)),\quad z_0(t,x) = \frac{x-2^{-8}\rho(t)}{\mu(t)}, \label{def:z0}\\
& \varphi_k(t,x)=\varphi^{(11)^k}(y_0(t,x)),\quad y_0(t,x) = \frac{x-2^{-7}\rho(t)}{\mu(t)}, \label{def:y0}
\end{align}
and
\[
\omega_k  = \chi_k \varphi_k  .
\]

\begin{lemma}\label{LE:qq}
For $A$ sufficiently large, for $k=0,1,2$ and for all $t>0$ sufficiently small,
\begin{equation}\label{eq:L1}
\int (\partial_x^k q)^2 \omega_k 
 \lesssim  \begin{cases} 
t^{4+\frac {2( \alpha-6-k)}{4\alpha+3}} &  \mbox{ if $1<\alpha<k+6$,}\\
 t^4 |\log t| &  \mbox{ if $\alpha=k+6$,}\\
t^4 & \mbox{ if $\alpha>k+6$.}
\end{cases}
\end{equation}
\end{lemma}
\begin{proof}[Proof of Lemma \ref{LE:qq}]
\emph{Step 1.} Preliminary estimates.
Observe that, for all $x \in \mathbb R$,
\begin{align*}
\varphi'(x)&=\sech\left(\frac {x}A\right) ,\quad
 \varphi'(x)\leq \frac1 A {\varphi(x)} , \\
\left|\varphi''(x)\right|&=\frac1A \left| \sech(\frac {x}A) \tanh (\frac {x}A)\right| \leq \frac1A\varphi'(x) ,\\
|\varphi'''(x)|&=\frac{1}{A^2}\left|\sech\left(\frac {x}A\right)
-2\sech^3\left(\frac {x}A\right)\right|
 \leq\frac{1}{A^2}\varphi'(x) .
\end{align*}
Then, we claim that for $A$ large enough, one has
for $k=0,1,2$,
\begin{equation}\label{eq:px}
|(\varphi^{(11)^k})'''| \le 2^{-9}{\beta}(\varphi^{(11)^k})' .
\end{equation}
Indeed, we compute
\begin{align*}
(\varphi^{(11)^k})'&=(11)^k \varphi^{(11)^{k}-1} \varphi' \\
(\varphi^{(11)^k})''& = 
(11)^k((11)^{k}-1)\varphi^{(11)^{k}-2}(\varphi')^2
+(11)^k \varphi^{(11)^{k}-1} \varphi''\\
(\varphi^{(11)^k})'''&=(11)^k((11)^{k}-1)((11)^k-2)\varphi^{(11)^{k}-3}(\varphi')^3
\\ &\quad +3 (11)^k((11)^{k}-1)\varphi^{(11)^k-2}\varphi' \varphi''+(11)^k \varphi^{(11)^{k}-1} \varphi''' .
\end{align*}
It follows that
\begin{equation*}
|(\varphi^{(11)^k})'''| 
\lesssim  \left(\varphi^{-1}\varphi'\right)^2|(\varphi^{(11)^k})'|+\left|\varphi^{-1}\varphi''\right| |(\varphi^{(11)^k})'|+ |\varphi^{(11)^{k}-1} \varphi'''|
\lesssim \frac1{A^2}(\varphi^{(11)^k})',
\end{equation*}
which implies \eqref{eq:ph} for any $0\leq k\leq 2$ by choosing $A$ sufficiently large.
Now, $A$ is fixed (the same $A$ will be used in the next Lemma). In particular, for $k=0,1,2$,
\begin{equation}\label{eq:ph}
| \partial_x^3 \varphi_k | \leq 2^{-9} \beta \mu^{-2} \partial_x \varphi_k.
\end{equation}

Observe that the functions $\chi_k$, $\varphi_k$ and $\omega_k$ have been constructed such that, for $k=0,1$,
\begin{equation}\label{eq:ww}
\chi_{k+1} = \chi_k^{11},\quad
\varphi_{k+1} = \varphi_k^{11},\quad
\omega_{k+1} = \omega_k^{11}.
\end{equation}
We also set, for $k= 0,1,2$,
\begin{equation} \label{eq:vc}
s_k 
=\partial_x \omega_k 
=  \chi_k \partial_x\varphi_k 
+\partial_x \chi_k \varphi_k ,
\end{equation}
where both terms on the right-hand side are nonnegative.
In particular, we have for $k=0,1$,
\begin{equation} \label{eq:sk}
s_{k+1} = \omega_k^{10} s_k =  \omega_{k+1}^{\frac12}\omega_{k}^{\frac12}\omega_k^4 s_k.
\end{equation}
We give for future use some simple identities involving the functions $\omega_k$
\begin{align}
\partial_x \omega_k  & =  s_k ,\label{eq:o1}\\
\partial_x^2 \omega_k  & =  \chi_k\partial_x^2\varphi_k  + 2\partial_x\chi_k \partial_x\varphi_k +  \partial_x^2\chi_k \varphi_k , \label{eq:o2}\\
\partial_x^3 \omega_k  & =  \chi_k\partial_x^3\varphi_k + 3\partial_x\chi_k \partial_x^2\varphi_k + 3 \partial_x^2\chi_k \partial_x\varphi_k 
+ \partial_x^3\chi_k\varphi_k ,\label{eq:o3}\\
\partial_t \omega_k  &  = -2^{-7}\rho_t s_k- \mu_t y_0 s_k 
+2^{-8}\Bigl(\rho_t-\rho\frac{\mu_t}{\mu} \Bigr)\partial_x \chi_k\varphi_k.\label{eq:ot}
\end{align}
Moreover, we have 
\begin{align*}
\partial_x(\omega_{k+1}^{\frac12}\omega_k^{\frac12}) = \frac12 \left(\frac{\omega_k}{\omega_{k+1}} \right)^{\frac12}\partial_x\omega_{k+1}+\frac12 \left(\frac{\omega_{k+1}}{\omega_{k}} \right)^{\frac12}\partial_x\omega_{k} ,
\end{align*}
so it follows from \eqref{eq:ww} and \eqref{eq:o1} that for $k=0,1$,
\begin{equation} \label{eq:oo}
\partial_x(\omega_{k+1}^{\frac12}\omega_k^{\frac12})=6 \omega_k^5 s_k.
\end{equation}
As a consequence, we claim that for $k=0,1$,
\begin{equation}\label{eq:Lz}
\bigr\|(\partial_x^k q)^2 \omega_{k+1}^\frac12\omega_k^\frac12\bigr\|_{L^\infty}\lesssim
\left(\int (\partial_x^{k+1} q)^2 \omega_{k+1} \right)^{\frac12}
\left(\int (\partial_x^{k} q)^2 \omega_k \right)^{\frac12} 
+\int (\partial_x^kq)^2  \omega_k^5s_k.
\end{equation}
Proof of \eqref{eq:Lz}.  We have
\begin{align*}
(\partial_x^k q)^2 \omega_{k+1}^\frac12\omega_k^\frac12
=- \int_x^{+\infty} &\Bigl( 2 (\partial_x^{k+1} q)(\partial_x^k q)\omega_{k+1}^\frac12 \omega_k^\frac12
+(\partial_x^k q)^2 \partial_x (\omega_{k+1}^\frac12\omega_k^\frac12)\Bigr).
\end{align*}
By the Cauchy-Schwarz inequality,
\[
\left| \int_x^{+\infty}  (\partial_x^{k+1} q)(\partial_x^k q) \omega_{k+1}^\frac12\omega_k^\frac12\right|
\lesssim 
\left(\int (\partial_x^{k+1} q)^2 \omega_{k+1} \right)^{\frac12}
\left(\int (\partial_x^{k} q)^2 \omega_k \right)^{\frac12}.
\]
Moreover, from \eqref{eq:oo}, it holds
\[
\int(\partial_x^k q)^2 \partial_x (\omega_{k+1}^\frac12\omega_k^\frac12)
= 6\int (\partial_x^kq)^2 \omega_k^5 s_k ,
\]
which suffices to prove \eqref{eq:Lz}.

Now, we claim that 
for any $j\geq 1$, $l\geq 0$ and $k=0,1,2$,
\begin{equation}\label{eq:cp}
|\partial_x^j\chi_k \partial_x^l\varphi_k|
\lesssim \mu^{-k-l}\exp\left(-2^{-8}A^{-1} t^{-\frac {2\alpha}{4\alpha+3}}\right).
\end{equation}
Proof of \eqref{eq:cp}.
Note that by definition of $\chi$, for $j\geq 1$, $\chi^{(j)}=0$ on $(-\infty,1]\cup[2,+\infty)$.
Thus for any $(t,x)$,
\[
\chi_k^{(j)}(t,x)\neq 0 \mbox{ implies }
z_0(t,x)\leq2 \mbox{ and so } 
y_0(t,x) = z_0(t,x)- 2^{-8} \frac\rho\mu\leq2- 2^{-8} t^{-\frac {2\alpha}{4\alpha+3}}.
\]
Thus, \eqref{eq:cp} follows from the definition of $\varphi_k$.
In what follows, for $t\in [0,1]$, using the estimate \eqref{eq:cp},
we will estimate terms containing $|\chi_k^{(j)}\varphi_k^{(l)}|$ for $j\geq 1$ by large powers of $t$, typically $t^{100}$, where $100$ has no special meaning. We also observe that for $k=2$, for any $1\leq j \leq 3$ and $0\leq l\leq 3$,
\begin{equation}\label{eq:cq}
|\partial_x^j\chi_2 \partial_x^l\varphi_2|\lesssim t^{100}\chi_1 \partial_x \varphi_1 \lesssim t^{100} s_1.
\end{equation}

Finally, we claim that for any $p\geq 1$
and $i=(i_1,\ldots,i_p)\in\NN^p$, such that $i_\nu \leq 6$,
for all $x\in\RR$ and $t\in [0,1]$
(using the notation $|i|=\sum_{\nu} i_\nu$)
\begin{equation}\label{eq:iT}
\left|\Pi_{\nu=1}^p(\partial_x^{i_\nu}\Theta_\sharp)\right|\omega_k
\lesssim \begin{cases}
1 & \mbox{if $|i|\leq p(\alpha-\frac12)$,} \\
t^{\beta(p(\alpha-\frac12)-|i|)} & \mbox{if $|i|\geq p(\alpha-\frac12)$}.
\end{cases}
\end{equation}
Indeed, note that if $x\leq 2^{-8} t^{\beta}=2^{-8}\rho(t)$ then $z_0(t,x)\leq 0$, which implies $\omega_k(t,x)=0$.
Thus, \eqref{eq:iT} follows from \eqref{eq:eT}.

\emph{Step 2.} We prove \eqref{eq:L1} for $k=0$: for all $t>0$ sufficiently small,
\begin{equation}\label{eq:G0}
\int q^2(t) \omega_0 
\lesssim \begin{cases} 
t^{4+\frac {2( \alpha-6)}{4\alpha+3}} &  \mbox{if $1<\alpha<6$,}\\
t^4 |\log t|&\mbox{if $\alpha=6$,}\\
t^4 & \mbox{if $\alpha>6$.}
\end{cases}
\end{equation}
Set 
\[
G_0(t) = \int q^2(t,x) \omega_0(t,x) dx
\]
and compute using \eqref{eq:qq} and integration
by parts,
\begin{align*}
G_0'
& = 2 \int (\partial_t q) q \omega_0 + \int q^2 \partial_t \omega_0,\\
& = -3 \int (\partial_x q)^2 \partial_x \omega_0
+ \int q^2 \partial_t \omega_0
+\int q^2 \partial_x^3 \omega_0 
- 2 m_0^4\int q \partial_x((q+\Theta_{\sharp})^5 - \Theta_{\sharp}^5)\omega_0
+ 2\int qF_\sharp \omega_0.
\end{align*}
Note that the term $R_\sharp$ in the equation of $q$ does not appear in the identity above since
$R_\sharp(t,x)=0$ for $x>2t^{\frac 76\beta}$, $\omega_0(t,x)=0$ for $x<2^{-8} t^\beta$ and
$2t^{\frac 76\beta}<2^{-8} t^\beta$ for $t>0$ small.

By integration by parts, we see that
\begin{align*}
\int q  \partial_x((q+\Theta_{\sharp})^5 - \Theta_{\sharp}^5)\omega_0
& = - \int (\partial_x q) [(q+\Theta_{\sharp})^5 - \Theta_{\sharp}^5] \omega_0
- \int q[(q+\Theta_{\sharp})^5 - \Theta_{\sharp}^5] \partial_x \omega_0\\
&=\frac16\int[(q+\Theta_{\sharp})^6-6q\Theta_{\sharp}^5-\Theta_{\sharp}^6]\partial_x\omega_0
 - \int q[(q+\Theta_{\sharp})^5 - \Theta_{\sharp}^5] \partial_x \omega_0
\\&\quad+\int[(q+\Theta_{\sharp})^5-5q\Theta_{\sharp}^4-\Theta_{\sharp}^5](\partial_x\Theta_{\sharp})\omega_0.
\end{align*}
Using \eqref{eq:o1}, \eqref{eq:o3}, \eqref{eq:ot},
we obtain
\begin{align*}
G_0' & = -3 \int (\partial_x q)^2 s_0 
-2^{-7} {\rho_t}\int q^2 s_0
+  \int q^2 \chi_0\partial_x^3\varphi_0
- {\mu_t}  \int q^2 y_0s_0\\
&\quad - \frac {m_0^4}{3} \int [ (q+\Theta_{\sharp})^6 - 6 q \Theta_{\sharp}^5 - \Theta_{\sharp}^6]s_0
+2m_0^4\int q[(q+\Theta_{\sharp})^5 - \Theta_{\sharp}^5] s_0 \\
&\quad - 2 m_0^4\int [ (q+\Theta_{\sharp})^5 - 5 q \Theta_{\sharp}^4 -\Theta_{\sharp}^5] (\partial_x \Theta_{\sharp}) \omega_0
+ 2\int qF_\sharp \omega_0 +g_0
\end{align*}
where we have gathered error terms
containing derivatives of the function $\chi_0$ in the function $g_0$
defined below
\[
g_0  = 
2^{-8} \left({\rho_t} -\rho\frac{\mu_t}{\mu}\right)\int q^2  \partial_x\chi_0\varphi_0
+ \int q^2 \left( \partial_x^3 \chi_0\varphi_0 +3 \partial_x^2 \chi_0 \partial_x\varphi_0 
+ 3 \partial_x\chi_0 \partial_x^2\varphi_0\right) .
\]
Note that, using $1-\beta=2\nu$,
$\rho_t = \beta t^{\beta-1}=\beta t^{-2\nu}=\beta\mu^{-2}$.
Thus, using \eqref{eq:ph}, we have
\begin{equation}\label{eq:pp}
\partial_x^3 \varphi_0
\leq 2^{-9} \beta \mu^{-2}\partial_x\varphi_0
=2^{-9}{\rho_t} \partial_x\varphi_0
\quad\mbox{and so}\quad
 \int q^2 \chi_0\partial_x^3\varphi_0
 \leq 2^{-9} {\rho_t} \int q^2 s_0.
\end{equation}
Now, we treat the term $-{\mu_t} \int q^2 y_0 s_0$ 
in the above expression of $G_0'$.
For this, we distinguish three regions in space, depending on the value of $y_0$
(and thus of $x$).
Firstly, by $s_0\geq 0$, we have
\[
- \ {\mu_t}  \int_{y_0>0} q^2 y_0 s_0
=-\nu t^{\nu-1}\int_{y_0>0}q^2y_0 s_0
\leq  0.
\]
Second, let $c=2^{-10}\frac \beta{\nu}$.
In the region in $x$ where $- c t^{-(\nu-\beta)} < y_0 <0$, one has
$ -  {\mu_t}  y_0 \leq 2^{-10}\rho_t $, and so
\[
\left|{\mu_t}
\int_{-ct^{-(\nu-\beta)}<y_0<0}q^2 y_0 s_0\right|
\leq 2^{-10}{\rho_t}\int q^2s_0.
\]
In the region $y_0<- c t^{-(\nu-\beta)}$, we estimate,
for $t>0$ small, using the decay properties of the function $\varphi$
and $\nu-\beta>0$,
\begin{equation*}
\left| {\mu_t}\int_{y_0 < - c t^{-(\nu-\beta)}} q^2 y_0 s_0\right|
\lesssim t^{\nu-1}\|q\|_{L^2}^2 \sup_{y >c t^{-(\nu-\beta)}}
[ A y \sech(y/A)]
\lesssim t^{100},
\end{equation*}
where the exponent $100$ has no special meaning and could be replaced by any large number.
Thus, we have proved
\begin{equation}\label{eq:ft}
-{\mu_t} \int q^2 y_0 s_0 \leq 
- \nu t^{\nu-1}\int_{y_0>0}q^2 y_0 s_0 +  2^{-10}{\rho_t}\int q^2s_0+Ct^{100}.
\end{equation}
For the next two terms, we have, using $\|q(t)\|_{L^\infty}\lesssim 1$ and
$\|\Theta_\sharp\|_{L^\infty}\lesssim 1$, for $t>0$ small enough,
\begin{align*}
&\frac{m_0^4}{3 }
\left|\int[(q+\Theta_{\sharp})^6 - 6 q \Theta_{\sharp}^5-\Theta_{\sharp}^6]s_0\right|
+2  m_0^4\left|\int q[(q+\Theta_{\sharp})^5 - \Theta_{\sharp}^5] s_0 \right|\\
 &\quad \leq  C\int (q^6 + q^2\Theta_{\sharp}^4 ) s_0 
 \leq C\int q^2s_0
\leq 2^{-11}{\rho_t}  \int q^2 s_0,
\end{align*}
where for the last estimate, we have used $\lim_{t\to 0} {\rho_t} =+\infty$.
Next, we get from \eqref{eq:eT}
and from the fact that $\omega_0(t,x)=0$ for $x<2^{-8} t^\beta$,
\begin{align*}
 \left|\int [ (q+\Theta_{\sharp})^5 - 5 q \Theta_{\sharp}^4 -\Theta_{\sharp}^5] (\partial_x \Theta_{\sharp}) \omega_0\right|
 &\lesssim \int (|q|^5+ q^2\Theta_{\sharp}^3) |\partial_x \Theta_{\sharp}| \omega_0\\
&\lesssim \int q^2 |\partial_x \Theta_{\sharp}| \omega_0 
\lesssim t^{-\frac 12 \beta} \int q^2   \omega_0 
\lesssim t^{-\frac1{14}}G_0 .
\end{align*}
Then, by the Cauchy-Schwarz inequality, we have
\[
\left| \int qF_\sharp \omega_0 \right| \leq {G_0^\frac12} \left( \int F_\sharp^2 \omega_0 \right)^{\frac 12}.
\]
For $x\leq 2^{-8}\rho$, one has $z_0(t,x)\leq 0$ and so $\omega_0(t,x)=0$, and for
$x\geq 2^{-8}\rho =2^{-8} t^\beta \geq 2 t^{\frac76\beta}$ (for $t>0$ sufficiently small),
one has 
$\chi(t^{-\frac76\beta}x)\equiv 1$.
Thus, for any $t>0$ small and $x\in \RR$, one has
\[
F_\sharp^2(t,x)\omega_0(t,x)
=(t \Theta_1''' + m_0^4\partial_x (\Theta_\sharp^5-\Theta_0^5))^2.
\]
Using 
\eqref{eq:eT}, we have
\[
F_\sharp^2(t,x)\omega_0(t,x) 
\lesssim t^2x^{2\alpha-13}+t^2x^{10\alpha-13}\lesssim t^2x^{2\alpha-13}.
\]
It follows that
\begin{equation}\label{eq:F0}
 \int F_\sharp^2 \omega_0 
 \lesssim t^2 \int_{2^{-8}t^\beta}^\delta x^{2\alpha-13} dx \lesssim
\begin{cases} 
t^{2+\frac {2(\alpha-6)}{4\alpha+3}} &  \mbox{ if $1<\alpha<6$,}\\
t^2| \log t| &  \mbox{ if $\alpha=6$,}\\
t^2\delta^{2(\alpha -6)}& \mbox{ if $\alpha>6$.}
\end{cases}
\end{equation}
Lastly, we estimate the error term $g_0$.
Using \eqref{eq:cp} and $\|q\|_{H^1}+\|\Theta_\sharp\|_{H^1}\lesssim 1$,
one has, for $t>0$ sufficiently small,
\[
|g_0|\lesssim t^{100}.
\]

Gathering these estimates, and setting
\[
S_0 = 3  \int (\partial_x q)^2s_0
+2^{-8}\beta t^{-2\nu}\int q^2s_0
+\nu t^{\nu-1}\int_{y_0>0} q^2 y_0 s_0,
\]
we obtain
\begin{equation}\label{eq:dG}
G_0'+S_0
\lesssim t^{100}+t^{-\frac1{14}}G_0+
{G_0^\frac12}\begin{cases}
 t^{1+\frac {\alpha-6}{4\alpha+3}}&\mbox{for $1<\alpha<6$,}\\
 t|\log t|^\frac12&\mbox{for $\alpha=6$,}\\
 t&\mbox{for $\alpha>6$.}
\end{cases}
\end{equation}
The above estimate will be useful in the next steps. For now, it is enough to write the
following simple consequence
\[
G_0'
\lesssim t^{100}+t^{-\frac1{14}}G_0+
{G_0^\frac12}\begin{cases}
 t^{1+\frac {\alpha-6}{4\alpha+3}}&\mbox{for $1<\alpha<6$,}\\
 t|\log t|^\frac12&\mbox{for $\alpha=6$,}\\
 t &\mbox{for $\alpha>6$.}
\end{cases}
\]
and to observe that \eqref{eq:G0} follows from Lemma \ref{LE:ax}
and $G_0(0)=0$. For convenience, we set $H_0:=G_0$.
\medskip

\emph{Step 3.}
In this step, we prove \eqref{eq:L1} for $k=1$: for all $t>0$ small,
\begin{equation}\label{eq:g1}
\int (\partial_x q)^2(t)\omega_1
\lesssim  \begin{cases} 
t^{4+\frac{2(\alpha-7)}{4\alpha+3}} &  \mbox{ if $1<\alpha<7$,}\\
t^4 |\log t| &  \mbox{ if $\alpha=7$,}\\
t^4 & \mbox{ if $\alpha>7$.}
\end{cases}
\end{equation}
Considering simply the functional $\int (\partial_x q)^2\omega_1$ does not seem sufficient to prove \eqref{eq:g1}, and we are led to consider a functional related to the energy conservation.
Let
\[
G_1(t) = \int \Bigl((\partial_x q)^2 - \frac {m_0^4}3 \left((q+\Theta_\sharp)^6 - \Theta_\sharp^6 - 6 q \Theta_\sharp^5\right)\Bigr)(t,x) \omega_1(t,x) dx.
\]
Since (as in Step 1),
\[
\left|\int ((q+\Theta_\sharp)^6 - \Theta_\sharp^6 - 6 q \Theta_\sharp^5)\omega_1\right|
\lesssim \int q^2 \omega_1\lesssim \int q^2 \omega_0,
\]
there exists a constant $C_0>1$ such that 
\[
H_1:=G_1+C_0 G_0 \geq \int  (\partial_x q)^2\omega_1  + \int q^2  \omega_0.
\]
We compute
\begin{align*}
G_1'  & = - 2 \int (\partial_t q) \left(\partial_x^2 q + m_0^4 \left((q+\Theta_\sharp)^5- \Theta_\sharp^5\right)\right)\omega_1
- 2 \int (\partial_t q) (\partial_x q) \partial_x \omega_1\\
&\quad + \int \Bigl((\partial_x q)^2 - \frac {m_0^4}3 \left((q+\Theta_\sharp)^6 - \Theta_\sharp^6 - 6 q \Theta_\sharp^5\right)\Bigr)\partial_t \omega_1  \\ 
& \quad -2m_0^4\int \left((q+\Theta_\sharp)^5-\Theta_\sharp^5-5q\Theta_\sharp^4 \right)(\partial_t\Theta_\sharp) \omega_1.
\end{align*}
Using \eqref{eq:qq},
we obtain after integration by parts
\begin{align*}
G_1'  & = -   \int \left(\partial_x^2 q + m_0^4 ((q+\Theta_\sharp)^5- \Theta_\sharp^5)\right)^2\partial_x\omega_1\\
&\quad + \int \Bigl((\partial_x q)^2 - \frac {m_0^4}3 ((q+\Theta_\sharp)^6 - \Theta_\sharp^6 - 6 q \Theta_\sharp^5)\Bigr)\partial_t \omega_1\\
&\quad + 2 \int \left(\partial_x (\partial_x^2 q + m_0^4 \left((q+\Theta_\sharp)^5- \Theta_\sharp^5)\right)\right) (\partial_x q)  \partial_x\omega_1\\
&\quad -2m_0^4\int \left((q+\Theta_\sharp)^5-\Theta_\sharp^5-5q\Theta_\sharp^4 \right)(\partial_t\Theta_\sharp) \omega_1 \\
&\quad +2 \int (\partial_x F_\sharp) (\partial_x q)  \omega_1
- 2  m_0^4 \int F_\sharp  ((q+\Theta_\sharp)^5- \Theta_\sharp^5)\omega_1,
\end{align*}
where, as in the previous step, the term $R_\sharp$ in equation \ref{eq:qq} does not appear in the expression of
$G_1'$.
Note that the first term on the right-hand side of the above identity
is nonnegative.
Then, using \eqref{eq:o1}, \eqref{eq:o3}, \eqref{eq:ot}, we obtain
\begin{align*}
G_1' & \leq
- 2 \int (\partial_x^2 q)^2 s_1 -2^{-7} \rho_t  \int(\partial_x q)^2s_1
+\int(\partial_xq)^2\chi_1\partial_x^3 \varphi_1 \\
&\quad + 2^{-7}\frac{m_0^4}3 \rho_t\int \left((q+\Theta_\sharp)^6 - \Theta_\sharp^6 - 6 q \Theta_\sharp^5\right)s_1\\
&\quad - {\mu_t}\int (\partial_x q)^2 y_0 s_1
+ \frac {m_0^4}3 \mu_t\int\left((q+\Theta_\sharp)^6-\Theta_\sharp^6-6q\Theta_\sharp^5\right)y_0 s_1\\
&\quad +10 m_0^4 \int (\partial_x q)^2 (q+\Theta_\sharp)^4 s_1
+10 m_0^4\int  \left((q+\Theta_\sharp)^4-\Theta_\sharp^4\right)
(\partial_x q) (\partial_x \Theta_\sharp)s_1\\
&\quad -2m_0^4\int \left((q+\Theta_\sharp)^5-\Theta_\sharp^5-5q\Theta_\sharp^4 \right)(\partial_t\Theta_\sharp) \omega_1 \\
&\quad +2\int (\partial_xF_\sharp) (\partial_x q) \omega_1- 2m_0^4 \int F_\sharp  ((q+\Theta_\sharp)^5- \Theta_\sharp^5)\omega_1
+g_1
\end{align*}
where
\begin{align*}
g_1 & = 2^{-8} \left({\rho_t} -\rho\frac{\mu_t}{\mu}\right)  \int \Bigl((\partial_x q)^2 - \frac {m_0^4}3 \left((q+\Theta_\sharp)^6 - \Theta_\sharp^6 - 6 q \Theta_\sharp^5\right)\Bigr) \partial_x \chi_1 \varphi_1\\
&\quad
+ \int (\partial_x q)^2 (\partial_x^3\chi_1\varphi_1+3 \partial_x^2\chi_1\partial_x\varphi_1+3\partial_x\chi_1\partial_x^2\varphi_1) .
\end{align*}
Firstly,  using \eqref{eq:ph} as in the previous step (see \eqref{eq:pp}), we have
\[
\int (\partial_x q)^2 \chi_1 \partial_x^3 \varphi_1
 \leq 2^{-9}  \rho_t  \int (\partial_x q)^2 s_1.
\]
Then, using $\|q\|_{L^\infty} + \|\Theta_\sharp\|_{L^\infty} \lesssim 1$, we have
\begin{equation} \label{eq:GS}
\left| \rho_t \int \left((q+\Theta_\sharp)^6 - \Theta_\sharp^6 - 6 q \Theta_\sharp^5\right) s_1\right|
\lesssim t^{-2\nu} \int q^2 s_1\lesssim S_0.
\end{equation}
For the next term, we distinguish three regions in space as in Step 1.
Firstly,
\[
- \mu_t\int_{y_0>0}  (\partial_x q)^2 y_0 s_1
\leq 0.
\]
Recall $c = 2^{-10}\frac \beta{\nu}$.
Then, in the region in $x$ where $- c t^{-(\nu-\beta)} < y_0 <0$, one has
$ -  {\mu_t}  y_0 \leq 2^{-10}\rho_t $, and so
\[
\left| \mu_t  \int_{- c t^{-(\nu-\beta)} < y_0 <0} (\partial_x q)^2 y_0 s_1\right|
\leq 2^{-10} \rho_t \int (\partial_x q)^2 s_1.
\]
Moreover, for $y_0 <- c t^{-(\nu-\beta)}$ and $t>0$ small enough,
\[
\left|  \mu_t \int_{y_0 <- ct^{-(\nu-\beta)}} (\partial_x q)^2 y_0 s_1\right|
\lesssim t^{100} \|\partial_x q\|_{L^2}^2 \lesssim t^{100}.
\]
Thus,
\begin{equation} \label{eq:ys}
 -  \mu_t  \int (\partial_x q)^2  y_0 s_1
\leq -\nu t^{\nu-1} \int_{y_0>0}(\partial_x q)^2y_0 s_1 
+ 2^{-9}  \rho_t \int (\partial_x q)^2 s_1+C t^{100}.
\end{equation}
For the next term, we proceed as before
\begin{align*}
\left| \mu_t \int  \left((q+\Theta_\sharp)^6 - \Theta_\sharp^6 - 6 q \Theta_\sharp^5\right) y_0s_1\right|
&\lesssim t^{\nu-1} \int q^2 |y_0| s_1\\
&\lesssim  t^{\nu-1} \int_{y_0>0} q^2y_0s_1+ \rho_t \int q^2s_1 +t^{100}
\lesssim S_0+t^{100}.
\end{align*}
Then, using again $\|q\|_{L^\infty} + \|\Theta_\sharp\|_{L^\infty} \lesssim 1$, we have, for $t$ small,
\[
\left|  \int (\partial_x q)^2 (q+\Theta_\sharp)^4s_1\right|
\leq C  \int (\partial_x q)^2  s_1
\leq 2^{-11} \rho_t \int (\partial_x q)^2 s_1.
\]
Next, by using \eqref{eq:eT}, the condition $\alpha>1$,
and the fact that $s_1(t,x)=0$ for $x<t^\beta$, we see that
$|\partial_x \Theta_\sharp|s_1\lesssim t^{-\frac 1{14}} s_1$.
Thus,
\begin{align*}
\left| \int((q+\Theta_\sharp)^4-\Theta_\sharp^4)
(\partial_x q) (\partial_x \Theta_\sharp)s_1\right|
&\leq C  \int |q| |\partial_x q|
|\partial_x\Theta_\sharp| s_1 \\ 
&\leq C t^{-\frac1{14}}\int |q||\partial_xq|s_1
 \leq 2^{-12}\rho_t \int (\partial_x q)^2 s_1 + C S_0.
\end{align*}
Using \eqref{eq:eT} and the fact that $\omega_1(t,x)=0$ for $x\ge 2^{-8}\rho\ge 2t^{\frac76 \beta}$ (for $t$ sufficiently small), we have $|\partial_t\Theta_\sharp|\lesssim t^{-5\beta}\lesssim t^{-\frac57}$, so that 
\begin{align*}
\left|\int \left((q+\Theta_\sharp)^5-\Theta_\sharp^5-5q\Theta_\sharp^4 \right)(\partial_t\Theta_\sharp) \omega_1\right| \lesssim \int q^2 |\partial_t \Theta_\sharp| \omega_1 \lesssim t^{-\frac57} H_1 .
\end{align*}
Finally, it follows from the Cauchy-Schwarz inequality that
\[
\left| \int (\partial_x F_\sharp) (\partial_x  q)\omega_1 \right|
\lesssim \left( \int (\partial_x q)^2 \omega_1\right)^{\frac 12}
 \left( \int (\partial_x F_\sharp)^2 \omega_1\right)^{\frac 12}.
\]
Using \eqref{eq:eT}, one has
$(\partial_x F_\sharp)^2\omega_1 
\lesssim t^2x^{2\alpha-15},
$
and thus
\[
\int (\partial_x F_\sharp)^2 \omega_1 
\lesssim t^2 \int_{2^{-8}t^\beta}^\delta x^{2\alpha-15} dx \lesssim
\begin{cases} 
t^{2+\frac {2(\alpha-7)}{4\alpha+3}} &  \mbox{ if $1<\alpha<7$,}\\
t^2 |\log t| &  \mbox{ if $\alpha=7$,}\\
t^2 \delta^{2(\alpha -7)}& \mbox{ if $\alpha>7$.}
\end{cases}
\]
Using \eqref{eq:G0} and \eqref{eq:F0}, for $t$ small,
\[
\left|\int F_\sharp ((q+\Theta_\sharp)^5- \Theta_\sharp^5) \omega_1\right|
\lesssim \left( \int F_\sharp^2 \omega_1\right)^\frac12
\left( \int q^2 \omega_1\right)^{\frac 12}.
\]
Lastly, similarly as for $g_0$, using \eqref{eq:cp}, we have $|g_1|\lesssim t^{100}$.

Gathering these estimates, and setting
\[
S_1 = 3 \int (\partial_x^2 q)^2s_1
+2^{-8} \beta t^{-2\nu}\int (\partial_x q)^2 s_1
+\nu t^{\nu-1} \int_{y_0>0} (\partial_x q)^2 y_0 s_1,
\]
we obtain
\begin{equation}\label{eq:G1}
G_1'+S_1\lesssim  t^{100} +S_0 +t^{-\frac57} H_1
+ {H_1^\frac12} \begin{cases}
 t^{1+\frac {\alpha-7}{4\alpha+3}}&\mbox{for $1<\alpha<7$,}\\
 t|\log t|^\frac12&\mbox{for $\alpha=7$,}\\
 t&\mbox{for $\alpha>7$.}
\end{cases}
\end{equation}
Combining \eqref{eq:G1} with \eqref{eq:dG}, we deduce that there exists a constant $\widetilde C_0\geq C_0$ such that
setting $\widetilde H_1=G_1+\widetilde C_0G_0$, we have
\begin{align*}
\widetilde H_1'  & \lesssim 
t^{100} +t^{-\frac57}\widetilde H_1 + {\widetilde H_1^\frac 12}
 \begin{cases}
 t^{1+\frac {\alpha-7}{4\alpha+3}}&\mbox{for $1<\alpha<7$,}\\
 t|\log t|^\frac12&\mbox{for $\alpha=7$,}\\
 t &\mbox{for $\alpha>7$.}
\end{cases}
\end{align*}
Using $\widetilde H_1(0)=0$, estimate \eqref{eq:g1} follows from Lemma \ref{LE:ax}
applied to the function $\widetilde H_1$.

Note for future reference that from \eqref{eq:G0} and \eqref{eq:g1}, we deduce
\begin{equation}\label{eq:H1}
H_0\lesssim t^\frac{18}7,\quad
H_1\lesssim t^\frac{16}7.
\end{equation}

\medskip

\emph{Step 4.} We prove \eqref{eq:L1} for $k=2$, i.e.
\begin{equation}\label{eq:g2}
\int (\partial_x^2 q)^2 \omega_2 
 \lesssim  \begin{cases} 
t^{4+\frac {2( \alpha-8)}{4\alpha+3}} &  \mbox{ if $1<\alpha<8$,}\\
 t^4 |\log t| &  \mbox{ if $\alpha=8$,}\\
t^4 & \mbox{ if $\alpha>8$.}
\end{cases}
\end{equation}
Since we do not have a good control on $\partial_xq$ in $L^\infty$, we use a modified energy. For  $m\in\{2,\ldots,5\}$, let
\[
G_2 = G_{2,0}+ \frac56m_0^4 \sum_{m=2}^5mC_5^m G_{2,m},\quad
G_{2,0} = \int (\partial_x^2 q)^2 \omega_2 ,\quad
G_{2,m} = \int (\partial_x q)^2q^{m-1}\Theta_\sharp^{5-m}\omega_2,
\]
and
\[
S_2 = 3  \int (\partial_x^3 q)^2s_2
+2^{-8} \beta t^{-2\nu}\int (\partial_x^2 q)^2 s_2
+\nu t^{1-\nu}\int_{y_0>0} (\partial_x^2 q)^2 y_0s_2.
\]
Since $|G_{2,m}|\lesssim H_1$, we can choose a positive constant $C_1>1$ such that setting
\[
  H_2 = G_2+  C_1  H_1,
\]
it holds
\[
  H_2 \geq \sum_{l=0}^{2} \int (\partial_x^l q)^2 \omega_l.
\]
We compute the derivative of the functional $G_2$ defined above.
First, we compute using \eqref{eq:qq}
\begin{align*}
G_{2,0}' & =  2\int (\partial_x^2\partial_t q) (\partial_x^2q) \omega_2
+ \int (\partial_x^2 q)^2 \partial_t \omega_2 \\
& = - 3 \int (\partial_x^3 q)^2 \partial_x \omega_2
+ \int (\partial_x^2 q)^2 \partial_t \omega_2 
+ \int (\partial_x^2 q)^2 \partial_x^3 \omega_2 \\
&\quad
- 2 m_0^4\int (\partial_x^3 ((q+\Theta_\sharp)^5-\Theta_\sharp^5)) (\partial_x^2 q)  \omega_2
+ 2\int (\partial_x^2 F_\sharp) (\partial_x^2 q)\omega_2.
\end{align*}
(The term $R_\sharp$ is eliminated as in the previous steps.)
Using \eqref{eq:qq}, we find
\begin{align*}
G_{2,0}' & \leq - 3  \int (\partial_x^3 q)^2 s_2
 - 2^{-7}\rho_t \int (\partial_x^2 q)^2 s_2
 +  \int (\partial_x^2 q)^2 \chi_2\partial_x^3 \varphi_2 \\
&\quad - 2\int (\partial_x^3 ((q+\Theta_\sharp)^5-\Theta_\sharp^5)) (\partial_x^2 q) \omega_2
- \mu_t \int (\partial_x^2 q)^2 y_0 s_2
+ 2m_0^4\int (\partial_x^2 F_\sharp) (\partial_x^2  q)\omega_2
+g_{2,0}
\end{align*}
where we set
\begin{align*}
g_{2,0} & =
2^{-8} \left(\rho_t-\rho \frac{\mu_t}{\mu}\right) \int (\partial_x^2 q)^2 \partial_x \chi_2\varphi_2 + \int (\partial_x^2 q)^2 \left(\partial_x^3 \chi_2\varphi_2
+3 \partial_x^2 \chi_2 \partial_x\varphi_2 + 3 \partial_x\chi_2\partial_x^2 \varphi_2\right).
\end{align*}

The second and third terms on the right-hand side are treated as before
\begin{equation*}
 -2^{-7} \rho_t \int (\partial_x^2 q)^2 s_2 
 + \int (\partial_x^2 q)^2 \chi_2 \partial_x^3\varphi_2
 \leq 
 - 3 \cdot 2^{-9} \rho_t \int (\partial_x^2 q)^2 s_2.
\end{equation*}
To deal with the next term, 
we argue as in \eqref{eq:ys} and find that
\begin{equation*} 
- \mu_t \int (\partial_x^2 q)^2  y_0 s_2
\leq
-  \nu t^{\nu-1} \int_{y_0>0}(\partial_x^2 q)^2y_0 s_2 
+2^{-10} \rho_t  \int (\partial_x^2 q)^2 s_2 +C t^{100}.
\end{equation*}
We handle the source term containing $F_\sharp$ by using the Cauchy-Schwarz inequality,
\[
\left| \int (\partial_x^2 F_\sharp) (\partial_x^2 q)\omega_2 \right| \leq G_{2,0}^\frac 12 \left( \int (\partial_x^2 F_\sharp)^2 \omega_2\right)^{\frac 12}.
\]  
Using \eqref{eq:eT}, one has
$(\partial_x^2 F_\sharp)^2\omega_2
\lesssim t^2x^{2\alpha-17},
$
and thus
\[
\int (\partial_x^2 F_\sharp)^2 \omega_2
\lesssim t^2 \int_{2^{-8}t^\beta}^\delta x^{2\alpha-17} dx \lesssim
\begin{cases} 
t^{2+\frac {2(\alpha-8)}{4\alpha+3}} &  \mbox{ if $1<\alpha<8$,}\\
t^2 |\log t| &  \mbox{ if $\alpha=8$,}\\
t^2 & \mbox{ if $\alpha>8$.}
\end{cases}
\]
Using \eqref{eq:cq}, we have $|g_{2,0}|\lesssim S_1$.

Lastly, we deal with the non-linear term
\[\int (\partial_x^3 ((q+\Theta_\sharp)^5-\Theta_\sharp^5)) (\partial_x^2 q) \omega_2=\sum_{m=1}^5C_5^m \int (\partial_x^3(q^m \Theta_\sharp^{5-m})) (\partial_x^2q) \omega_2 . \]
By expanding the derivatives, we observe that for $m=1,\ldots,5$,
\[
\int (\partial_x^3(q^m \Theta_\sharp^{5-m})) (\partial_x^2 q) \omega_2
=\sum_{l=0}^{3} C_3^l I_{2,l,m}
\]
where for $l=0,\ldots,3$,
\[
I_{2,l,m}=\int[\partial_x^{l}(q^m)](\partial_x^2 q) [\partial_x^{3-l}(\Theta_\sharp^{5-m})]\omega_2.
\]
We estimate the terms $I_{2,l,m}$, $m=1,\ldots, 5$, starting with the case $l=3$.
We decompose $I_{2,3,m}$ as follows
\begin{align*}
I_{2,3,m} &=
m \int q^{m-1}(\partial_x^3q)(\partial_x^2q)\Theta_\sharp^{5-m} \omega_2+3m(m-1)\int q^{m-2}(\partial_xq)(\partial_x^2q)^2 \Theta_\sharp^{5-m} \omega_2 \\
 &\quad + m(m-1)(m-2) \int q^{m-3} (\partial_xq)^3 (\partial_x^2 q) \Theta_\sharp^{5-m} \omega_2 .
\end{align*}
We set
\[
T_{2,m}=(m-1)\int (\partial_x^2 q)^2 (\partial_xq) q^{m-2}\Theta_\sharp^{5-m}\omega_2.
\]
After integrating by parts the first term of $I_{2,3,m}$, we obtain 
\begin{align*}
I_{2,3,m}&=
\frac52 m T_{2,m}+\sum_{j=1}^5I_{2,3,m}^j,
\end{align*}
where 
\begin{align*}
I_{2,3,m}^1&= - \frac14 m(m-1)(m-2)(m-3) \int q^{m-4} (\partial_xq)^5  \Theta_\sharp^{5-m} \omega_2 ; \\
I_{2,3,m}^2&=-\frac{m}2 \int (\partial_x^2 q)^2 q^{m-1} (\partial_x (\Theta_\sharp^{5-m}))\omega_2 ;\\ 
I_{2,3,m}^3 &= - \frac14 m(m-1)(m-2) \int q^{m-3} (\partial_xq)^4  (\partial_x(\Theta_\sharp^{5-m})) \omega_2 ;\\ 
I_{2,3,m}^4&= -\frac{m}2 \int (\partial_x^2 q)^2 q^{m-1} \Theta_\sharp^{5-m}(\partial_x \omega_2) ;\\ 
I_{2,3,m}^5 &=- \frac14 m(m-1)(m-2) \int q^{m-3} (\partial_xq)^4  \Theta_\sharp^{5-m} (\partial_x\omega_2) .
\end{align*}
The term $\frac 52 m T_{2,m}$ in the expression of $I_{2,3,m}$ 
will be canceled out by using the modified energy. To handle the terms $I_{2,3,m}^{1,\ldots,5}$, we observe from \eqref{eq:ww} $\omega_2^{\frac14}=\omega_1^2 \omega_1^{\frac34}$. 
Moreover, we recall from \eqref{eq:eT} that $|\partial_x\Theta_\sharp| \omega_2 \lesssim t^{-\frac1{14}} \omega_2$. 
Thus, it follows from $\|\Theta_\sharp\|_{L^{\infty}} \lesssim 1$, $\|q\|_{L^{\infty}} \lesssim 1$, \eqref{eq:Lz}
and \eqref{eq:H1} that
\begin{align*}
\left| I_{2,3,m}^1 \right| &\lesssim \bigr\| (\partial_xq)^2 \omega_1^{\frac12}\omega_2^{\frac12}\bigr\|_{L^{\infty}}^{\frac32} \left(\int (\partial_xq)^2 \omega_1^2 \right)  \lesssim \left( H_1^{\frac12}H_2^{\frac12}+t^{-\nu} H_1\right)^{\frac32} H_1
\lesssim H_2 \\
\left| I_{2,3,m}^2 \right|  &\lesssim \int (\partial_x^2 q)^2 |\partial_x\Theta_\sharp| \omega_2 \lesssim t^{-\frac1{14}}H_2, \\ 
\left| I_{2,3,m}^3 \right|
&\lesssim \bigr\| (\partial_xq)^2 \omega_1^{\frac12}\omega_2^{\frac12}\bigr\|_{L^{\infty}} \int (\partial_xq)^2 |\partial_x\Theta_\sharp| \omega_1^5 \lesssim \left( H_1^{\frac12}H_2^{\frac12}+t^{-\nu}H_1\right) t^{-\frac 1{14}} H_1
\lesssim H_2 ,\\
\left|I_{2,3,m}^4  \right| &\lesssim \int (\partial_x^2 q)^2 s_2 \lesssim t^{2\nu}S_2 ,\\
\left|I_{2,3,m}^5 \right| 
& \lesssim  
\bigr\| (\partial_xq)^2 \omega_1^{\frac12}\omega_2^{\frac12}\bigr\|_{L^{\infty}} \int (\partial_xq)^2 s_1  \lesssim t^{-\nu}\left( H_1^{\frac12}H_2^{\frac12}+t^{-\nu}H_1\right) H_1 \lesssim H_2
\end{align*}
Now, we turn to $I_{2,2,m}$,
for $m=1,\ldots,4$.
After integrating by parts, we decompose $I_{2,2,m}$ as 
\begin{align*}
I_{2,2,m}=\sum_{j=1}^4 I_{2,2,m}^j ,
\end{align*}
where
\begin{align*}
I_{2,2,m}^1 &=m \int q^{m-1} (\partial_x^2q)^2  (\partial_x(\Theta_\sharp^{5-m})) \omega_2 \\ 
I_{2,2,m}^2 & =-\frac13 m(m-1)(m-2)  \int q^{m-3} (\partial_xq)^4  (\partial_x(\Theta_\sharp^{5-m})) \omega_2 \\ 
I_{2,2,m}^3 &=-\frac13 m(m-1) \int q^{m-2}(\partial_xq)^3 (\partial_x^2(\Theta_\sharp^{5-m})) \omega_2 \\ 
I_{2,2,m}^4 &=-\frac13 m(m-1) \int q^{m-2}(\partial_xq)^3 (\partial_x(\Theta_\sharp^{5-m}))   (\partial_x\omega_2) .
\end{align*}
We estimate the terms $I_{2,2,m}^j$ one by one. Note that  $I_{2,2,m}^1 = c I_{2,3,m}^2$ and 
$I_{2,2,m}^2 =c I_{2,3,m}^3 $, so these contributions are   controlled as above. It follows from 
$\|q\|_{L^{\infty}} \lesssim 1$, \eqref{eq:eT}, \eqref{eq:Lz} and \eqref{eq:H1} that
\begin{align*}
\left| I_{2,2,m}^3 \right|   &\lesssim t^{-\frac 3{14}} \bigr\| (\partial_xq)^2 \omega_1^{\frac12}\omega_2^{\frac12}\bigr\|_{L^{\infty}}^\frac12 \left( \int (\partial_xq)^2 \omega_1\right)    \lesssim t^{-\frac 3{14}} \left( H_1^{\frac12}H_2^{\frac12}+t^{-\nu}H_1\right)^\frac12 H_1 \lesssim H_2,\\
\left| I_{2,2,m}^4 \right|  
&\lesssim t^{ -\nu-\frac1{14}}   \left( H_1^{\frac12}H_2^{\frac12}+t^{-\nu}H_1\right)^\frac12 H_1
  \lesssim H_2 
\end{align*}
Now, for $m=1,\ldots,4$, we deal with 
\[
I_{2,1,m}= m \int (\partial_x^2 q) (\partial_x q) q^{m-1}
\partial_x^2 (\Theta_\sharp^{5-m}) \omega_2.
\]
By \eqref{eq:eT}, $\|q\|_{L^\infty}\lesssim 1$ and \eqref{eq:H1},
\[
|I_{2,1,m}|\lesssim t^{-\frac3{14}} \int |\partial_x^2 q| |\partial_x q|   \omega_2
\lesssim t^{-\frac3{14}} H_2^\frac12 H_1^\frac12
\lesssim t^{-\frac3{14}} H_2
\]
Lastly, we deal similarly with $ I_{2,0,m}$ for $m=1,\ldots,4$,
using \eqref{eq:G0},
\begin{equation*}
     \left| I_{2,0,m} \right|=t^{-\frac5{14}}
     \int |\partial_x^2 q| |q|   \omega_2
\lesssim t^{-\frac5{14}} H_2^\frac12 H_0^\frac12
\lesssim t^{-\frac5{14}} H_2
\end{equation*}
We summarize the above estimates, for $t$ small,
\begin{equation}\label{eq:8d}
G_{2,0}'+ S_2-\frac52 m_0^4 \sum_{m=2}^5 m C_5^m T_{2,m}
\lesssim 
t^{100}+S_1+t^{-\frac5{14}} H_2 
+ H_2^\frac12\begin{cases} 
t^{1+\frac {\alpha-8}{4\alpha+3}} &  \mbox{ if $1<\alpha<8$,}\\
t |\log t|^\frac 12 &  \mbox{ if $\alpha=8$,}\\
t & \mbox{ if $\alpha>8$.}
\end{cases}
\end{equation}
Second, we compute, for $m=2,\ldots,5$, the time derivative of $G_{2,m}$,
using \eqref{eq:qq} and \eqref{eq:ot},
\begin{align*}
G_{2,m}' & = 2 \int (\partial_x\partial_t q)(\partial_xq)q^{m-1}\Theta_\sharp^{5-m}\omega_2
+(m-1)\int (\partial_xq)^2 (\partial_t q)q^{m-2}\Theta_\sharp^{5-m}\omega_2\\
&\quad +\int(\partial_xq)^2q^{m-1}\Theta_\sharp^{5-m}\partial_t\omega_2
+\int(\partial_xq)^2q^{m-1}\partial_t\left(\Theta_\sharp^{5-m}\right)\omega_2\\
& = -3T_{m,2}
-3 \int (\partial_x^2 q)^2 [\partial_x (q^{m-1}\Theta_\sharp^{5-m}\omega_2)
-(m-1)(\partial_x q)q^{m-2}\Theta_\sharp^{5-m}\omega_2]\\
&\quad 
+ \int (\partial_x q)^2 [\partial_x^3 (q^{m-1}\Theta_\sharp^{5-m} \omega_2) - (m-1)(\partial_x^3q) q^{m-2} \Theta_\sharp^{5-m} \omega_2]\\
&\quad
-2^{-7} \rho_t\int(\partial_xq)^2q^{m-1}\Theta_\sharp^{5-m}s_2
- \mu_t \int(\partial_xq)^2q^{m-1}\Theta_\sharp^{5-m}y_0 s_2
\\& \quad -2 m_0^4\int [\partial_x^2((q+\Theta_\sharp)^5-\Theta_\sharp^5)] (\partial_x q) q^{m-1}\Theta_\sharp^{5-m} \omega_2\\
&\quad -(m-1)m_0^4 \int (\partial_x q)^2 (\partial_x ((q+\Theta_\sharp)^5-\Theta_0^5)) q^{m-2}\Theta_\sharp^{5-m}\omega_2
\\
&\quad +\int(\partial_xq)^2q^{m-1}\partial_t\left(\Theta_\sharp^{5-m}\right)\omega_2 +2\int (\partial_x F_\sharp) (\partial_x q) q^{m-1} \Theta_\sharp^{5-m} \omega_2\\ & \quad
+(m-1)\int (\partial_x q)^2 F_\sharp q^{m-2}\Theta_\sharp^{5-m} \omega_2
+g_{2,m},
\end{align*}
where
\[
g_{2,m} = 2^{-8}\left(\rho_t-\rho
\frac{ \mu_t}{\mu}\right)
\int(\partial_xq)^2q^{m-1}\Theta_\sharp^{5-m}
\partial_x \chi_2\varphi_2. 
\]
The term $T_{m,2}$ will be used to cancel a similar term in \eqref{eq:8d}.
Note that arguing as in the proof of Lemma \ref{LE:qq},
$|g_{2,m}| \lesssim t^{100}$.

Now, we estimate the other terms in the expression of $G_{2,m}'$.
Observe that 
\begin{align*}
 \int (\partial_x^2 q)^2 [\partial_x (q^{m-1}\Theta_\sharp^{5-m}\omega_2)
-(m-1)(\partial_x q)q^{m-2}\Theta_\sharp^{5-m}\omega_2] =
c_1I_{2,3,m}^2+c_2I_{2,3,m}^4 
\end{align*}
for some constants $c_1$, $c_2$, so that this term is controled as before.
Then, there exist constants $c_3,\ldots,c_8$, such that
\begin{align*}
& \int (\partial_x q)^2  [\partial_x^3 (q^{m-1}\Theta_\sharp^{5-m} \omega_2) - (m-1)(\partial_x^3q) q^{m-2} \Theta_\sharp^{5-m} \omega_2] \\ & \quad= c_3I_{2,3,m}^3+c_4I_{2,3,m}^5
+c_5I_{2,2,m}^3+c_6I_{2,2,m}^4
+c_7(m-1)J_1
+c_8(m-2)J_2 +c_9 (m-3)J_3,
\end{align*}
where
\begin{align*}
J_1 & = \int(\partial_x^2 q) (\partial_x q) q^{m-2} \partial_x^2(\Theta_\sharp^{5-m} \omega_2),\\
J_2 & = \int(\partial_x^2 q) (\partial_x q)^2 q^{m-3} \Theta_\sharp^{5-m}\partial_x \omega_2,\\
J_3 & = \int (\partial_x q)^5 q^{m-4} \Theta_\sharp^{5-m} \omega_2.
\end{align*}
First, we deal with $J_1$
\begin{align*}
J_1
&  = 
\int(\partial_x^2 q) (\partial_x q) q^{m-2} \partial_x^2(\Theta_\sharp^{5-m})
\omega_2
+\int(\partial_x^2 q) (\partial_x q) q^{m-2} \partial_x(\Theta_\sharp^{5-m})
s_2\\
&\quad  -\int(\partial_x^3 q) (\partial_x q) q^{m-2}  \Theta_\sharp^{5-m} s_2
- \int(\partial_x^2 q)^2  q^{m-2} \Theta_\sharp^{5-m} s_2\\
&\quad - (m-2) 
\int(\partial_x^2 q) (\partial_x q)^2 q^{m-3}  \Theta_\sharp^{5-m} s_2\\
& = J_{1,1}+J_{1,2}+J_{1,3}+J_{1,4}+J_{1,5}.
\end{align*}
We estimate $J_{1,1}$,\ldots,$J_{1,4}$ using from \eqref{eq:eT} and the restriction on $\alpha>1$ that $\left|\partial_x(\Theta_\sharp^5)\right| \lesssim t^{-\frac1{14}}$ and $\left|\partial_x^2(\Theta_\sharp^5)\right| \lesssim t^{-\frac3{14}}$. We have, for $t$ small,
\begin{align*}
|J_{1,1}| &\lesssim t^{-\frac3{14}} H_1^{\frac12}H_2^{\frac12} \lesssim t^{-\frac3{14}} H_2  ;\\
|J_{1,2}|&\lesssim t^{-\frac{\nu}2}t^{-\frac1{14}} H_1^{\frac12} \left(t^{2\nu}S_2\right)^{\frac12} \leq  \frac1{10}S_2+CH_2;\\
|J_{1,3}|&\lesssim  S_0^\frac12 S_2^{\frac12} \leq \frac1{10}S_2+CS_0 ;\\
|J_{1,4}|&\leq  \frac1{10}S_2 .
 \end{align*}
Moreover, from \eqref{eq:Lz} and $H_1 \lesssim t^{\frac{16}7}$, 
\begin{align*}
\left|J_{1,5} \right| &\lesssim 
  t^{-\frac\nu2}\Bigr\| (\partial_xq) \omega_1^{\frac12}\omega_2^{\frac12} \Bigr\|_{L^{\infty}}^{\frac12} \left(\int (\partial_xq)^2 
\omega_1\right)^{\frac12} \left(\int (\partial_x^2q)^2 
s_2\right)^{\frac12} \\ 
& \lesssim t^{-\frac\nu2} \left( H_1^{\frac12}H_2^{\frac12}+t^{-\nu}H_1\ \right)^{\frac12}H_1^{\frac12} \left(t^{2\nu}S_2 \right)^{\frac12}
\leq \frac 1{10}S_2+C H_2 .
\end{align*}
Moreover, up to a multiplicative constant $J_2$ is $J_{1,5}$.
To estimate $J_3$, we write
\begin{equation*}
|J_3|\lesssim \|(\partial_x q)^2 \omega_1^\frac12\omega_2^\frac12\|_{L^\infty}^\frac32 H_1
\lesssim \left( H_1^{\frac12}H_2^{\frac12}+t^{-\nu}H_1\right)^\frac32  H_1\lesssim  H_2.
\end{equation*}

Then, we have
\begin{equation} \label{eq:HS}
\left|  \rho_t \int(\partial_xq)^2q^{m-1}\Theta_\sharp^{5-m}s_2\right|
\lesssim S_1.
\end{equation}
We also observe using \eqref{eq:ys} that
\[
 \left| {\mu_t} \int(\partial_xq)^2q^{m-1}\Theta_\sharp^{5-m}y_0 s_2\right|
\lesssim t^{\nu-1}\int_{y_0\geq 0} (\partial_x q)^2 y_0 s_1
+ \rho_t   \int (\partial_xq)^2 s_1
+ t^{100}\lesssim S_1+t^{100}.
\]
Moreover, by using \eqref{eq:Lz}, \eqref{eq:eT}, $\|q\|_{L^{\infty}} \lesssim 1$ and 
\eqref{eq:H1},
\begin{align*}
&\left|\int [\partial_x^2((q+\Theta_\sharp)^5-\Theta_\sharp^5)] (\partial_x q) q^{m-1}\Theta_\sharp^{5-m} \omega_2\right|
+\left| \int (\partial_x q)^2 (\partial_x ((q+\Theta_\sharp)^5-\Theta_\sharp^5)) q^{m-2}\Theta_\sharp^{5-m}\omega_2
\right|
 \\
& \lesssim \int \left(|\partial_x^2 q||\partial_x q|
+ |\partial_x q|^3 +|\partial_x \Theta_\sharp||\partial_xq|^2 
+ |\partial_x\Theta_\sharp|^2|\partial_x q| |q|
+ |\partial_x^2 \Theta_\sharp| |\partial_x q||q|\right)\omega_2 \\ 
& \lesssim H_1^{\frac12}H_2^{\frac12} 
+\left(H_1^{\frac12}H_2^{\frac12}+t^{-\nu}H_1 \right)^{\frac12}H_1+t^{-\frac1{14}}H_1+t^{-\frac3{14}}H_0^{\frac12}H_1^{\frac12}
\lesssim t^{-\frac 3{14}} H_2.
\end{align*}
Next, observe from \eqref{eq:eT} that $\left|\partial_t(\Theta_\sharp^{5-m})\right| \lesssim t^{-\frac5{14}}$ that 
\[\left| \int(\partial_xq)^2q^{m-1}\partial_t\left(\Theta_\sharp^{5-m}\right)\omega_2\right| \lesssim t^{-\frac5{14}} \int(\partial_xq)^2\omega_2 \lesssim t^{-\frac5{14}} H_1 .\]
Lastly, using \eqref{eq:eT}, we have
\[
|F_\sharp| \omega_2 \lesssim t x^{\alpha-\frac{13}2} \omega_2 \lesssim t^{1-\frac{11}{14}}\omega_2\lesssim \omega_2 ,\quad
|\partial_x F_\sharp|\omega_2 \lesssim t x^{\alpha-\frac{15}2} \omega_2 \lesssim t^{1-\frac{13}{14}}\omega_2\lesssim \omega_2.
\]
Thus,
\begin{equation*}
\left|\int (\partial_x F_\sharp) (\partial_x q) q^{m-1} \Theta_\sharp^{5-m} \omega_2 \right|
+\left|\int (\partial_x q)^2 F_\sharp q^{m-2}\Theta_\sharp^{5-m} \omega_2\right|
 \lesssim  H_1 \lesssim H_2.
\end{equation*}
Summarizing for $G'_{2,m}$, we have
\begin{equation}\label{eq:9d}
G_{2,m}'-\frac 12 S_2+3  T_{2,m}
\lesssim 
t^{100}+t^{-\frac5{14}} H_2 +S_0+ S_1 .
\end{equation}
Therefore, by \eqref{eq:8d} and \eqref{eq:9d},
\begin{align*}
G_2'+\frac 12 S_2
&\lesssim 
t^{100}+S_0+ S_1 +t^{-\frac 5{14}} H_2  
+H_2^\frac12 \begin{cases} 
t^{1+\frac {\alpha-8}{4\alpha+3}} &  \mbox{ if $1<\alpha<8$,}\\
t |\log t|^\frac12 &  \mbox{ if $\alpha=8$,}\\
t  & \mbox{ if $\alpha>8$.}
\end{cases}
\end{align*}
It follows from \eqref{eq:dG}, \eqref{eq:G1}  that there exists a positive constant
$\widetilde C_1\geq C_1$ such that setting
\[
\widetilde H_2 = G_2+ \widetilde C_1 G_1.
\]
it holds 
\[
\widetilde H_2' \lesssim t^{100}+t^{-\frac5{14}} \widetilde H_2
+ {\widetilde H_2^\frac12} \begin{cases} 
t^{1+\frac {\alpha-8}{4\alpha+3}} &  \mbox{ if $1<\alpha<8$,}\\
t |\log t|^\frac12 &  \mbox{ if $\alpha=8$,}\\
t  & \mbox{ if $\alpha>8$.}
\end{cases}
\]
Now, we obtain \eqref{eq:g1} by applying again Lemma \ref{LE:ax}
and using $\widetilde H_2 \geq G_{2,0}$.
\end{proof}

\subsection{Estimates for higher order derivatives}
Now, we estimate some higher order space derivatives of the function $q$ in the same spirit as in the previous lemma.
However, in order to avoid the use of modified energies for $k=3,4,5$, we slightly change the method for the higher-order space derivatives of $q$.
Note that the definitions of $\chi_k$ and $\varphi_k$
before Lemma \ref{LE:qq} were given so far only for $k=0,1,2$.
For $k=3,4,5$, we define
\begin{align*}
&z_3(t,x)=\frac{x-2^{-6}\rho(t)}{\mu(t)},\quad
y_3(t,x)=\frac{x-2^{-5}\rho(t)}{\mu(t)},\\
&z_4(t,x)=\frac{x-2^{-4}\rho(t)}{\mu(t)},\quad
y_4(t,x)=\frac{x-2^{-3}\rho(t)}{\mu(t)},\\
&z_5(t,x)=\frac{x-2^{-2}\rho(t)}{\mu(t)},\quad
y_5(t,x)=\frac{x-2^{-1}\rho(t)}{\mu(t)},
\end{align*}
and the weight functions
\[
\chi_k(t,x)=\chi(z_k(t,x)),\quad \varphi_k(t,x)=\varphi(y_k(t,x)),\quad
\omega_k=\chi_k \varphi_k, \quad s_k =\partial_x \omega_k.
\]
The function $\varphi$ is defined in \eqref{eq:dp} for the value of $A$ fixed at the beginning of the proof of Lemma \ref{LE:qq}.
In particular, using \eqref{eq:ph} for $k=0$, it holds for $k=3,4,5$,
\begin{equation}\label{eq:p9}
|\partial_x^3 \varphi_k|\le 2^{-9} \beta \mu^{-2} \partial_x \varphi_k.
\end{equation}

\begin{remark}
Instead of increasing the power of the weight functions as for $k=0,1,2$, we shift
the location of the cut-off function for each function.
It seems that we could not have used this second method for the first three derivatives, because of some unfavorable terms appearing after differentiating the modified energies.
In the proof of Lemma \ref{LE:qq}, 
these terms are controlled by using the Kato term $S_j$ of a lower order term (see for example \eqref{eq:GS}, \eqref{eq:HS}).
\end{remark}

\begin{lemma}\label{LE:q2}
For all $k=3,4,5$, for $t>0$ sufficiently small,
\begin{equation}\label{eq:L3}
\int (\partial_x^k q)^2 \omega_k 
 \lesssim  \begin{cases} 
t^{4 + \frac {2(\alpha-6-k)}{4\alpha+3}} &  \mbox{ if $1<\alpha<k+6$,}\\
 t^4 |\log t| &  \mbox{ if $\alpha=k+6$,}\\
t^4 & \mbox{ if $\alpha>k+6$.}
\end{cases}
\end{equation}
\end{lemma}
\begin{remark}\label{rk:22}
We see easily that Lemma \ref{LE:qq} and Lemma \ref{LE:q2} imply Proposition \ref{PR:qq}.
Indeed, for $k=0,1,2,3,4$, for $t$ sufficiently small and for $x>\frac 12 \rho(t)$, we have
\begin{align*}
|\partial_x^k q(t,x)|^2 &\lesssim \int_{\frac12\rho(t)}^{+\infty} |\partial_x^{k+1}q(t,y)||\partial_x^kq(t,y)| dy\\
&\lesssim \left(\int (\partial_x^{k+1}q)^2w_{k+1} \right)^\frac12\left(\int (\partial_x^{k}q)^2w_k \right)^\frac12 
 \lesssim \begin{cases}
t^{4+\frac{2\alpha-2k-13}{4\alpha+3}} & 1<\alpha<k+6,\\
t^{4-\frac{1}{4\alpha+3}} |\log t|^\frac12 & \alpha=k+6,\\
t^{4+\frac{\alpha - k-7 }{4\alpha+3}}  & k+6<\alpha<k+7,\\
t^4 |\log t|^\frac12 & \alpha=k+7,\\
t^4 & \alpha>k+7.
\end{cases}
\end{align*}
\end{remark}
\begin{proof}
We prove the estimates for $t>0$ small.
Let $k=3,4,5$ and $t>0$ small. Note that for all 
$x\leq 2^{2k-12}\rho$,  one has $z_k(t,x)\leq 0$ and so $\chi_k(t,x)=0$.
Moreover, for $0\leq j\leq k-1$ and for all $x\geq 2^{2k-12}\rho$, 
one has $y_j(t,x)\geq (2^{2k-12}-2^{2j-11})\frac\rho\mu>0$ and 
$z_j(t,x)\geq (2^{2k-12}-2^{2j-12})\frac{\rho}{\mu}\geq 2$ for $t>0$ small,
and so for such~$x$, $\omega_{j}(t,x)\geq \frac \pi2 A$. Note that we have taken by convention that $z_2(t,x)=z_1(t,x)=z_0(t,x)$ and $y_2(t,x)=y_1(t,x)=y_0(t,x)$, where $z_0(t,x)$ and $y_0(t,x)$ are defined in \eqref{def:z0}-\eqref{def:y0}.
Thus, for $k=3,4,5$ and $0\leq j \leq k-2$, as in Remark \ref{rk:22},
\begin{equation}\label{eq:Lg}
\begin{aligned}
\|\partial_x^j q \|_{L^\infty(x \geq 2^{2k-12}\rho(t))}^2
&\lesssim 
\left(\int_{x \geq 2^{2k-12}\rho(t)} (\partial_x^j q)^2\right)^\frac12
\left(\int_{x \geq 2^{2k-12}\rho(t)} (\partial_x^{j+1} q)^2\right)^\frac12
\lesssim H_j^\frac12 H_{j+1}^\frac12.
\end{aligned}
\end{equation}
In particular, for $k=3$ and $j=1$, from Lemma \ref{LE:qq}, it holds
\begin{equation}\label{eq:Lf}
\|\partial_x q \|_{L^\infty(x \geq 2^{-6}\rho(t))}^2
\lesssim 
\left(\int_{x \geq 2^{-6}\rho(t)} (\partial_x q)^2\right)^\frac12
\left(\int_{x \geq 2^{-6}\rho(t)} (\partial_x^2 q)^2\right)^\frac12\\
\lesssim t^{\frac {15}7}\lesssim 1.
\end{equation}
We prove \eqref{eq:L3}
for all $3\leq k \leq 5$ using an induction argument on the integer $k$.
Let
\[
G_k = \int (\partial_x^k q)^2 \omega_k 
\]
and
\[
S_k = 3  \int (\partial_x^{k+1} q)^2s_k
+2^{-8} \beta t^{-2\nu}\int (\partial_x^k q)^2 s_k
+\nu t^{1-\nu}\int_{y_0>0} (\partial_x^k q)^2 y_0s_k.
\]
We observe that setting, for $k=3,4,5$,
\[
  H_k = G_k+\sum_{l=0}^{k-1}H_{l},
\]
where the expressions of $H_0$, $H_1$ and $H_2$ are taken from the proof of Lemma \ref{LE:qq},
it holds
\[
  H_k \geq \sum_{l=0}^{k} \int (\partial_x^l q)^2 \omega_l.
\]
The exact statement to be proved by induction on $3\leq j\leq 5$ is composed of two estimates:
\begin{equation}\label{eq:Gj}
G_j' + \frac 12 S_j
\lesssim   t^{100}+S_{j-1}+t^{-\frac{5}{7}} H_j+ {  H_j^\frac12}  \begin{cases} 
t^{1+\frac {\alpha-j -6}{4\alpha+3}} &  \mbox{ if $1<\alpha<j+6$,}\\
t|\log t|^\frac12 &  \mbox{ if $\alpha=j+6$,}\\
t & \mbox{ if $\alpha>j+6$,}
\end{cases}
\end{equation}
and
\begin{equation}\label{eq:Lj}
\int (\partial_x^j q)^2 \omega_j
 \lesssim \begin{cases} 
t^{4 + \frac {2(\alpha-6-j)}{4\alpha+3}} &  \mbox{ if $1<\alpha<j+6$,}\\
 t^4 |\log t| &  \mbox{ if $\alpha=j+6$,}\\
t^4 & \mbox{ if $\alpha>j+6$.}
\end{cases}
\end{equation}
Note that \eqref{eq:Gj}-\eqref{eq:Lj} for $j=0,1,2$ were already proved in Lemma \ref{LE:qq}.
Now, we fix $3\leq k\leq 5$ and assume that \eqref{eq:Gj}-\eqref{eq:Lj} hold for all $j\leq k-1$
(which is true for $k=3$).
Our goal is to prove that \eqref{eq:Gj}-\eqref{eq:Lj} hold for $j=k$.
We compute using \eqref{eq:qq},
\begin{align*}
G_k' & =  2\int (\partial_x^{k}\partial_t q) (\partial_x^{k }q)\omega_k
+ \int (\partial_x^k q)^2 \partial_t \omega_k \\
& = - 3 \int (\partial_x^{k+1} q)^2 \partial_x \omega_k 
+ \int (\partial_x^k q)^2 \partial_t \omega_k 
+ \int (\partial_x^k q)^2 \partial_x^3 \omega_k \\
&\quad
- 2 m_0^4\int (\partial_x^{k+1} ((q+\Theta_\sharp)^5-\Theta_\sharp^5)) (\partial_x^{k} q)  \omega_k
+ 2\int (\partial_x^k F_\sharp) (\partial_x^k q)\omega_k.
\end{align*}
(As in the proof of Lemma \ref{LE:qq}, the term $R_\sharp$ does not appear
because its support and the support of the function $\omega_k$ are disjoint.)
Thus,
\begin{align*}
G_{k}' & \leq - 3\int (\partial_x^{k+1} q)^2 s_k
 - 2^{2k-11}\rho_t  \int (\partial_x^k q)^2 s_k 
 +\int (\partial_x^k q)^2 \chi_k\partial_x^3\varphi_k\\
&\quad - 2m_0^4\int (\partial_x^{k+1} ((q+\Theta_\sharp)^5-\Theta_\sharp^5)) (\partial_x^k q) \omega_k
- {\mu_t} \int (\partial_x^k q)^2 y_k s_k
+ 2\int (\partial_x^k F_\sharp) (\partial_x^k q)\omega_k
+g_{k}
\end{align*}
where we set
\begin{align*}
g_{k} & =
-2^{2k-12} \left( \rho_t-\rho \frac{\mu_t}{\mu}\right) \int (\partial_x^k q)^2 \partial_x\chi_k \varphi_k
\\
&\quad+  \int (\partial_x^k q)^2 \left[ \partial_x^3 \chi_k\varphi_k
+3 \partial_x^2 \chi_k \partial_x \varphi_k + 3 \partial_x \chi_k\partial_x^2\varphi_k\right].
\end{align*}
As in the proof of Lemma \ref{LE:qq}, we check that
$|g_{k}|\lesssim S_{k-1}$.

Then, the second and third terms on the right-hand side of the estimate of $G_k'$ are treated as the corresponding terms 
in Lemma \ref{LE:qq}, using \eqref{eq:p9},
\begin{equation*}
 -2^{2k-11} {\rho_t}  \int (\partial_x^k q)^2 s_k 
 +  \int (\partial_x^k q)^2 \chi_k\partial_x^3 \varphi_k
 \leq 
 -3\cdot 2^{2k-13} \beta t^{-2\nu} \int (\partial_x^k q)^2 s_k.
\end{equation*}

The next term is a linear combination of terms of the form
$\int (\partial_x^{k+1}(q^m \Theta_\sharp^{5-m})) (\partial_x^kq) \omega_k$
for $m\in \{1,\ldots,5\}$. 
By expanding the derivatives, we observe that
\[
\int (\partial_x^{k+1}(q^m \Theta_\sharp^{5-m})) (\partial_x^k q) \omega_k
=\sum_{l=0}^{k+1} C_{k+1}^l I_{k,l,m}
\]
where for $l=0,\ldots,k+1$,
\[
I_{k,l,m}=\int(\partial_x^{l}(q^m))(\partial_x^k q) \partial_x^{k+1-l}(\Theta_\sharp^{5-m})\omega_k.
\]
The Faà di Bruno formula gives
\[
\partial_x^{l}(q^m)
=\sum_{{\bf n}\in \Delta(m,l)}
D(m,l,{\bf n}) q^{m-|{\bf n}|} \Pi_{j=1}^l (\partial_x^j q)^{n_j}
\]
where
\[
\Delta(m,l)
=\{ {\bf n}=(n_1,\ldots,n_l) :
|{\bf n}|=\sum_{j=1}^l n_j\leq m, \sum_{j=1}^{l} j n_j = l\}
\]
and $D(m,l,{\bf n})$ are constants.

For $l=k+1$, $I_{k,k+1,m}=\sum_{j=1}^9 E_{k,m,j} I_{k,k+1,m}^j$, where (part of the first term  and the second term are obtained through integration by parts)
\begin{align*}
I_{k,k+1,m}^1&=\int (\partial_x^k q)^2 (\partial_x q) q^{m-2} \Theta_\sharp^{5-m} \omega_k,\\ 
I_{k,k+1,m}^2 &= \int (\partial_x^k q)^2  q^{m-1} \partial_x(\Theta_\sharp^{5-m} \omega_k)  ,
\\
I_{k,k+1,m}^3 &=\int (\partial_x^{k} q) (\partial_x^{k-1}q) (\partial_x^2 q) q^{m-2} \Theta_\sharp^{5-m} \omega_k,\\ 
I_{k,k+1,m}^4 &=\int (\partial_x^{k} q) (\partial_x^{k-1}q) (\partial_x q)^2 q^{m-3} \Theta_\sharp^{5-m} \omega_k,\\
I_{k,k+1,m}^5 &=\int (\partial_x^{k} q) (\partial_x^{k-2} q) (\partial_x^3 q)  q^{m-2} \Theta_\sharp^{5-m} \omega_k,\\
I_{k,k+1,m}^6 &=\int (\partial_x^{k} q) (\partial_x^{k-2} q) (\partial_x^2 q) (\partial_x q)  q^{m-3} \Theta_\sharp^{5-m} \omega_k,\\ 
I_{k,k+1,m}^7 &=\int (\partial_x^{k} q) (\partial_x^{k-2} q) (\partial_x q)^3  q^{m-4} \Theta_\sharp^{5-m} \omega_k,\\
I_{k,k+1,m}^8 &=\int (\partial_x^{k} q) (\partial_x^{k-3} q) (\partial_x^2 q)^2 q^{m-3} \Theta_\sharp^{5-m} \omega_k,\\
I_{k,k+1,m}^9 &=\int (\partial_x^{k} q) (\partial_x^{k-3} q) (\partial_x q)^4 q^{m-5} \Theta_\sharp^{5-m} \omega_k,
\end{align*}
and $E_{k,m,j}$ are constants (note that for $k=3,4$, some terms are redundant).
In the above terms, we use the convention that terms containing negative powers of 
$q$ do not appear.
We have, using \eqref{eq:Lf},
\[
|I_{k,k+1,m}^1 |\lesssim H_k
\|\partial_x q\|_{L^\infty(x > 2^{-7} \rho)} \lesssim H_k.
\]
We estimate $I_{k,k+1,m}^2 $ as follows,
\[
\left|\int (\partial_x^k q)^2 q^{m-1}\partial_x(\Theta_\sharp^{5-m})\omega_k\right|\lesssim t^{-\frac1{14}} H_{k},
\]
and
\[
\left|\int (\partial_x^k q)^2 q^{m-1} \Theta_\sharp^{5-m}\partial_x \omega_k\right|\lesssim
 \int(\partial_x^kq)^2s_k\lesssim t^{2\nu} S_k.
\]
Similarly as \eqref{eq:Lz}, using $|\partial_x (\omega_{j+1}^\frac12\omega_j^\frac12)|
\lesssim t^{-\nu} \omega_j$, for $j=2,3,4$
(from the definitions of $\omega_j$ and $\omega_{j+1}$),
we claim that
\begin{align}\label{eq:zb}
\bigr\|(\partial_x^j q)^2 \omega_{j+1}^\frac12\omega_j^\frac12\bigr\|_{L^\infty} &\lesssim
\left(\int (\partial_x^{j+1} q)^2 \omega_{j+1} \right)^{\frac12}
\left(\int (\partial_x^j q)^2 \omega_j \right)^{\frac12} 
+t^{-\nu} \int (\partial_x^jq)^2  \omega_j \nonumber \\ &
\lesssim  H_{j+1}^\frac12 H_j^\frac12 + t^{-\nu} H_j.
\end{align}
In particular, using $ t^{-\nu} H_2\lesssim 1$, and $H_{k-1}\lesssim 1$ by the induction hypothesis,
for $k=3,4,5$,
\[
|I_{k,k+1,m}^3|\lesssim H_k^\frac12H_{k-1}^\frac12 (H_3^\frac14 H_2^\frac14 + t^{-\frac\nu2} H_2^\frac12)
\lesssim H_k.
\]
Moreover, by \eqref{eq:Lf},
\[
|I_{k,k+1,m}^4|\lesssim H_k^\frac12 H_{k-1}^\frac12 \|\partial_x q\|_{L^\infty(x>2^{-7}\rho)}^2
\lesssim H_k.
\]
For $k=3$, $I^5_{k,k+1,m}$ is identical to $I^1_{k,k+1,m}$ and for $k=4$, it is identical
to $I^3_{k,k+1,m}$. For $k=5$, we have
\[
|I^5_{5,6,m}|\lesssim  \|(\partial_x q_3)^2 \omega_4^\frac12\omega_3^\frac12\|_{L^\infty}^\frac 12  H_5^\frac12H_3^\frac 12 \lesssim H_5.
\]
Arguing similarly, we see easily that
\[
|I_{k,k+1,m}^5|+|I_{k,k+1,m}^6|+|I_{k,k+1,m}^7|+|I_{k,k+1,m}^8|\lesssim H_k.
\]
In conclusion for $I_{k,k+1,m}$, we have proved, for $m=1,\ldots,5$,
\[
|I_{k,k+1,m}|\lesssim t^{-\frac1{14}} H_k+t^{2\nu} S_k.
\]
Now, we estimate $I_{k,k,m}$, \emph{i.e.} the case where $l=k$, observing that integrating by parts,
\begin{align*}
I_{k,k,m}&=
\int\partial_x^k(q^m)(\partial_x^k q)\partial_x(\Theta_\sharp^{5-m})\omega_k\\
&=-\int[\partial_x^{k+1}(q^m)](\partial_x^k q)  \Theta_\sharp^{5-m} \omega_k
-\int[\partial_x^k(q^m)](\partial_x^{k+1} q)  \Theta_\sharp^{5-m} \omega_k\\
&\quad -\int\partial_x^k(q^m)(\partial_x^k q) \Theta_\sharp^{5-m} \partial_x \omega_k
\end{align*}
Each of term was either already controlled
(for example, the first term is $-I_{k,k+1,m}$), or is controlled in a similar way
by expanding the derivative of $q^m$.
We obtain as for $I_{k,k+1,m}$,
\[
|I_{k,k,m}|\lesssim t^{-\frac1{14}} H_k+t^{2\nu} S_k.
\]

For $l=0,\ldots,k-1$, we have, using \eqref{eq:Lg} and the induction assumption,
\begin{align*}
|I_{k,l,m}|& \lesssim \int |\partial_x^l(q^m)| |\partial_x^k q|  |\partial_x^{k+1-l}(\Theta_\sharp^{5-m})|\omega_k\\
& \lesssim
 \sup_{0\leq j\leq k-2} \|(\partial_x^j q) \omega_k \|_{L^\infty(x\geq 2^{2k-12}\rho)}^{m} 
\int  |\partial_x^k q|  |\partial_x^{k+1-l}(\Theta_\sharp^{5-m})|\omega_k
\\
&\quad +\|\partial_x^2(\Theta_\sharp^{5-m})\|_{L^\infty(x> 2^{-7} \rho)}\|q\|_{L^\infty}^{m-1}
\int |\partial_x^{k-1} q| |\partial_x^k q|  \omega_k
\\
& \lesssim H_{k-1}^\frac 12 H_k^\frac 12 \sum_{j=2}^{k+1}\left( \int_{x>2^{-7}\rho} |\partial_x^{j} (\Theta_\sharp^{5-m})|^2  \right)^{\frac 12}
+ H_{k-1}^\frac12 H_k^\frac12\|\partial_x^2(\Theta_\sharp^{5-m})\|_{L^\infty(x> 2^{-7} \rho)} .
\end{align*}
Here, $m=1,\ldots,4$, since for $m=5$ the derivative of $\Theta_0^{5-m}$ is zero.
For the first term on the right-hand side, using \eqref{eq:eT}, we have, 
for $j=2,\ldots,k+1$, and $0\leq x\leq \delta$,
\[
|\partial_x^{j} (\Theta_\sharp^{5-m})|^2
\lesssim x^{2\alpha-1-2(k+1)}\lesssim x^{2\alpha-13}\lesssim x^{-11}.
\]
and so for $j=2,\ldots,6$,
\[
\int_{x>2^{-7}\rho} |\partial_x^{j} (\Theta_\sharp^{5-m})|^2 dx
\lesssim t^{-\frac{10}{4\alpha+3}}\lesssim t^{-\frac{10}{7}}.
\]
For the second term on the right-hand side, using \eqref{eq:eT}, we have,
for  $2^{-7}\rho\leq x\leq \delta$,
\[
|\partial_x^2 (\Theta_\sharp^{5-m})|^2
\lesssim 1+t^{\frac{2\alpha-5}{4\alpha+3}}\lesssim t^{-\frac37}.
\]
Thus, for $l=0,\ldots, k-1$,
\begin{equation*}
|I_{k,l,m}|
 \lesssim t^{-\frac57} H_k^\frac 12 H_{k-1}^\frac12 \lesssim t^{-\frac57} H_k.
\end{equation*}

Then, as in the proof of Lemma \ref{LE:qq},
\[
-\mu_t  \int (\partial_x^k q)^2 y_k s_k
\leq - \nu t^{1-\nu} \int_{y>0}(\partial_x^kq)^2y_k s_k 
+ 2^{-10} \rho_t \int (\partial_x^k q)^2 s_k+Ct^{100}.
\]
Finally, by the Cauchy-Schwarz inequality,
\[
\left| \int (\partial_x^k F_\sharp) (\partial_x^k q)\omega_k \right| \leq H_k^\frac 12 \left( \int (\partial_x^k F_\sharp)^2 \omega_k\right)^{\frac 12}.
\]
Using \eqref{eq:eT} and \eqref{eq:Fd}, for $0<x<2\delta$,
\[
(\partial_x^k F_\sharp)^2(x)\lesssim t^2 x^{2\alpha-13-2k}.
\]
Thus, using also that $x\leq 2^{-8}\rho$ implies $\omega_k=0$, 
it follows that
\[
\int (\partial_x^k F_\sharp)^2 \omega_k \lesssim
t^2 \int_{2^{-8}t^\beta}^\delta x^{2\alpha-13-2k} dx \lesssim
\begin{cases} 
t^{2+\frac {2(\alpha-6-k)}{4\alpha+3}} &  \mbox{ if $1<\alpha<6+k$,}\\
t^2 |\log t| &  \mbox{ if $\alpha=6+k$,}\\
t^2  & \mbox{ if $\alpha>6+k$.}
\end{cases}
\]

Gathering all the estimates above, we have proved \eqref{eq:Gj}
for $l=k$.
Moreover, it follows from \eqref{eq:dG}, \eqref{eq:G1} and \eqref{eq:Gj} applied
for all $l=2,\ldots,k$, that there exist positive constants
$\widetilde C_{k,0}\geq 1$, $\widetilde C_{k,1}\geq 1$, \ldots, $\widetilde C_{k,k-1}\geq 1$ such that setting
\[
\widetilde H_k = G_k+\sum_{l=0}^{k-1}\widetilde C_{k,k-l} G_{k-l},
\]
we obtain, for $t>0$ small,
\[
\widetilde H_k'+\frac 12 S_k \lesssim t^{100} + t^{-\frac{5}7} \widetilde H_k+ \widetilde H_k^\frac12 \begin{cases} 
t^{1+\frac {\alpha-6-k}{4\alpha+3}} &  \mbox{ if $1<\alpha<6+k$,}\\
t|\log t|^\frac12 &  \mbox{ if $\alpha=6+k$,}\\
t  & \mbox{ if $\alpha>6+k$.}
\end{cases}
\]
and we obtain \eqref{eq:Lj} for $j=k$ by applying Lemma \ref{LE:ax}
to the function $\widetilde H_k$.
\end{proof}

\section{The blowup profile}\label{S:3}

\subsection{The linearized operator}

We recall standard properties of the operator $\cL$ and introduce useful functions for the construction of the blow-up profile.

\begin{lemma}\label{le:li}
The self-adjoint operator $\cL$ on $L^2(\mathbb R)$ defined by~\eqref{eq:cL} satisfies the following properties.
\begin{enumerate} 
\item[(i)] \emph{Spectrum of $\cL$.} The operator $\cL$ has only one negative eigenvalue $-8$ associated to the eigenfunction $Q^3$.
Moreover, $\ker \cL=\{aQ' : a \in \RR\}$ and $\sigma_\textnormal{ess}( \cL)=[1,+\infty)$.
\item[(ii)] \emph{Scaling.} The operator $\Lambda$ being defined in~\eqref{eq:La}, it holds $\cL\Lambda Q=-2Q$ and $(Q,\Lambda Q)=0$.
\item[(iii)] \emph{Coercivity of $\cL$.}
There exists $\nu_0>0$ such that, for all $\phi \in H^1(\RR)$,
\begin{equation} \label{eq:co}
(\cL\phi,\phi) \ge \nu_0\|\phi\|_{H^1}^2-\frac1{\nu_0}\left((\phi,Q)^2+(\phi,y\Lambda Q)^2 +(\phi,\Lambda Q)^2 \right) .
\end{equation}
\item[(iv)] \emph{Inversion of $\cL$ in $\cY$.} For any function $h \in \cY$ such that $\int hQ'=0$,
there exists a unique function $f \in \cY$ orthogonal to $Q'$ and such that $\cL f=h$; moreover, if $h$ is even (resp. odd) then $f$ is even (resp. odd). 
\item[(v)] \emph{Inversion of $\cL$ in $\cZ_k^-$.}
Let $k\geq 0$. For any function $h \in \cZ_k^-$ such that $\int h Q'=0$,
there exists a unique function $f \in \cZ_{k+1}^-$ orthogonal to $Q'$ and such that $\cL f=h$.
\end{enumerate}
\end{lemma}
\begin{proof}
The points (i)--(iv) are standard, see for example \cite[Lemma 2.1]{MMR1}.
For (v), we refer to \cite[Lemma 5]{CoM1} and \cite[Section 2.3]{MP24}.
\end{proof}

\begin{lemma} \label{le:PA}
\begin{enumerate}
\item[(i)] There exists a unique even function $R \in \cY$ such that 
$\cL R=5 Q^4$.
\item[(ii)] There exists a unique function $P \in \cZ_0^-$ such that 
\begin{equation} \label{eq:PP}
(\cL P)'=\Lambda Q, \quad \lim_{y \to - \infty}P(y)= 2m_0, \quad \lim_{y \to +\infty} P(y)=0,\quad (P,Q')=0.
\end{equation}
Moreover,
\begin{equation} \label{eq:PQ} 
(P,Q)=m_0^2 .
\end{equation}
\item[(iii)] Let $A_1=P-m_0(1-R)$.
Then, $A_1\in \cZ_0$ is an odd function of class $\cC^\infty$ satisfying
\begin{equation} \label{eq:A1}
A_1' \in \cY,\quad (\cL A_1)'=\Lambda Q \quad \mbox{and}\quad \lim_{y \to \mp \infty}A_1(y)=\pm m_0.
\end{equation}
Moreover,
\begin{equation} \label{eq:AQ} 
(A_1,Q)=0 ,\quad (A_1,Q')=0.
\end{equation}
\item[(iv)] The following identity holds
\begin{equation} \label{eq:id}
\left( -\frac12P+yP' +20(Q^3 A_1P)' ,Q\right)=m_0^2 .
\end{equation}
\begin{equation} \label{eq:pq}
\left( P, Q^5 \right)=-m_0^2
\end{equation}
\end{enumerate}
\end{lemma}
\begin{proof}
The proofs of (i), (ii) and (iii) are given in \cite{MP20}.
Now, we prove the identities in (iv). We compute the scalar product of $(\mathcal{L}{A}_1)'$ with $P$.  On the one hand, we have by integrating by parts and using the limits in \eqref{eq:PP} and \eqref{eq:A1}
\begin{align*}
\left( (\mathcal{L}{A}_1)',P\right)&=\left( (\mathcal{L}{A}_1'),P\right)-20\left( Q^3Q'{A}_1,P\right) \\ 
&=\left( {A}_1',\mathcal{L}P\right)+20\left( (Q^3{A}_1P)',Q\right)\\ 
&={A}_1P\Big|_{-\infty}^{+\infty}-\left( {A}_1,(\mathcal{L}P)'\right)+20\left( (Q^3{A}_1P)',Q\right) \\ &
=-2m_0^2+20\left( (Q^3{A}_1P)',Q\right) ,
\end{align*}
where we used in the last step that $\left( {A}_1,(\mathcal{L}P)'\right)=\left( {A}_1,\Lambda Q\right)=0$ since
the function ${A}_1\Lambda Q$ is odd.
 On the other hand, we see from  \eqref{eq:A1}  that 
\begin{align*}
\left( (\mathcal{L}{A}_1)',P\right)=\left( \Lambda Q,P\right)=-\left( Q,\Lambda P\right)=-\frac12\left( Q,P\right)-\left( Q,yP'\right)
\end{align*}
Therefore, we conclude the proof of \eqref{eq:id} by combining these identities and using  $\left(P,Q\right)=m_0^2$ (see \eqref{eq:PQ}).

We observe by using the equation in \eqref{eq:PP} that 
\begin{equation*} 
\left( (\mathcal{L}P)', \int_y^{+\infty}Q \right)=\left(  \Lambda Q, \int_y^{+\infty}Q \right)=-\frac12 \left(   Q, \int_y^{+\infty}Q \right)=-\frac14 \left( \int_{-\infty}^{+\infty}Q \right)^2
\end{equation*}
On the other hand, we have by integrating by parts, and then using \eqref{eq:Q5}, \eqref{eq:m0} and \eqref{eq:PP},
\begin{align*} 
\left( (\mathcal{L}P)', \int_y^{+\infty}Q \right)&=-\left( P'', Q \right)+\left( P', \int_y^{+\infty}Q  \right)-\left( 5Q^4P, Q \right) \\ 
&=-\left( P, Q'' \right)-2m_0 \left(\int_{-\infty}^{+\infty}Q\right)+\left( P, Q \right)-5\left( Q^5, P \right) \\ 
&=-\frac12\left(\int_{-\infty}^{+\infty}Q\right)^2-4\left( Q^5, P \right) .
\end{align*}
Thus, we infer combing these identities that 
\begin{equation*}
\left( Q^5, P \right)=-\frac1{16}\left(\int_{-\infty}^{+\infty}Q\right)^2=-m_0^2 ,
\end{equation*}
which is \eqref{eq:pq}.

\end{proof}

\subsection{Bootstrap estimates on the parameters}
Define
\begin{equation}\label{eq:ka}
\kappa_\tau= (4\alpha+3)^{-1} (2\alpha)^{-\frac{4\alpha+3}{2\alpha}},\quad
\kappa_\sigma=(2\alpha)^{-\frac1{2\alpha}},\quad
\kappa_\lambda=(2\alpha)^{-\frac{2\alpha+1}{2\alpha}}.
\end{equation}
Let $I$ be an interval of $(-\infty,-s_0]$
where $s_0\geq 1$ is large.
\begin{remark}
 The fact that we work for negative rescaled time $s$ is related to the method used for constructing the blowup solution, which is backward in time $t$.
Since we argue from $t=0$ to some positive time $t_0>0$, and since 
$|s|\to+\infty$ is related to the blowup time $t\downarrow0$, to keep the same sense of time,
we  consider large negative $s$.
\end{remark}

Let $C^1$ functions $\tau$, $\lambda$ and $\sigma$ defined on $I$,
where $\tau>0$, $\lambda>0$, and satisfying the following
estimates, for a bootstrap constant $C_1^\star>0$ to be fixed:
\begin{align}
 \left|\tau(s)-\kappa_\tau |s|^{-\frac{4\alpha+3}{2\alpha}}\right|&\leq 
C_1^\star |s|^{-\frac{4\alpha+3}{2\alpha}-1}\log|s|,\label{eq:B1}
\\
 \left|\sigma(s)-\kappa_\sigma  |s|^{-\frac1{2\alpha}}\right|&\leq C_1^\star |s|^{-\frac1{2\alpha}-1}\log|s|,\label{eq:B2}\\
 \left|\lambda(s)-\kappa_\lambda  |s|^{-\frac{2\alpha+1}{2\alpha}}\right|&\leq C_1^\star |s|^{-\frac{2\alpha+1}{2\alpha}-1}\log|s|\label{eq:B3}
\end{align}
We observe that \eqref{eq:B1}--\eqref{eq:B3} imply
\begin{equation} \label{eq:B4}
\left| \lambda^\frac12\sigma^{\alpha-\frac12} - (2\alpha)^{-1} |s|^{-1}\right|
+\left| \frac\lambda\sigma - (2\alpha)^{-1} |s|^{-1}\right|
+\left| \sigma^{2\alpha} - (2\alpha)^{-1} |s|^{-1}\right|
\lesssim C_1^\star |s|^{-2}\log|s| 
\end{equation}
and 
\begin{equation} \label{eq:B4.1}
\left| \tau \sigma^{-3}-(4\alpha+3)^{-1}(2\alpha)^{-2}|s|^{-2}\right| \lesssim C_1^\star|s|^{-3}\log|s| .
\end{equation}
We will take $s_0$ sufficiently large, possibly
depending on $C_1^\star$. 

\subsection{Definition of the blowup profile}
Let $T\in (0,1)$ and $S\in I$. 
For $s\in I$, let
\begin{equation}\label{eq:ta}
\tau(s) = T + \int_S^s \lambda^3(s') ds'.
\end{equation}
The function $\Theta(t,x)$ being defined in Section \ref{S:2} (see \eqref{eq:T0}, \eqref{eq:TT} and Proposition \ref{PR:qq}), we set
\begin{equation} \label{def:theta}
\theta(s,y)= \lambda^\frac12(s)\Theta(\tau(s),\lambda(s)y+\sigma(s)).
\end{equation}
We also defined a cut-off function
\begin{equation}\label{def:chiL}
\ZL(s,y)=\CL\left(|s|^{-1}y\right).
\end{equation}
Let $A_1\in \cZ_0$ be defined in Lemma \ref{le:PA} and
Let $A_2,A_2^*,A_3:\RR\to\RR$ to be chosen later with
\begin{equation}\label{eq:aa}
A_1\in \cZ_0,\quad A_2\in\cY,\quad A_2^*\in\cY, \quad A_3\in\cZ_1^-
\end{equation}
to be defined later. Define
\begin{equation}\label{eq:v9}
V_0= Q \ZL,\quad
V_1=c_1 A_1 \ZL,\quad V_2=(c_2 P + A_2) \ZL,\quad
V_2^*= A_2^* \ZL,\quad V_3=A_3 \ZL,
\end{equation}
and
\begin{equation}\label{eq:WW}
V=V_1\theta+V_2\theta^2+V_2^*\partial_y\theta+V_3\theta^3,\quad 
W=V_0+V,
\end{equation}
where $c_1=-(2\alpha+1)$ and where the constant $c_2$ is to be chosen later.
Let
\begin{align}
\widetilde\cL & = -\partial_x^2 + 1 - 5 V_0^4,\label{eq:Lt}\\
\beta(s )& =c_1 \lambda^\frac12(s)\sigma^{\alpha-\frac12}(s)
+c_2 \lambda(s)\sigma^{2\alpha-1}(s).\label{eq:bt}
\end{align}
\begin{remark}\label{rk:ap}
We comment on the terms composing the approximate rescaled solution $W$.
The first term $V_1$ is the main term, which will actually produce the desired blowup rate.
This term will also produce a nontrivial blowup residue.
The other terms in $V$ are used to remove error terms due to the presence of $V_1$, up to a certain order of $|s|^{-1}$,
which is needed to obtain a sufficiently small error term and to close the estimates.
Theoretically, it would be possible to push the expansion of $V$ to any arbitrary order of $|s|^{-1}$,
to provide a better approximate solution, at least locally around the soliton.

Note that the functions $A_2$, $A_2^*$, $A_3$ have exponential decay on the right.
Only $A_1$ behaves as a nonzero constant on the right. This will allow us to obtain an error term with
exponential decay on the right. This remarkable property of the ansatz is important to close the estimates
in the construction proof. See \eqref{eq:PW} below.

Note that the terms $V_1\theta$, $V_2\theta^2$ and $V_3\theta^3$ formally have the
decay in $s$ respectively $|s|^{-1}$, $|s|^{-2}$ and $|s|^{-3}$ for $s$ large.
The expression $V$ is thus formally an asymptotic decomposition of the approximate solution in powers of $s^{-1}$. Moreover, the term $\partial_y \theta$ is formally of the size $|s|^{-2}$.
However, we choose not merge it with the term $V_2\theta^2$ of the same size in $|s|^{-1}$.
Indeed, this term appears intrinsically and more precise asymptotic computations for the
parameters $\lambda$ and $\sigma$ would be needed otherwise. See also Remark \ref{rk:ag}.
\end{remark}

We define the following cut-off functions
\begin{align*}
    \eta(s,y)&=
    \ONE_{[-1,+\infty)}(|s|^{-1}y),\\
    \eta_L(s,y)&=\ONE_{[-1,-\frac12]}(|s|^{-1}y),\\
    \eta_{I}(s,y)&=
    \ONE_{[-1,0]}(|s|^{-1}y),\\
     \eta_R(s,y)&=
    \ONE_{[1,\delta |s|^{-1} \lambda^{-1]}}(|s|^{-1} y)
    =\ONE_{[1,+\infty)}(|s|^{-1} y) \cdot \ONE_{(-\infty,1]}(\delta^{-1} \lambda y).
\end{align*}
Let
\begin{equation}\label{eq:EW}
\mathcal{E} (W) = 
\partial_s W + \partial_y (\partial_y^2 W - W +W^5)
-\frac{\lambda_s}{\lambda}\Lambda W - \left(\frac{\sigma_s}{\lambda} - 1\right)\partial_y W.
\end{equation}
Estimating $\cE(W)$ is equivalent to measuring to what extend $W$ is an accurate approximate solution
of the rescaled gKdV equation
\begin{equation}\label{rescaled}
\partial_s w + \partial_y (\partial_y^2 w - w +w^5)
-\frac{\lambda_s}{\lambda}\Lambda w - \left(\frac{\sigma_s}{\lambda} - 1\right)\partial_y w=0.
\end{equation}
In the following proposition, we give basic estimates on $W$ and we estimate $\cE(W)$.

\begin{proposition}\label{pr:bp}
Assume \eqref{eq:B1}--\eqref{eq:B3}.
There exist $c_2$ and functions $A_2,A_2^*,A_3$  as in \eqref{eq:aa} such that
the following holds.
\begin{enumerate}
\item[(i)] \emph{Pointwise estimates.} For $p=0,1,2,3$,
\begin{equation}\label{eq:w0}
|\partial_y^p W| 
\lesssim 
\omega\eta+|s|^{-1-p}\eta +\lambda^{\frac12+p} \eta_R
\end{equation}
\item[(ii)] \emph{Error of $W$ for the rescaled equation.}
\begin{equation} \label{eq:ew}
\cE(W)=-\left(\frac{\lambda_s}{\lambda}+\beta\right) \left( \Lambda V_0+\Psi_{\lambda} \right)
- \left( \frac{\sigma_s}{\lambda} - 1 \right) (\partial_y V_0+\Psi_{\sigma})+ \Psi_{W}
\end{equation}
where for any $p=0,1$, for any $y\in \RR$,
\begin{align}
|\partial_y^p \Psi_\lambda|
&\lesssim |s|^{-1}\omega\eta + |s|^{-1-p}\eta_L
+|s|^{-2-p}\eta_I,\label{eq:ls}\\
|\partial_y^p \Psi_\sigma|&\lesssim |s|^{-1}\omega\eta + |s|^{-2-p}\eta_L
+|s|^{-3-p}\eta_I , \label{eq:ss} \\ 
|\partial_y^p \Psi_W|&\lesssim  C_1^\star|s|^{-4}(\log|s|) (1+y^4) \omega\eta+|s|^{-4}(1+|y|)\eta_I
+C_1^\star|s|^{-2}(\log|s|)\eta_L. 
\label{eq:PW}
\end{align} 
Moreover,
\begin{equation}\label{est:Psi_lambda:Q}
|(\Psi_\lambda,Q)|\lesssim C_1^\star |s|^{-2}\log|s|.
\end{equation}
\item[(iii)] \emph{Norms of $W$.}
\begin{align} 
\left| \int W^2 - \int Q^2 \right|& \lesssim \delta^{2\alpha},\label{eq:ma}\\
\left| \int (\partial_y W)^2 - \int (Q')^2 \right| &\lesssim |s|^{-2}.\label{eq:h1}
\end{align}
\item[(iv)] \emph{Variation of the energy of $W$.}
\begin{equation} \label{eq:en}
\left|\frac{d}{ds}\left[\frac {E(W)}{\lambda^2}\right]\right|
\lesssimD \frac 1{\lambda^2} \left( |s|^{-1}\left| \frac{\lambda_s}{\lambda}+\beta \right|
+ |s|^{-1}\left| \frac{\sigma_s}{\lambda}-1\right|+ C_1^\star|s|^{-4}\log |s| \right).
\end{equation}

\end{enumerate}
\end{proposition}
Subsections \ref{s:3.4}--\ref{s:3.9} are devoted to the proof of Proposition \ref{pr:bp}.

\subsection{Estimates on the blowup profile}\label{s:3.4}

We prove estimates on the components of the blowup profile $W$,
which will imply (i) of Proposition \ref{pr:bp}
and which will also be used to estimate $\cE(W)$.

We start by estimates on the functions $V_j$.
\begin{lemma}\label{le:vv}
For all $(s,y)\in I\times \RR$, it holds
\begin{align*}
|\partial_y^mV_0|
+|\partial_y^m\Lambda V_0|
+
|\partial_y^mV_2^*|
+|\partial_y^m\Lambda_3 V_2^*|&\lesssim 
\omega\eta  \quad \mbox{for $m\geq 0$,}\\
|V_1|&\lesssim \eta,\\
|\Lambda_1V_1|&\lesssim \omega\eta+\eta_L,\\
|\partial_y^mV_1|+|\partial_y^m \Lambda_1V_1|
&\lesssim 
\omega\eta + |s|^{-m} \eta_L\quad  \mbox{for $m\geq 1$,}\\
|\partial_y^mV_2|+|\partial_y^m\Lambda_2 V_2|+|\partial_y^m\Lambda_3 V_3|
&\lesssim 
\begin{cases}
\omega \eta+ \eta_I & \mbox{for $m=0$,}\\
\omega\eta + |s|^{-m} \eta_L & \mbox{for $m\geq1$,}
\end{cases}\\
|\partial_y^mV_3|&\lesssim 
\begin{cases}
\omega\eta+(1+|y|)\eta_I & \mbox{for $m=0$,}\\
\omega\eta +\eta_I & \mbox{for $m=1$,}
\\
\omega\eta + |s|^{-m+1}\eta_L & \mbox{for $m\geq 2$.}
\end{cases} 
\end{align*}
\end{lemma}
\begin{proof}
The above estimates follow directly from the definition of the functions 
$V_0$, $V_1$, $V_2$, $V_2^*$ and $V_3$ in \eqref{eq:aa}--\eqref{eq:v9} and the definitions 
of the spaces $\cY$, $\cZ_0$ and $\cZ_k^-$.
For $\Lambda_1 V_1$, we have used the fact that by \eqref{eq:zz}, $\Lambda_1 A_1\in \cY$ and $\Lambda_3 A_3\in\cZ_0^-$.
\end{proof}

We claim the following global estimate on $\theta$.
\begin{lemma}\label{le:th}
The function $\theta$ satisfies $\theta \in \cC(I,H^1(\mathbb R))$ and, for all $s \in I$,
\begin{equation} \label{theta:H1}
\|\theta(s)\|_{L^2} \lesssim \delta^{\alpha} \quad \text{and} 
\quad \|\partial_y\theta(s)\|_{L^2} \lesssim \lambda(s) \delta^{\alpha-1} .
\end{equation}
Moreover, for any $k=0,1,2,3,4$ and for all $y\geq -|s|$, the following pointwise estimate holds
\begin{equation}\label{eq:lt}
  |\partial_y^k\theta(s,y)|\lesssim |s|^{-1-k}\eta(s,y) + \lambda^{\frac12+k}(s)\eta_R(s,y) .
\end{equation}
\end{lemma}
\begin{proof}
Recall from \eqref{def:theta} that 
$\theta(s,y)= \lambda^\frac12(s)\Theta(\tau(s),\lambda(s)y+\sigma(s))$.
The bounds in \eqref{theta:H1} follow directly from the definition of $\theta$, the conservation of the $L^2$-norm and the $H^1$-bound on the function $\Theta$ in \eqref{eq:T9}.

To prove \eqref{eq:lt}, we start with some preliminary estimates.
Since $\Theta = \Theta_\sharp+q$ (see \eqref{eq:Td} for the definition of $\Theta_\sharp$ 
and \eqref{def:q} for the definition of $q$), and since the function $q(t,x)$ is estimated in Proposition \ref{PR:qq} only for $x>\frac 12\rho(t)$, we need to relate
the condition $y\geq -|s|$ to the condition $x=\lambda y + \sigma \geq \frac12 \rho$.
Because of the presence of the cut-off $\chi(t^{-\frac76\beta})$ in the definition of
$\Theta_\sharp$, we also need a lower bound on $t^{-\frac76 \beta}x$.
For $y\geq - |s|$, we estimate using \eqref{eq:B1}, \eqref{eq:B2}, \eqref{eq:B3},
\begin{align*}
\lambda y + \sigma -\frac12\rho\circ \tau
& \geq -\lambda |s|  + \sigma-\frac 12 |\tau|^{\frac1{4\alpha+3}}\\
& \geq \left(\kappa_\sigma-\kappa_\lambda-\frac12 \kappa_\tau^\frac1{4\alpha+3}\right)
|s|^{-\frac 1{2\alpha}} 
-C C_1^\star  |s|^{-\frac1{2\alpha}-1}\log|s| 
\end{align*}
for some constant $C>0$.
We compute using \eqref{eq:ka},
\[
\kappa_\sigma-\kappa_\lambda-\frac12 \kappa_\tau^\frac1{4\alpha+3}
=(2\alpha)^{-\frac1{2\alpha}} \left( 1-(2\alpha)^{-1}
- \frac 12 (4\alpha+3)^{-\frac{1}{4\alpha+3}}\right) >0
\]
for $\alpha>1$.
Thus, taking $s_0$ sufficiently large, we obtain for all $y\geq - |s|$,
\begin{equation} \label{est:lambda y}
\lambda y + \sigma \geq \frac12\rho\circ \tau
\geq \frac 12 \kappa_\tau^{\frac1{4\alpha+3}} |s|^{-\frac1{2\alpha}} 
- C_1^\star C |s|^{-\frac1{2\alpha}-1}\log |s|
\gtrsim  |s|^{-\frac1{2\alpha}}
\end{equation}
and
\begin{equation}\label{est:cutoff}
\tau^{-\frac76\beta}(\lambda y+\sigma) \gtrsim |s|^{\frac{7}{12\alpha}}|s|^{-\frac1{2\alpha}} =c|s|^{\frac{1}{12\alpha}} \ge 1.
\end{equation}
Note that estimate \eqref{est:cutoff} means that in the expression of
$\Theta_\sharp$, the term $\chi(t^{-\frac76\beta} x)$, where $t=\tau(s)$ and
$x=\lambda(s)y+\sigma(s)$, is identically equal to $1$ in the region $y>-|s|$.
From the definition of $\theta$ and omitting the $s$ variable 
for the functions $\tau$, $\lambda$, $\sigma$, we have 
for $k=0,1,2,3,4$, and for all $y>-|s|$,
\begin{equation} \label{decomp:dktheta}
\partial_y^k\theta(s,y)
= \lambda^{\frac12+k}\Theta_0^{(k)}(\lambda y + \sigma)+\lambda^{\frac12+k}\tau \Theta_1^{(k)}(\lambda y + \sigma)
+ \lambda^{\frac12+k} \partial_x^k q(\tau,\lambda y +\sigma),
\end{equation}
where have used the previous observation concerning the term $\chi(t^{-\frac76\beta} x)$.

Now, we estimate separately each term in \eqref{decomp:dktheta}, starting with the last one.
We claim that
for $k=0,1,2,3,4$, for $y\geq -|s|$,
\begin{equation}\label{eq:qQ}
\lambda^{\frac 12 + k} |\partial_x^k q(\tau, \lambda y +\sigma)|
\lesssim |s|^{- k -\frac 92}.
\end{equation}
Indeed, using Proposition \ref{PR:qq} with \eqref{est:lambda y}, we have the following bounds
\begin{itemize}
\item For $1<\alpha<k+6$,
\[ \lambda^{\frac 12 + k} |\partial_x^k q(\tau, \lambda y +\sigma)|
\lesssim \lambda^{\frac12+k}\tau^{\frac{18\alpha-2k-1}{2(4\alpha+3)}} \lesssim |s|^{-5-k} .\]
\item For $ \alpha=k+6$,
\[
\lambda^{\frac 12 + k} |\partial_x^k q(\tau, \lambda y +\sigma)|
\lesssim \lambda^{\frac12+k}\tau^{\frac{16\alpha+11}{2(4\alpha+3)}} |\log \tau|^\frac14
\lesssim |s|^{-5-k}\log |s|.
\]
\item For $k+6<\alpha<k+7$,
\[
\lambda^{\frac 12 + k} |\partial_x^k q(\tau, \lambda y +\sigma)|
\lesssim \lambda^{\frac12+k}\tau^{\frac{17\alpha - k+5 }{2(4\alpha+3)}}
\lesssim |s|^{-\frac{19}4-k-\frac{6+k}{4\alpha}}.
\]
\item For $\alpha=k+7$,
\[
\lambda^{\frac 12 + k} |\partial_x^k q(\tau, \lambda y +\sigma)|
\lesssim \lambda^{\frac12+k}\tau^2 |\log\tau|^\frac14 
\lesssim |s|^{-\frac92-k-\frac1{2\alpha}(\frac{13}2+k)}\log|s|.
\]
\item For $\alpha>k+7$,
\[
\lambda^{\frac 12 + k} |\partial_x^k q(\tau, \lambda y +\sigma)|
\lesssim \lambda^{\frac12+k}\tau^2
\lesssim |s|^{-\frac92-k-\frac1{2\alpha}(\frac{13}2+k)}.
\]
\end{itemize}
We observe that \eqref{eq:qQ} holds in all the cases above.

Next, for the first term in \eqref{decomp:dktheta}, we claim that
\begin{equation}\label{eq:dk1}
\lambda^{\frac12+k} |\Theta_0^{(k)}(\lambda y +\sigma)|
\lesssim |s|^{-1-k}\eta+ \lambda^{\frac12+k} \delta^{\alpha-\frac12-k}\eta_{R}.
\end{equation}
To prove this, we expand the derivative of order $k$
\begin{equation}\label{eq:Tk}
\Theta_0^{(k)}(x)=
c_{k,k}x^{\alpha-\frac12-k} \CR(\delta^{-1}x)+
\sum_{l=1}^{k} c_{k,l} \delta^{-l} x^{\alpha-\frac12-k+l} \CR^{(l)}(\delta^{-1}x), 
\end{equation}
for some constants $c_{k,l}$.
We treat the first term in \eqref{eq:Tk}.
For $-|s| \le y \le |s|$,  from \eqref{eq:B4} and $\alpha>1$, we have
$\left| \frac{\lambda}{\sigma} y \right| <1$, so that, by using again \eqref{eq:B4} and \eqref{eq:B4.1},
\begin{equation*} 
\lambda^{\frac12+k} (\lambda y+\sigma)^{\alpha-\frac12-k} 
= \lambda^{\frac12}\sigma^{\alpha-\frac12}\left(\frac{\lambda}{\sigma}\right)^k\left(1+\frac{\lambda}{\sigma} y\right)^{\alpha-\frac 12-k} \lesssim |s|^{-1-k}.
\end{equation*}
For $ |s|\leq y \leq \lambda^{-1}\delta$, we have obviously
$\sigma \leq  \lambda y + \sigma$. But for $x\geq \sigma$, the following holds
\[
x^{\alpha-\frac12-k} \CR(\delta^{-1}x)
\lesssim \sigma^{\alpha-\frac12-k}
+\delta^{\alpha-\frac12-k} ,
\]
which implies, for $ |s|\leq y \leq \lambda^{-1}\delta$,
\begin{align*}
\lambda^{\frac12+k} (\lambda y+\sigma)^{\alpha-\frac12-k}\CR(\delta^{-1}(\lambda y + \sigma))
&\lesssim \lambda^{\frac12+k} \sigma^{\alpha-\frac12-k}
+\lambda^{\frac12+k} \delta^{\alpha-\frac12-k} \\
&\lesssim \lambda^\frac12 \sigma^{\alpha-\frac12} \left( \frac\lambda\sigma\right)^k
+\lambda^{\frac12+k} \delta^{\alpha-\frac12-k} \\
&\lesssim |s|^{-1-k} + \lambda^{\frac12+k} \delta^{\alpha-\frac12-k} .
\end{align*}
Note also that for $y\geq \lambda^{-1} \delta$, we have $x=\lambda y + \sigma 
\geq \delta + \sigma \geq \delta$
and thus $\CR(\delta^{-1}x)=0$. In conclusion for this term, for $|y|>-|s|$,
\[
\lambda^{\frac12+k} (\lambda y+\sigma)^{\alpha-\frac12-k}\CR(\delta^{-1}(\lambda y + \sigma))
\lesssim |s|^{-1-k}\eta+ \lambda^{\frac12+k} \delta^{\alpha-\frac12-k}\eta_{R}.
\]
For the other terms in \eqref{eq:Tk}, we use that 
by the definition of $\CR$,  for $l=1,\ldots,k$, for all $x\in \RR$,
\[
|\delta^{-l} x^{\alpha-\frac12-k+l} \CR^{(l)}(\delta^{-1}x)|
\lesssim \delta^{\alpha-\frac12-k}.
\]
Moreover, as before, 
if $y\geq \lambda^{-1} \delta$ then $x =\lambda y + \sigma >\lambda y \geq  \delta$
and if $y<|s|$ then 
$x=\lambda y + \sigma \lesssim \sigma <\frac 12\delta$.
Thus, in both cases, for $l\geq 1$,
$\CR^{(l)}(\delta^{-1}(\lambda y + \sigma))=0$.
It follows that
\[
\lambda^{\frac 12 + k} \sum_{l=1}^{k}   \delta^{-l} (\lambda y+\sigma)^{\alpha-\frac12-k+l} 
|\CR^{(l)}(\delta^{-1} (\lambda y+\sigma))| \lesssim 
\lambda^{\frac12+k} \delta^{\alpha-\frac12-k}\eta_{R}.
\]
We have prove \eqref{eq:dk1}.

Finally, we estimate the second term in \eqref{decomp:dktheta}.
We recall that $\Theta_1=-\Theta_0'''-m_0^4(\Theta_0^5)'$.
Estimate \eqref{eq:dk1} implies that
\[
\lambda^{\frac12+k+3}|\Theta_0^{(k+3)}(\lambda y +\sigma)|
\lesssim |s|^{-4-k}\eta + \lambda^{\frac12+k+3}\delta^{\alpha-\frac72-k}\eta_{R},
\]
and so
\[
\lambda^{\frac12+k}\tau |\Theta_0^{(k+3)}(\lambda y +\sigma)|
\lesssim |s|^{-1-k} (s^{-3}\tau \lambda^{-3}) \eta + \lambda^{\frac12+k}\tau\delta^{\alpha-\frac72-k}\eta_{R}.
\]
Since $s^{-3}\tau \lambda^{-3} \lesssim s^{-12}$, we obtain
\[
\lambda^{\frac12+k}\tau |\Theta_0^{(k+3)}(\lambda y +\sigma)|
\lesssim |s|^{-3-k}  \eta + \lambda^{\frac12+k}\tau\delta^{\alpha-\frac72-k}\eta_{R}.
\]
The term containing $(\Theta_0^5)'$ is similar and easier.
\end{proof}

Next, we derive precise asymptotics for $\theta$ in the solitonic region.
\begin{lemma}\label{le:35}
For all $-|s|\leq y\leq |s|$, it holds
\begin{equation}\label{eq:0T}
\begin{aligned}
\theta(s,y)& = 
(2\alpha)^{-1} |s|^{-1}+O(C_1^\star(1+|y|)|s|^{-2}\log|s|),
\\
\partial_y\theta(s,y) & =  \left(\alpha-\frac12\right)(2\alpha)^{-2} |s|^{-2}+O(C_1^\star(1+|y|)|s|^{-3}\log|s|),\\
\partial_y^2\theta(s,y) & = \left(\alpha-\frac12\right)\left(\alpha-\frac32\right)(2\alpha)^{-3} |s|^{-3}+O(C_1^\star(1+|y|)|s|^{-4}\log|s|).
\end{aligned}
\end{equation}
Moreover, the following expansions hold
\begin{equation}\label{eq:8T}
\begin{aligned}
\lambda^\frac12\sigma^{\alpha-\frac12}&=\theta-y\partial_y\theta+\left(\kappa_\alpha
+\tfrac12 (\alpha-\tfrac12)(\alpha-\tfrac32)y^2\right)\theta^3
+ \Psi_1,\\
\theta  \partial_y \theta &= (\alpha-\tfrac12) \theta^3 + \Psi_2,\\
\partial_y^2 \theta & = (\alpha-\tfrac12) (\alpha-\tfrac3 2)\theta^3+ \Psi_3,\\
\lambda \sigma^{2\alpha-1} &= \theta^2 - 2(\alpha-\tfrac12) y \theta^3 + \Psi_4,\\
\lambda^\frac12\sigma^{\alpha-\frac12}\theta  &= \theta^2 - (\alpha-\tfrac12) y \theta^3 + \Psi_5,\\
\lambda \sigma^{2\alpha-1} \theta   &=   \theta^3 + \Psi_6,\\
\lambda^\frac12\sigma^{\alpha-\frac12} \theta^2   &=  \theta^3 + \Psi_7,\\
\lambda^\frac12\sigma^{\alpha-\frac12} \partial_y \theta &= (\alpha-\tfrac12) \theta^3 + \Psi_8,
\end{aligned}
\end{equation}
where
\[
\kappa_\alpha = \frac{(\alpha-\tfrac12)(\alpha-\tfrac32)(\alpha-\tfrac52)}{4\alpha+3}
\]
and,  for $-|s|\leq y \leq |s|$,
\begin{align*}
|\Psi_1|+|\partial_y\Psi_1| &\lesssim  C_1^\star (1+|y|^3) |s|^{-4}\log|s|; \\
|\Psi_2|+|\partial_y\Psi_2|+|\Psi_3|+|\partial_y\Psi_3| &\lesssim  C_1^\star (1+|y|) |s|^{-4}\log|s| ;  \\
|\Psi_4|+|\partial_y\Psi_4|+|\Psi_5|+|\partial_y\Psi_5| &\lesssim  C_1^\star (1+y^2)|s|^{-4}\log|s|; \\
|\Psi_6|+|\partial_y\Psi_6|+|\Psi_7|+|\partial_y\Psi_7|   &\lesssim  (1+|y|)|s|^{-4};\\
|\Psi_8|+|\partial_y\Psi_8| &\lesssim  C_1^\star(1+|y|) |s|^{-4}\log|s|.
\end{align*}
\end{lemma}
\begin{remark}\label{rk:ag}
For $k=0,1,\ldots,4$, Lemma \ref{le:th} gives bounds on $\theta$ of the form
\[
\mbox{for $-|s|\leq y\leq |s|$,}\quad 
|\partial_y^k \theta(s,y)|\lesssim |s|^{-1-k}.
\]
Estimates \eqref{eq:0T} give the first order expansions of
$\theta$, $\partial_y\theta$ and $\partial_y^2\theta$ in $|s|^{-1}$.
Since we use the bootstrap estimate \eqref{eq:B2}--\eqref{eq:B3}
to estimate the parameters $\lambda$ and $\sigma$, the error term depends on the bootstrap
constant $C_1^\star$.

The first estimate in \eqref{eq:8T} says that $\lambda^\frac12\sigma^{\alpha-\frac12}$
can be expanded in terms of $\theta$, $y\partial_y \theta$, $\theta^3$
and $y^2\theta^3$ up to order $|s|^{-4}$. While $\partial_y \theta$ and $\theta^2$ are of the same
order $|s|^{-2}$, to replace $\partial_y \theta$ by $\theta^2$ in this
expansion would require an expansion of the parameters 
$\lambda$ and $\sigma$ at the next order.

An approximate solution at an order in $|s|^{-1}$ higher that $|s|^{-3}$ 
(for example, in an attempt to treat the case where $0<\alpha\leq 1$), would also
require more precise estimates for the parameters than just the bootstrap
estimates \eqref{eq:B1}--\eqref{eq:B3}.
\end{remark}

\begin{proof}
In this proof, we assume $|y|\leq |s|$, so that by \eqref{eq:B4}, it holds $\frac{\lambda}{\sigma}|y|\leq \frac 12 $.
In particular, for $|s|$ large enough,
$\lambda y + \sigma \in (0,\frac \delta2)$.
Since $\Theta_0(x)=x^{\alpha-\frac12}$ in the region $0<x<\frac\delta 2$, (see \eqref{eq:T0}), we have
by explicit computations
\begin{align*}
\lambda^{\frac12}\Theta_0(\lambda y+\sigma) & =\lambda^{\frac12}\sigma^{\alpha-\frac 12} \left(1+\frac{\lambda}{\sigma} y\right)^{\alpha-\frac 12};\\
\lambda^{\frac32}\Theta_0'(\lambda y+\sigma) & =(\alpha-\tfrac12)\lambda^{\frac32}\sigma^{\alpha-\frac 32} \left(1+\frac{\lambda}{\sigma} y\right)^{\alpha-\frac 32};\\
\lambda^{\frac52}\Theta_0''(\lambda y+\sigma) & =(\alpha-\tfrac12)(\alpha-\tfrac32)\lambda^{\frac52}\sigma^{\alpha-\frac 52} \left(1+\frac{\lambda}{\sigma} y\right)^{\alpha-\frac 52}; \\ 
\tau \lambda^{\frac12}\Theta_0'''(\lambda y+\sigma)&=(\alpha-\tfrac12)(\alpha-\tfrac32)(\alpha-\tfrac52)\lambda^{\frac12}\sigma^{\alpha-\frac 12} (\tau \sigma^{-3})\left(1+\frac{\lambda}{\sigma} y\right)^{\alpha-\frac 72} ;
\\ 
\tau \lambda^{\frac12}(\Theta_0^5)'(\lambda y+\sigma)&=5(\alpha-\tfrac12)\lambda^{\frac12}\sigma^{\alpha-\frac 12} (\tau \sigma^{4\alpha-3})\left(1+\frac{\lambda}{\sigma} y\right)^{5\alpha-\frac 72}.
\end{align*}
Then, we expand the first three expressions above at the first order
of $\frac\lambda\sigma y$,
\begin{equation}\label{eq:7T}
\begin{aligned}
\lambda^\frac12\Theta_0(\lambda y+\sigma) & = 
\lambda^{\frac12}\sigma^{\alpha-\tfrac12} + O((1+|y|)|s|^{-2});\\
\lambda^\frac32\Theta_0'(\lambda y+\sigma) & =  (\alpha-\tfrac12)\lambda^{\frac12}\sigma^{\alpha-\frac12}\frac{\lambda}{\sigma}  +O((1+|y|)|s|^{-3});\\
\lambda^\frac52\Theta_0''(\lambda y+\sigma) & = (\alpha-\tfrac12)(\alpha-\tfrac32)\lambda^{\frac12}\sigma^{\alpha-\tfrac12}\left(\frac{\lambda}{\sigma}\right)^2
+O( (1+|y|)|s|^{-4}).
\end{aligned}
\end{equation}
Using the decomposition of $\partial_y^k\theta$ in \eqref{decomp:dktheta}, the estimate of
$q$ in \eqref{eq:qQ} and the bootstrap assumptions \eqref{eq:B2}--\eqref{eq:B4}, 
we easily obtain \eqref{eq:0T}. We observe in particular that the contributions of
both $\Theta_1$ and $q$ are negligible at this level. Observe that the error terms in \eqref{eq:0T}
contain the bootstrap constant $C_1^\star$ since we have used the bootstrap assumption to obtain
 precise asymptotics in terms of $|s|^{-1}$.

Now we expand all the expressions above at the same order $|s|^{-4}$, which gives
\begin{align*}
\lambda^\frac12\Theta_0(\lambda y+\sigma) & = 
 \lambda^{\frac12}\sigma^{\alpha-\tfrac12}+(\alpha-\tfrac12)\lambda^{\tfrac12}\sigma^{\alpha-\frac12}\frac{\lambda}{\sigma}y
+\frac 12(\alpha-\tfrac12)(\alpha-\tfrac32)\lambda^{\frac12}\sigma^{\alpha-\frac12}\left(\frac{\lambda}{\sigma}\right)^2y^2 \\ & \quad+O((1+|y|^3)|s|^{-4});
\\
\lambda\Theta_0^2(\lambda y+\sigma) & = 
\lambda\sigma^{2\alpha-1}+2(\alpha-\tfrac12)\lambda\sigma^{2\alpha-1}\frac{\lambda}{\sigma} y+O((1+y^2)|s|^{-4});
\\
\lambda^\frac32\Theta_0'(\lambda y+\sigma) & =  (\alpha-\tfrac12)\lambda^{\frac12}\sigma^{\alpha-\frac12}\frac{\lambda}{\sigma}{+(\alpha-\tfrac12)(\alpha-\tfrac32)\lambda^{\frac12}\sigma^{\alpha-\frac12}\left(\frac{\lambda}{\sigma}\right)^2y}  +O((1+y^2)|s|^{-4}); \\
\lambda^\frac52\Theta_0''(\lambda y+\sigma) & = (\alpha-\tfrac12)(\alpha-\tfrac32)\lambda^{\frac12}\sigma^{\alpha-\tfrac12}\left(\frac{\lambda}{\sigma}\right)^2+O( (1+|y|)|s|^{-4}); \\
\tau \lambda^{\frac12}\Theta_0'''(\lambda y+\sigma)&=(\alpha-\tfrac12)(\alpha-\tfrac32)(\alpha-\tfrac52)\lambda^{\frac12}\sigma^{\alpha-\frac 12} (\tau \sigma^{-3})+O( (1+|y|)|s|^{-4});\\
\tau \lambda^{\frac12}(\Theta_0^5)'(\lambda y+\sigma)&= O( |s|^{-5}).
\end{align*}
We observe that in the decomposition of $\theta$ given by \eqref{decomp:dktheta},
the term $(\Theta_0^5)'$ in the definition of $\Theta_1$ and the term $q$ are both negligible at the order $|s|^{-4}$. In contrast, the term $\Theta_0'''$ in the definition of
$\Theta_1$ has a non zero contribution at the order $|s|^{-3}$ 
(the first term in the expansion of $\tau\lambda^\frac12\Theta_0'''$ above).

An important observation is that all the expansions above  are still valid after differentiation
in $y$, at the same order, \emph{i.e.} 
with identical error terms of size $|s|^{-4}$
(actually, error terms may be slightly more favorable in terms of $y$ 
after differentiation but we shall not exploit this).

We also observe that the error terms above do not depend on the bootstrap constant
$C_1^\star$ since we have not used yet the bootstrap estimates \eqref{eq:B1}--\eqref{eq:B3}.
Now, we claim
\begin{equation}\label{eq:1T}
\begin{aligned}
\theta(s,y)& =  \lambda^{\frac12}\sigma^{\alpha-\frac12}+y\partial_y\theta(s,y)
- (\alpha-\tfrac12)(\alpha-\tfrac32)(\alpha-\tfrac52)\lambda^{\frac12}\sigma^{\alpha-\frac12}(\tau \sigma^{-3})
\\ &\quad -\frac12 (\alpha-\tfrac12)(\alpha-\tfrac32)\lambda^{\frac12}\sigma^{\alpha-\frac12}\left(\frac{\lambda}{\sigma}\right)^2y^2
 +O( (1+|y|^3)|s|^{-4} );
\\
\theta^2(s,y) & =  \lambda \sigma^{2\alpha-1}+2y\theta(s,y)\partial_y\theta(s,y)+O( (1+y^2)|s|^{-4} ) .
\end{aligned}
\end{equation}
To prove \eqref{eq:1T}, we first observe by combining the above identities that
\begin{align*}
\lambda^\frac12\Theta_0(\lambda y+\sigma) &= 
\lambda^{\frac12}\sigma^{\alpha-\frac12}+\lambda^\frac32\Theta_0'(\lambda y+\sigma)y 
-\frac12 (\alpha-\tfrac12)(\alpha-\tfrac32)\lambda^{\frac12}\sigma^{\alpha-\frac12}\left(\frac{\lambda}{\sigma}\right)^2y^2
\\&\quad +O((1+|y|^3)|s|^{-4}),\\
\lambda\Theta_0^2(\lambda y+\sigma)&=\lambda\sigma^{2\alpha-1}+2\lambda^2\Theta_0(\lambda y+\sigma)\Theta_0'(\lambda y+\sigma)y+O((1+y^2)|s|^{-4}) ,
\end{align*}
and second we insert the contribution of $\Theta_1$.

Then, we use $|\theta^3-\lambda^\frac32\sigma^{3\alpha-\frac32}|\lesssim C(1+|y|)|s|^{-4}$
and the bootstrap estimates \eqref{eq:B4}--\eqref{eq:B4.1} to obtain
\begin{align*}
\big|\theta^3-(2\alpha)^{-3} |s|^{-3}\big|&\lesssim C_1^\star (1+|y|) |s|^{-4}\log|s|;\\
\big| (\lambda^\frac12\sigma^{\alpha-\frac12}){\tau\sigma^{-3}}-(4\alpha+3)^{-1}(2\alpha)^{-3} |s|^{-3} \big | &\lesssim C_1^\star (1+|y|) |s|^{-4}\log|s|;\\
 \bigg| \lambda^\frac12\sigma^{\alpha-\frac12}\left(\frac\lambda\sigma\right)^2 -(2\alpha)^{-3} |s|^{-3}\bigg|
&\lesssim C_1^\star (1+|y|) |s|^{-4}\log|s|.
\end{align*}
The estimate of $\Psi_1$ follows. To derive the estimate of $\partial_y \Psi_1$, we use the identity 
\begin{equation*}
\partial_y \Psi_1=y\left(\partial_y^2\theta-(\alpha-\tfrac12)(\alpha-\tfrac32)\theta^3\right)-\frac32(\alpha-\tfrac12)(\alpha-\tfrac32)y^2\theta^2\partial_y\theta 
\end{equation*}
and the additional estimate
\begin{equation} \label{est:add:theta}
\big|\partial_y^2\theta - (\alpha-\tfrac12)(\alpha-\tfrac32) (2\alpha)^{-3} |s|^{-3}\big|
\lesssim C_1^\star(1+|y|) |s|^{-4} \log|s| .
\end{equation}
Thus, the estimate holds at the same
order after differentiation in $y$ and the estimate of $\partial_y \Psi_1$ follows.
The estimates of $\Psi_2$ and $\Psi_3$ follow from \eqref{est:add:theta} and
\begin{align*}
\big|\theta\partial_y\theta - (\alpha-\tfrac12) (2\alpha)^{-3} |s|^{-3}\big|
\lesssim C_1^\star(1+|y|) |s|^{-4} \log|s|,
\end{align*}
and similarly for their $y$-partial derivatives.
The estimates of $\Psi_4$ and its $y$-derivative follows from the  estimate of $\theta^2$ in \eqref{eq:1T} and the above estimate for $\theta\partial_y \theta$.
The estimate of $\Psi_5$ follows from multiplying the estimate
$
|\theta-  \lambda^{\frac12}\sigma^{\alpha-\frac12}-y\partial_y\theta|
\lesssim (1+|y|^2) |s|^{-3}
$
by $\theta$ and using $y\Psi_2$.
Now, we estimate $\Psi_6$ and $\Psi_7$.
As before, we have
\begin{equation*}
\theta = \lambda^{\frac12}\sigma^{\alpha-\tfrac12} + O((1+|y|)|s|^{-2});\quad
\theta^2 = \lambda\sigma^{2\alpha-\tfrac12} +  O((1+|y|)|s|^{-3}).
\end{equation*}
Multiplying the first line by $\theta^2$ and multiplying the second line by $\theta$, we
obtain the estimates on $\Psi_6$ and $\Psi_7$.
Their spatial partial derivatives are treated similarly.
We note that the error term does not contain the bootstrap constant $C_1^\star$.
This fact is not important in the sequel.

Finally, the estimate of $\Psi_8$ follows from
\begin{align*}
\big|\partial_y\theta - (\alpha-\tfrac12) (2\alpha)^{-2} |s|^{-2}\big|
&\lesssim C_1^\star(1+|y|) |s|^{-3} \log|s|;\\
\big|\theta^3-(2\alpha)^{-3} |s|^{-3}\big|&\lesssim C_1^\star (1+|y|) |s|^{-4}\log|s|,
\end{align*}
and the bootstrap estimates \eqref{eq:B2}--\eqref{eq:B3}.
The estimate for $\partial_y \Psi_8$ is proved similarly.
\end{proof}

\begin{lemma}\label{le:ww} For $k=0,\ldots,3$, and $m \in \mathbb N$,
\begin{align*}
|V_1\partial_y^k\theta| & \lesssim |s|^{-1-k}\eta
+\lambda^{\frac12+k}\eta_{R},\\
|(\Lambda_1V_1)\partial_y^k\theta|
&\lesssim 
|s|^{-1-k}\omega\eta + |s|^{-1-k} \eta_L,\\
|(\partial_y^mV_1)\partial_y^k \theta|
+|(\partial_y^m \Lambda_1V_1)\partial_y^k \theta|
&\lesssim 
|s|^{-1-k}\omega\eta + |s|^{-1-m-k} \eta_L, \quad\mbox{for $m\geq 1$,}\\
|(\partial_y^mV_2)\partial_y^k(\theta^2)|+|(\partial_y^m\Lambda_2 V_2)\partial_y^k(\theta^2)|
&\lesssim 
\begin{cases}
|s|^{-2-k}\omega \eta+|s|^{-2-k} \eta_I, & \mbox{for $m=0$,}\\
|s|^{-2-k}\omega\eta + |s|^{-2-m-k} \eta_L, & \mbox{for $m\geq1$,}
\end{cases}\\
|(\partial_y^mV_2^\star)\partial_y^{k+1} \theta|
+|(\partial_y^m\Lambda_3 V_2^\star)\partial_y^{k+1} \theta|
&\lesssim |s|^{-2-k}\omega\eta,\\
|(\partial_y^mV_3)\partial_y^k (\theta^3)|
&\lesssim 
\begin{cases}
|s|^{-3-k}\omega\eta+|s|^{-3-k}(1+|y|)\eta_I,& \mbox{for $m=0$,}\\
|s|^{-3-k}\omega\eta +|s|^{-3-k}\eta_I, & \mbox{for $m=1$,}
\\
|s|^{-3-k}\omega\eta + |s|^{-3-k-m+1}\eta_L, & \mbox{for $m\geq 2$,}
\end{cases} \\ 
|(\partial_y^m\Lambda_3 V_3)\partial_y^k(\theta^3)|
&\lesssim 
\begin{cases}
|s|^{-3-k}\omega\eta +|s|^{-3-k}\eta_I, & \mbox{for $m=0$,}
\\
|s|^{-3-k}\omega\eta + |s|^{-3-k-m}\eta_L, & \mbox{for $m\geq 1$.}
\end{cases}
\end{align*}
Moreover,
\begin{equation}\label{eq:vv}
|\partial_y^pV| 
\lesssim 
\begin{cases}
|s|^{-1} \eta + \lambda^\frac 12 \eta_{R},&\mbox{for $p=0$,}\\
|s|^{-1}\omega\eta+|s|^{-1-p}\eta + \lambda^{\frac 12 +p}\eta_R, & \mbox{for $p=1,2,3,4$} .
\end{cases}
\end{equation}
In addition, we have the following $L^2$-bounds
\begin{equation}\label{eq:mg}
\int V_1^2 \theta^2 \lesssim \delta^{2\alpha} , \
\int V_2^2 \theta^4+ 
\int V_3^2 \theta^6 \lesssim  |s|^{-3}, \ \int (V_2^\star)^2 (\partial_y\theta)^2 \lesssim |s|^{-4}
\end{equation}
and 
\begin{equation}\label{eq:mg.1}
\begin{aligned}
&\int \left(\partial_y(V_1\theta)\right)^2 \lesssim |s|^{-2} , \quad
\int \left(\partial_y(V_2\theta^2)\right)^2\lesssim |s|^{-3}, \\
&\int  \left(\partial_y(V_2^*\partial_y\theta)\right)^2 \lesssim |s|^{-4}, \quad
\int  \left(\partial_y(V_3\theta^3)\right)^2 \lesssim  |s|^{-5},
\end{aligned}
\end{equation}
so that 
\begin{equation} \label{est:V:H1}
\int V^2\lesssim \delta^{2\alpha} \quad \mbox{and} \quad \int (\partial_yV)^2\lesssim |s|^{-2} .
\end{equation}
\end{lemma}
\begin{proof}
The above estimates easily follow from the ones of
Lemmas \ref{le:vv} and \ref{le:th}.
Note that all the functions $V_j$ vanish identically
for $y<-|s|$ and so one can apply the estimates
on $\theta$ from Lemma \ref{le:th}. Estimate \eqref{eq:vv} follows by combining the Leibniz rule with the above estimates. Finally, estimates \eqref{eq:mg}, \eqref{eq:mg.1} and \eqref{est:V:H1} follow from integrating the previous pointwise estimates and also using \eqref{theta:H1}.
\end{proof}

\subsection{Components of the blowup profile}\label{s:3.6}
Here, we choose the functions
$A_1$, $A_2$, $A_2^*$, $A_3$ and the constant $c_2$ so that 
the error term $\mathcal E(W)$ is sufficiently small
(formally, we cancel all terms up to order $|s|^{-3}$, and thus
we obtain an error term of size $|s|^{-4}$).

\emph{$\bullet$ Equation of $W$.}
First, we insert the explicit expression of $W$ into the rescaled gKdV equation.
In practice, we compute $\cE(W)$ defined in \eqref{eq:EW}.

We will make use of the following identities 
\begin{align*}
\Lambda \left(\theta^j V_j \right)&=jV_j \theta^{j-1} \Lambda \theta +\left(\Lambda_jV_j\right) \theta^j , \\
\partial_y \Lambda \theta&=\Lambda \partial_y \theta+\partial_y\theta,
\end{align*}
where $\Lambda$ and $\Lambda_j$ are defined in \eqref{eq:La}.

Using the equation \eqref{eq:TT} of $\Theta$ and the definition of $\tau(s)$ in \eqref{eq:ta}
\begin{align}
\partial_s \theta
&=\frac{\lambda_s}{\lambda}\Lambda\theta +  \frac{\sigma_s}{\lambda}  \partial_y \theta
+\lambda^\frac72\partial_t \Theta(\tau(s),\lambda(s) y+\sigma(s))\nonumber\\
&=\frac{\lambda_s}{\lambda}\Lambda\theta +   \frac{\sigma_s}{\lambda} 
\partial_y \theta -\partial_y^3 \theta - m_0^4 \partial_y (\theta^5).\label{eq:TA}
\end{align}
Thus,
\begin{align*}
\partial_s W & = \frac{\lambda_s}{\lambda} \left( V_1 \Lambda\theta+2V_2\theta\Lambda\theta
+V_2^*\partial_y \Lambda\theta+3V_3\theta^2\Lambda\theta\right)\\
&\quad +  \frac{\sigma_s}{\lambda} \left(V_1 \partial_y \theta + 2 V_2\theta \partial_y \theta + 
V_2^* \partial_y^2 \theta + 3V_3 \theta^2\partial_y \theta\right)\\
&\quad-\left(V_1 \partial_y^3 \theta + m_0^4 V_1  \partial_y(\theta^5) 
+ 2V_2\theta\partial_y^3 \theta + V_2^*\partial_y^4\theta+3V_3\theta^2\partial_y^3 \theta\right)\\
&\quad 
+\cS_1+\cS_2
\end{align*}
where
\begin{align}
\cS_1 &= \partial_s V_0+ (\partial_s V_1)\theta+(\partial_s V_2)\theta^2+(\partial_s V_2^*)\partial_y\theta+(\partial_s V_3)\theta^3, \label{def:S1}\\
\cS_2  & = -2m_0^4V_2\theta \partial_y(\theta^5)
-m_0^4V_2^*\partial_y^2(\theta^5)
-3 m_0^4V_3\theta^2\partial_y(\theta^5). \label{def:S2}
\end{align}
We continue by
\begin{align*}
\partial_y(\partial_y^2W-W+W^5)
& =
-(V_1\partial_y\theta+V_2 \partial_y(\theta^2)+V_2^*\partial_y^2\theta 
+V_3\partial_y(\theta^3))\\
&\quad
+V_1 \partial_y^3 \theta + m_0^4 V_1  \partial_y(\theta^5) 
+ 2V_2\theta\partial_y^3 \theta + V_2^*\partial_y^4\theta+3V_3\theta^2\partial_y^3 \theta\\
&\quad 
- (\partial_y \cLt V_1)\theta
-(\partial_y \cLt V_2) \theta^2 - (\partial_y\cLt V_2^*) \partial_y\theta
-(\partial_y \cLt V_3) \theta^3\\
&\quad +3 (\partial_y^2V_1)\partial_y\theta
+ 3 (\partial_yV_1)\partial_y^2\theta\ 
+3(\partial_y^2V_2)\partial_y(\theta^2) 
+3(\partial_y^2V_2^*)\partial_y^2\theta
\\
&\quad
+\cRR+\cS_3,
\end{align*}
where
\begin{align*}
\cRR & = \partial_y (W^5-V_0^5)- V_1^5 \partial_y(\theta^5)\\
&\quad -5 \partial_y (V_0^4V_1)\theta-5\partial_y (V_0^4V_2)\theta^2
-5\partial_y(V_0^4V_2^*)\partial_y\theta-5\partial_y(V_0^4V_3)\theta^3 
\end{align*}
and
\begin{equation}\label{def:S3}
\begin{aligned}
\cS_3 & = (V_1^4-m_0^4)V_1\partial_y(\theta^5)
+6V_2(\partial_y\theta)(\partial_y^2 \theta)+3(\partial_yV_2)\partial_y^2(\theta^2)
+3(\partial_yV_2^*)\partial_y^3\theta   \\
&\quad 
+3(\partial_y^2V_3)\partial_y(\theta^3)
+3(\partial_yV_3)\partial_y^2(\theta^3)+ 6 V_3(3\theta(\partial_y \theta)(\partial_y^2 \theta)
+ (\partial_y\theta)^3)   \\
&\quad +Q \partial_y^3 \ZL+3Q'\partial_y^2\ZL+2Q''\partial_y\ZL
+\partial_y (Q^5  (\ZL^5-\ZL)). 
\end{aligned}
\end{equation}
Thus, combining the two identities above, we find
\begin{align*}
\partial_s W +\partial_y(\partial_y^2W-W+W^5) 
& = \frac{\lambda_s}{\lambda} \left( V_1 \Lambda\theta+2V_2\theta\Lambda\theta
+V_2^*\partial_y \Lambda\theta+3V_3\theta^2\Lambda\theta\right)\\
&\quad +  \left(\frac{\sigma_s}{\lambda}-1\right) \left(V_1 \partial_y \theta + 2 V_2\theta \partial_y \theta + 
V_2^* \partial_y^2 \theta + 3V_3 \theta^2\partial_y \theta\right)\\
&\quad - (\partial_y \cLt V_1)\theta
-(\partial_y \cLt V_2) \theta^2 - (\partial_y\cLt V_2^*) \partial_y\theta
-(\partial_y \cLt V_3) \theta^3\\
&\quad +3 (\partial_y^2V_1)\partial_y\theta
+ 3 (\partial_yV_1)\partial_y^2\theta\ 
+3(\partial_y^2V_2)\partial_y(\theta^2) 
+3(\partial_y^2V_2^*)\partial_y^2\theta\\
&\quad 
+\cRR+\cS_1+\cS_2+\cS_3.
\end{align*}
Then,
\begin{align*}
-\frac{\lambda_s}{\lambda}\Lambda W & = -\frac{\lambda_s}{\lambda}
\big( V_1\Lambda \theta + 2V_2\theta\Lambda \theta +V_2^*\partial_y\Lambda \theta+3V_3\theta^2\Lambda \theta\big)\\
&\quad  -\frac{\lambda_s}{\lambda}
\big(\Lambda V_0+\left(\Lambda_1V_1\right)\theta +\left(\Lambda_2V_2\right)\theta^2 +\left(\Lambda_3V_2^*\right)\partial_y\theta  +\left(\Lambda_3V_3\right)\theta^3\big)
\end{align*}
and
\begin{align*}
- \left(\frac{\sigma_s}{\lambda} - 1\right)\partial_y W
& = - \left(\frac{\sigma_s}{\lambda} - 1\right)
\big( V_1 \partial_y \theta
+ 2 V_2\theta  \partial_y \theta  + V_2^* \partial_y^2 \theta 
+  3V_3 \theta^2  \partial_y \theta  \big)\\
&\quad 
 - \left(\frac{\sigma_s}{\lambda} - 1\right)
\big( \partial_y V_0 + (\partial_y V_1)\theta 
+ (\partial_y V_2)\theta^2 
+ (\partial_y V_2^*)\partial_y\theta 
+ (\partial_y V_3)\theta^3 \big)  .
\end{align*}
Therefore,
\begin{align*}
\cE(W)
& = -\frac{\lambda_s}{\lambda} \left( \Lambda V_0 + (\Lambda_1 V_1) \theta
+(\Lambda_2 V_2)\theta^2 + (\Lambda_3V_2^*) \partial_y \theta
+(\Lambda_3V_3)\theta^3\right)   \\
&\quad - \left(\frac{\sigma_s}{\lambda}-1\right) \left(\partial_yV_0+(\partial_y V_1) \theta+(\partial_yV_2)\theta^2
+(\partial_yV_2^*)\partial_y\theta+(\partial_yV_3)\theta^3\right)\\
&\quad - (\partial_y \cLt V_1)\theta
-(\partial_y \cLt V_2) \theta^2 - (\partial_y\cLt V_2^*) \partial_y\theta
-(\partial_y \cLt V_3) \theta^3\\
&\quad +3 (\partial_y^2V_1)\partial_y\theta
+ 3 (\partial_yV_1)\partial_y^2\theta\ 
+3(\partial_y^2V_2)\partial_y(\theta^2) 
+3(\partial_y^2V_2^*)\partial_y^2\theta\\
&\quad 
+\cRR+\cS_1+\cS_2+\cS_3.
\end{align*}
Now, we introduce 
\begin{align}
\Psi_\lambda & =   
(\Lambda_1 V_1) \theta
+(\Lambda_2 V_2)\theta^2 + (\Lambda_3V_2^*) \partial_y \theta
+(\Lambda_3V_3)\theta^3,\label{eq:dL}\\
\Psi_\sigma & =   (\partial_y V_1) \theta+(\partial_yV_2)\theta^2
+(\partial_yV_2^*)\partial_y\theta+(\partial_yV_3)\theta^3.\label{eq:dS}
\end{align}
and $\beta=c_1 \lambda^\frac12\sigma^{\alpha-\frac12}+c_2\lambda\sigma^{2\alpha-1}$
so that 
\begin{align*}
\cE(W) & = 
-\left(\frac{\lambda_s}{\lambda} + \beta\right)
(\Lambda V_0 + \Psi_\lambda) 
-\left(\frac{\sigma_s}{\lambda}-1\right)
(\partial_yV_0+\Psi_\sigma)\\
&\quad - (\partial_y \cLt V_1)\theta
-(\partial_y \cLt V_2) \theta^2 - (\partial_y\cLt V_2^*) \partial_y\theta
-(\partial_y \cLt V_3) \theta^3\\
&\quad+c_1\lambda^\frac12 \sigma^{\alpha-\frac12}\Lambda V_0 
+c_2\lambda \sigma^{2\alpha-1} \Lambda V_0 + c_1\lambda^\frac12\sigma^{\alpha-\frac12} \theta \Lambda_1V_1
+(\partial_y\theta)  3\partial_y^2 V_1\\
&\quad
+\theta^3 \left(-\partial_y({\cLt} V_3)\right)
+3(\partial_y^2V_2)\partial_y(\theta^2) 
+( 3 \partial_yV_1+ 3\partial_y^2V_2^* )\partial_y^2\theta\\
&\quad + c_2\lambda\sigma^{2\alpha-1}\theta\Lambda_1V_1
+c_1 \lambda^\frac12\sigma^{\alpha-\frac12}\theta^2\Lambda_2V_2
+c_1\lambda^\frac12\sigma^{\alpha-\frac12}\partial_y\theta\Lambda_3V_2^*
\\
&\quad +\cRR+\cS_1+\cS_2+\cS_3+\cS_4
\end{align*}
where
\begin{equation} \label{def:S4}
\cS_4 = c_2\lambda\sigma^{2\alpha-1}\theta^2\Lambda_2V_2 +c_2\lambda\sigma^{2\alpha-1}(\partial_y\theta) \Lambda_3V_2^\star+c_1\lambda^{\frac12}\sigma^{\alpha-\frac12}\theta^3\Lambda_3V_3+c_2\lambda\sigma^{2\alpha-1}\theta^3\Lambda_3V_3.
\end{equation}
Now, we expand $\cR$
\begin{align*}
\cRR & = \partial_y (5V_0^4V+10V_0^3V^2+10V_0^2V^3+5V_0V^4+V^5)\\
&\quad
-5 \partial_y (V_0^4V_1)\theta-5\partial_y (V_0^4V_2)\theta^2
-5\partial_y(V_0^4V_2^*)\partial_y\theta-5\partial_y(V_0^4V_3)\theta^3 - V_1^5 \partial_y(\theta^5)\\
& = 
10\partial_y(V_0^3V_1^2)\theta^2
+5V_0^4V_1 \partial_y \theta
+(10\partial_y (V_0^2V_1^3)+20\partial_y(V_0^3V_1V_2)) \theta^3\\
&\quad +(10V_0^3V_1^2+5V_0^4V_2+10\partial_y(V_0^3V_1V_2^*))\partial_y (\theta^2)
+5V_0^4V_2^*\partial_y^2 \theta
+\cS_5
\end{align*}
where
\begin{align*}
\cS_5 &= 5V_0^4V_3\partial_y(\theta^3) +\partial_y (10V_0^3V^2+10V_0^2V^3+5V_0V^4+V^5)- V_1^5 \partial_y(\theta^5)\\
&\quad -10(\partial_y(V_0^3V_1^2))\theta^2  
-(10\partial_y (V_0^2V_1^3)  
+20\partial_y(V_0^3V_1V_2)) \theta^3\\
&\quad -(10V_0^3V_1^2+10\partial_y(V_0^3V_1V_2^*)) \partial_y (\theta^2).
\end{align*}

Thus,
\begin{align*}
\cE(W) & = 
-\left(\frac{\lambda_s}{\lambda} + \beta\right)
(\Lambda V_0 + \Psi_\lambda) 
-\left(\frac{\sigma_s}{\lambda}-1\right)
(\partial_yV_0+\Psi_\sigma)\\
&\quad - (\partial_y \cLt V_1)\theta
-(\partial_y \cLt V_2) \theta^2 - (\partial_y\cLt V_2^*) \partial_y\theta
-(\partial_y \cLt V_3) \theta^3\\
&\quad +F+\cS_1+\cS_2+\cS_3+\cS_4+\cS_5,
\end{align*}
where
\begin{align*}
F &= c_1\lambda^\frac12 \sigma^{\alpha-\frac12}\Lambda V_0 \\
&\quad+\theta^2  10\partial_y(V_0^3V_1^2)  
+c_2\lambda \sigma^{2\alpha-1} \Lambda V_0 + c_1\lambda^\frac12\sigma^{\alpha-\frac12} \theta \Lambda_1V_1
+(\partial_y\theta) \left( 3\partial_y^2 V_1+5V_0^4V_1\right)\\
&\quad
+\theta^3 \left(  10\partial_y (V_0^2V_1^3)+20\partial_y(V_0^3V_1V_2)\right)\\
&\quad + \partial_y(\theta^2)  \left(3\partial_y^2V_2+10V_0^3V_1^2+5V_0^4V_2+10\partial_y(V_0^3V_1V_2^*)\right)\\
&\quad +\partial_y^2\theta\left( 3 \partial_yV_1 + 3\partial_y^2V_2^*
+5V_0^4V_2^*\right)\\
&\quad + c_2\lambda\sigma^{2\alpha-1}\theta\Lambda_1V_1
+c_1 \lambda^\frac12\sigma^{\alpha-\frac12}\theta^2\Lambda_2V_2
+c_1\lambda^\frac12\sigma^{\alpha-\frac12}\partial_y\theta\Lambda_3V_2^*
\end{align*}
Now, using Lemma \ref{le:35}, we  write $F$ in terms
of $\theta$, $\theta^2$, $\partial_y\theta$ and $\theta^3$,
up to the error terms $\Psi_1$, $\Psi_2$ and $\Psi_3$.
We find
 \[
 F = \theta (c_1 \Lambda V_0 )   + \theta^2(c_2 \Lambda V_0+ \widetilde F_2 )
 +\partial_y \theta \widetilde F_2^* + \theta^3 \widetilde F_3
 +\cS_6
 \]
 where
 \begin{align*}
\widetilde F_2 & =  10\partial_y(V_0^3V_1^2)+c_1 \Lambda_1V_1\\
\widetilde F_2^* & = -c_1 y \Lambda V_0 + 3\partial_y^2 V_1+5V_0^4V_1\\
\widetilde F_3 & =
c_1\left(\kappa_\alpha +\tfrac12 (\alpha-\tfrac12)(\alpha-\tfrac32)y^2\right)\Lambda V_0
-2 c_2(\alpha-\tfrac12)y\Lambda V_0
-c_1 (\alpha-\tfrac12) y \Lambda_1 V_1
+ 10\partial_y (V_0^2V_1^3)\\
&\quad +20\partial_y(V_0^3V_1V_2)+ 2(\alpha-\tfrac12) \left(3\partial_y^2V_2+10V_0^3V_1^2+5V_0^4V_2+10\partial_y(V_0^3V_1V_2^*)\right)\\
&\quad +(\alpha-\tfrac12) (\alpha-\tfrac32) \left( 3 \partial_yV_1 + 3\partial_y^2V_2^*
+5V_0^4V_2^*\right)+c_2 \Lambda_1 V_1
+c_1 \Lambda_2 V_2+c_1 (\alpha-\tfrac12)\Lambda_3 V_2^*\\
\end{align*}
and
\begin{align*}
\cS_6 
&= c_1 \Lambda V_0 \Psi_1 +c_2\Lambda V_0 \Psi_4
+c_1 \Lambda_1 V_1\Psi_5\\
&\quad +2 \left(3\partial_y^2V_2+10V_0^3V_1^2+5V_0^4V_2+10\partial_y(V_0^3V_1V_2^*)\right)\Psi_2\\
&\quad + \left( 3 \partial_yV_1 + 3\partial_y^2V_2^*
+5V_0^4V_2^*\right)\Psi_3+c_2 \Lambda_1 V_1 \Psi_6
+c_1 \Lambda_2 V_2\Psi_7 +c_1 \Lambda_3V_2^*\Psi_8.
\end{align*}
In relation to functions defined above, it is natural to introduce the corresponding time-independent functions (without cut-off)
\begin{align*}
F_2 &=  c_1^2 \left( 10 (Q^3A_1^2)' + \Lambda_1 A_1 \right)\\
F_2^*&= c_1 (-y\Lambda Q + 3 A_1'' + 5Q^4A_1 )\\
F_3 &=c_1\left(\kappa_\alpha +\tfrac12 (\alpha-\tfrac12)(\alpha-\tfrac32)y^2\right)\Lambda Q
-c_1^2 (\alpha-\tfrac12) y \Lambda_1 A_1
+ 10c_1^3 \partial_y (Q^2A_1^3)\\
&\quad +20c_1(Q^3A_1A_2)'+ 2(\alpha-\tfrac12) \left(3A_2''+10c_1^2Q^3A_1^2+5Q^4A_2+10c_1\partial_y(Q_0^3A_1A_2^*)\right)\\
&\quad +(\alpha-\tfrac12) (\alpha-\tfrac32) \left( 3c_1 A_1' + 3(A_2^*)''
+5Q^4A_2^*\right)
+c_1 \Lambda_2 A_2+c_1 (\alpha-\tfrac12)\Lambda_3 A_2^* \\
\Omega &=-2  (\alpha-\tfrac12)y\Lambda Q
+20c_1(Q^3A_1P)'+2(\alpha-\tfrac12)(3P''+5Q^4P)
+c_1\Lambda_1A_1 +c_1\Lambda_2P.
\end{align*}
(We have separated $\widetilde F_3 \approx c_2 \Omega + F_3$, where $F_3$ does not depend on $c_2$, since we need later to fix the parameter $c_2$ depending on $F_3$, so that 
$c_2 \Omega + F_3$ satisfies a suitable orthogonality condition).
Observe that (note that $\Lambda_1 A_1 = y A_1'\in \cY$)
\begin{equation}\label{eq:FF}
F_2 \in \cY,\quad F_2^*\in\cY,\quad F_3\in \cY, \quad \Omega\in \cZ_0^-.
\end{equation}

\emph{$\bullet$ Construction of $A_2$ and $A_2^*$.} We claim that there exist unique even functions $A_2, A_2^* \in \cY$ such that 
\[
(\cL A_2)'=F_2 \quad \mbox{and} \quad (\cL A_2^*)'=F_2^*.
\]
Indeed,  $F_2$ is odd, and thus 
$y\mapsto\int_{-\infty}^{y}F_2$ is an even function that also belongs to $\cY$.
In particular $(\int_{-\infty}^{y}F_2,\partial_y Q)=0$. Hence, using (iv) of Lemma~\ref{le:li}, there exists a unique function $A_2 \in \cY$, even, such that 
$\cL A_2=\int_{-\infty}^y F_2(y) dy$,
which implies that $(\cL A_2)'=F_2$. 
The argument for $A_2^*$ is identical.
Note that the above equations for $A_2$ and $A_2^*$ do not involve the parameter $c_2$, which
actually matters only at the order $|s|^{-3}$.

\emph{$\bullet$ Definition of $c_2$ and construction of $A_3$.}
We claim the following non-degeneracy condition
\begin{equation}\label{eq:nd}
(\Omega, Q) =  -2m_0^2 .
\end{equation}
(We only need that $(\Omega,Q)\neq 0$ for the range of $\alpha$ considered.
However, we observe that the value does not depend on $\alpha$ and thus never vanishes, which means that, at least for this step, such construction can be extended to any $\alpha>0$.)

\emph{Proof of \eqref{eq:nd}.}
Indeed, it follows from \eqref{eq:Q5}, \eqref{eq:pq} and \eqref{eq:PQ} that 
\begin{equation*} 
\left(3P''+5Q^4 P,Q \right)=3\left( P, Q\right)+2\left( P,Q^5 \right)=3m_0^2 -2 m_0^2=m_0^2 .
\end{equation*}
Combining this identity with \eqref{eq:id} and recalling $c_1=-(2\alpha+1)$ yields \eqref{eq:nd}.

We deduce from \eqref{eq:nd} that there
exists a unique value of $c_2$ such that
\[
(c_2 \Omega + F_3,Q)=0.
\]
The parameter $c_2$ is fixed to such value
so that by (v) of Lemma \ref{le:li}, there exists $A_3\in \cZ_1^-$ such that
\[
(LA_3)'=c_2\Omega+F_3. 
\]

$\bullet$ In conclusion of these computations, we have obtained
\[
\cE(W)   = 
-\left(\frac{\lambda_s}{\lambda} + \beta\right)
\left( \Lambda V_0+\Psi_\lambda\right)
-\left(\frac{\sigma_s}{\lambda}-1\right)
\left(\partial_y V_0+\Psi_\sigma\right) +\Psi_W
\]
where
\begin{equation}\label{eq:dW}
\Psi_W=
\cS_1+\cS_2+\cS_3+\cS_4+\cS_5+\cS_6+\cS_7
\end{equation}
and, by the equations satisfied by the functions $A_1$, $A_2$, $A_2^*$ and $A_3$,
\begin{align*}
\cS_7 
& = c_1 \theta  \left(-\partial_y(\cLt (A_1\ZL))  +   \Lambda V_0 \right)
+ c_2 \theta^2 \left(
-\partial_y(\cLt (P\ZL)) + \Lambda V_0 \right) + \theta^2 \left(-\partial_y(\cLt (A_2\ZL))  
+ \widetilde F_2\right)\\
&\quad+ \partial_y\theta  \left(-\partial_y({\cLt} (A_2^*\ZL)) +  \widetilde F_2^*\right)
+\theta^3 \left(-\partial_y(\cLt(A_3\ZL)) + \widetilde F_3\right)\\
& = c_1 \theta  \left(-\partial_y(\cLt (A_1\ZL)- \cL A_1) +   \Lambda V_0-\Lambda Q\right)+ c_2 \theta^2 \left(
-\partial_y(\cLt (P\ZL)- \cL P) + \Lambda V_0-\Lambda Q\right)\\
&\quad + \theta^2 \left(-\partial_y(\cLt (A_2\ZL) -\cL A_2)
+ \widetilde F_2-F_2\right)+ \partial_y\theta  \left(-\partial_y({\cLt} (A_2^*\ZL)-\cL A_2^*) +  \widetilde F_2^*-F_2\right)\\
&\quad
+\theta^3 \left(-\partial_y(\cLt(A_3\ZL)-\cL A_3) + \widetilde F_3-(c_2\Omega + F_3)\right).
\end{align*}

\subsection{Estimates of the error terms}
In this subsection, we prove the pointwise estimates for the error terms $\Psi_\lambda$, $\Psi_\sigma$ and $\Psi_W$ defined in \eqref{eq:dL}, \eqref{eq:dS} and \eqref{eq:dW}.

\smallskip
\noindent\emph{Estimates for $\Psi_\lambda$ 
and $\Psi_\sigma$.}
Estimates \eqref{eq:ls} and \eqref{eq:ss} follows by combining the Leibniz rule with the estimates in Lemma \ref{le:ww}.
Moreover, the same estimates show that
\[
(\Psi_\lambda - \Lambda_1 V_1 \theta, Q) \lesssim |s|^{-2}.
\]
Now, we compute $(\Lambda_1 V_1\theta ,Q)$. Since $\Lambda_1V_1$ is odd and $Q$ is even,
we have $(\Lambda_1 V_1 ,Q)=0$. Thus, by \eqref{eq:0T} and \eqref{eq:lt},
\[
\left|(\Lambda_1 V_1\theta ,Q) \right| =
\left|(\Lambda_1 V_1(\theta-(2\alpha)^{-1} |s|^{-1}) ,Q) \right|
\lesssim C_1^\star |s|^{-2}\log|s|,
\]
which proves \eqref{est:Psi_lambda:Q}.

\smallskip

\noindent\emph{Estimate of $\Psi_W$.}
We will estimate each term $S_j$ separately. 

\noindent\emph{Estimate of $\cS_1$.} We claim that 
 for any $p=0,1,2$,
\begin{equation} \label{es:s1}
|\partial_y^p\cS_1| \lesssim |s|^{-2-p}\eta_L .
\end{equation}

Indeed, $\cS_1$ is defined in \eqref{def:S1}. Observe first that 
\begin{equation*} 
\left| \partial_y^p\partial_s\chi_L(|s|^{-1}y)\right| \lesssim |s|^{-2-p}|y|\eta_L \lesssim |s|^{-1-p}\eta_L .
\end{equation*}
Thus, using the estimates in Lemmas \ref{le:vv} and \ref{le:th} as in Lemma \ref{le:ww}, we find, for $p=0,1$,
\begin{align*} 
 |\partial_y^p\partial_sV_0| &=\left|\partial_y^p\left(\partial_s\chi_L(|s|^{-1}y) Q\right)\right| \lesssim |s|^{-1}\eta_L w \lesssim  |s|^{-100} \eta_L, \\
|\partial_y^p((\partial_sV_1) \theta)| &=\left|\partial_y^p\left(\partial_s\chi_L(|s|^{-1}y) A_1\theta \right)\right|\lesssim |s|^{-2-p}\eta_L, \\
|\partial_y^p((\partial_sV_2) \theta^2)| &=\left|\partial_y^p\left(\partial_s\chi_L(|s|^{-1}y) (c_2P+A_2)\theta^2 \right)\right|\lesssim |s|^{-3-p}\eta_L, \\
|\partial_y^p((\partial_sV_2^*) \partial_y\theta)| &=\left|\partial_y^p\left(\partial_s\chi_L(|s|^{-1}y) A_2^*\partial_y\theta \right)\right|\lesssim |s|^{-100} \eta_L, \\
|\partial_y^p((\partial_sV_3) \theta^3)| &=\left|\partial_y^p\left(\partial_s\chi_L(|s|^{-1}y) A_3\theta^3 \right)\right|\lesssim |s|^{-3-p}\eta_L,
\end{align*}
which yields \eqref{es:s1}.

\noindent\emph{Estimate of $\cS_2$.} We claim that 
 for any $p=0,1$,
\begin{equation} \label{es:s2}
|\partial_y^p\cS_2| \lesssim |s|^{-7}\omega \eta+|s|^{-7-p}\eta_I .
\end{equation} 

We recall that $\cS_2$ is defined in \eqref{def:S2}. By using the estimates in Lemmas \ref{le:vv} and \ref{le:th} as in Lemma \ref{le:ww}, we find, for $p=0,1$,
\begin{align*} 
 \left|\partial_y^p\left(V_2 \theta \partial_y(\theta^5)\right) \right| & \lesssim |s|^{-7}\omega \eta+|s|^{-7-p}\eta_I, \\
 \left|\partial_y^p\left(V_2^* \partial_y^2(\theta^5)\right) \right| &\lesssim |s|^{-7}\omega , \\ 
 \left|\partial_y^p\left(V_3 \theta^2 \partial_y(\theta^5)\right) \right| &\lesssim |s|^{-8}\omega \eta+|s|^{-7-p}\eta_I ,
\end{align*}
which yields \eqref{es:s2}.

\noindent\emph{Estimate of $\cS_3$.} We claim that 
 for any $p=0,1$,
\begin{equation} \label{es:s3}
|\partial_y^p\cS_3 | \lesssim |s|^{-4}\omega \eta+|s|^{-5-p}\eta_I .
\end{equation} 

Recall that $\cS_3$ is defined in \eqref{def:S3}. 
By using Lemmas \ref{le:vv}, \ref{le:th}, \ref{le:ww}, and the properties of $A_1$ in Lemma \ref{le:PA}, we estimate, for $p=0,1$,
\begin{align*} 
 \left|\partial_y^p\left((V_1^4-m_0^4)V_1\partial_y(\theta^5)\right) \right| & \lesssim |s|^{-6}\omega \eta+|s|^{-6-p}\eta_I, \\
 \left|\partial_y^p\left(V_2(\partial_y\theta)(\partial_y^2 \theta)\right) \right|+\left|\partial_y^p\left((\partial_yV_3)\partial_y^2(\theta^3)\right) \right| &\lesssim  |s|^{-5}\omega \eta+|s|^{-5-p}\eta_I, \\ 
 \left|\partial_y^p\left((\partial_yV_2)\partial_y^2(\theta^2)\right) \right|+\left|\partial_y^p\left((\partial_y^2V_3)\partial_y(\theta^3)\right) \right| &\lesssim  |s|^{-4}\omega \eta+|s|^{-5-p}\eta_L, \\
 \left|\partial_y^p\left(V_3(\theta(\partial_y \theta)(\partial_y^2 \theta)+(\partial_y\theta)^3\right) \right| &\lesssim  |s|^{-6}\omega \eta+|s|^{-5-p}\eta_I, \\
 \left|\partial_y^p\left(Q \partial_y^3 \ZL+3Q'\partial_y^2\ZL+2Q''\partial_y\ZL
+\partial_y (Q^5  (\ZL^5-\ZL))\right) \right| & \lesssim |s|^{-100} \omega \eta_L ,
\end{align*}
which yields \eqref{es:s3}. 

\noindent\emph{Estimate of $\cS_4$.} We claim that 
 for any $p=0,1$,
\begin{equation} \label{es:s4}
|\partial_y^p\cS_4| \lesssim |s|^{-4}\omega \eta+|s|^{-4-p}\eta_I . 
\end{equation} 

 Recall that $\cS_4$ is defined in \eqref{def:S4}. We deduce from the bootstrap estimates \eqref{eq:B4} and Lemma \ref{le:ww} that for $p=0,1$, 
\begin{align*} 
 \left|\partial_y^p\left(\lambda\sigma^{2\alpha-1}\Lambda_2V_2\theta^2\right) \right| & \lesssim |s|^{-4}\omega \eta+|s|^{-4-p}\eta_I, \\
\left|\partial_y^p\left(\lambda\sigma^{2\alpha-1} \Lambda_3V_2^\star (\partial_y\theta)\right) \right| &\lesssim |s|^{-4}\omega \eta , \\ 
\left|\partial_y^p\left(\lambda^{\frac12}\sigma^{\alpha-\frac12}\theta^3\Lambda_3V_3\right) \right| &\lesssim |s|^{-4}\omega \eta+|s|^{-4-p}\eta_I , 
\end{align*}
which yields \eqref{es:s4}.

\noindent\emph{Estimate of $\cS_5$.} We claim that 
 for any $p=0,1$,
\begin{equation} \label{es:s5}
|\partial_y^p\cS_5| \lesssim |s|^{-4}\omega \eta +|s|^{-6-p}\eta_I
\end{equation}
We rewrite the error term $\cS_5$ as follows
\begin{align*}
\cS_5 &= 5V_0^4V_3\partial_y(\theta^3) +5\partial_y (V_0V^4)\\
&\quad +10 \partial_y (V_0^3V^2- V_0^3V_1^2 \theta^2)
-20 \partial_y (V_0^3V_1V_2) \theta^3
-20 \partial_y(V_0^3V_1V_2^*)\theta\partial_y\theta
 \\
&\quad 
+10\partial_y(V_0^2V^3)-10 \partial_y(V_0^2V_1^3) \theta^3
\\
&\quad +\partial_y(V^5)- V_1^5 \partial_y(\theta^5)\\
&\quad = \cS_{5,1}+\cS_{5,2}+\cS_{5,3}+\cS_{5,4}.
\end{align*}
The two terms composing $\cS_{5,1}$ are controlled directly by
\[
|\cS_{5,1}| \lesssim |V_0^4V_3\partial_y(\theta^3)|+|\partial_y (V_0V^4)| \lesssim |s|^{-4}\omega \eta .
\]
For $\cS_{5,2}$, we expand $V^2=(V_1\theta+V_2\theta^2+V_2^*\partial_y\theta+V_3\theta^3)^2$, so
that
\begin{equation*}
V^2 -V_1^2\theta^2
 = 2 V_1V_2\theta^3 + 2V_1V_2^*\theta \partial_y \theta
+2V_1V_3\theta^4 + (V_2\theta^2+V_2^*\partial_y\theta+V_3\theta^3)^2.
\end{equation*}
Thus,
\begin{align*}
 \partial_y (V_0^3V^2- V_0^3V_1^2 \theta^2)
 & = 
2 \partial_y (V_0^3V_1V_2\theta^3) + 2\partial_y(V_0^3V_1V_2^*\theta \partial_y \theta)
+2\partial_y(V_0^3V_1V_3\theta^4)\\ 
&\quad+ \partial_y (V_0^3(V_2\theta^2+V_2^*\partial_y\theta+V_3\theta^3)^2),
\end{align*}
and
\begin{align*}
\cS_{5,2} & =
20 (V_0^3 V_1 V_2) \partial_y (\theta^3)
+ 20 (V_0^3V_1V_2^*) \partial_y(\theta\partial_y\theta)
+2V_0^3V_1V_3\partial_y(\theta^4)\\ 
&\quad+ \partial_y (V_0^3(V_2\theta^2+V_2^*\partial_y\theta+V_3\theta^3)^2).
\end{align*}
Observe that all the terms appearing in $\cS_{5,2}$ contain $V_0^3$. By using that $|V_0^3| \lesssim \omega \eta$, we deduce from  Lemmas \ref{le:vv} and \ref{le:th} that 
\[
|\cS_{5,2}|  \lesssim |s|^{-4}\omega \eta .
\]
Next, we decompose $\cS_{5,3}$ as 
\[\cS_{5,3}=10\partial_y(V_0^2V_1^3)\partial_y(\theta^3)+10\partial_y\left(V_0^2\left((V_1\theta+V_2\theta^2+V_2^*\partial_y\theta+V_3\theta^3)^3-V_1^3\theta^3\right) \right), \]
so that we deduce arguing similarly as for $\cS_{5,2}$ that 
\[
|\cS_{5,3}|  \lesssim |s|^{-4}\omega \eta .
\]
It remains to deal with $\cS_{5,4}$. We first write
\begin{align*}
    \cS_{5,4}=\partial_y(V^5)- V_1^5 \partial_y(\theta^5)
    =\partial_y(V^5- V_1^5 \theta^5) + \partial_y(V_1^5) \theta^5.
\end{align*}
By Lemma \ref{le:vv}, we have, for $p=0,1$,
\[
|\partial_y^p[\partial_y(V_1^5) \theta^5]|\lesssim |s|^{-5} (\omega\eta + |s|^{-1-p} \eta_L).
\]
Moreover,
\[
V^5-V_1^5\theta^5 = \sum_{l=1}^5 (V-V_1\theta)^l(V_1\theta)^{5-l}=
\sum_{l=1}^5 (V_2\theta^2+V_2^*\partial_y\theta+V_3\theta^3)^l(V_1\theta)^{5-l},
\]
and by using Lemma \ref{le:ww}, we deduce, for $l=1,\ldots,5$,
\[
|  [(V-V_1\theta)^l(V_1\theta)^{5-l}]| \lesssim |s|^{-2l}(\omega \eta+\eta_I)(|s|^{-(5-l)}\eta+\lambda^\frac{l}2\eta_R)
\lesssim |s|^{-5-l}(\omega \eta +\eta_I)
\]
and similarly, 
\[
|  \partial_y[(V-V_1\theta)^l(V_1\theta)^{5-l}]| 
\lesssim |s|^{-5-l}\omega \eta +|s|^{-6-l}\eta_I .
\]
Thus, for $p=0,1$,
\[
|\partial_y^p \cS_{5,4}|\lesssim |s|^{-6}\omega \eta + |s|^{-6-p}\eta_I.
\]
Therefore, we conclude the proof of \eqref{es:s5} gathering these estimates. 

\noindent\emph{Estimate of $\cS_6$.} We claim that 
 for any $p=0,1$,
\begin{align} 
|\partial_y^p\cS_6| &\lesssim C_1^\star|s|^{-4}(\log|s|)(1+y^4) \omega \eta+ |s|^{-4}(1+|y|)\eta_I
+C_1^\star|s|^{-2}(\log|s|)\eta_L; \label{es:s6.1} 
\end{align} 
We estimate each term of $\cS_6$ separately. 
Note that by its definition, $\cS_6=0$ for $y<-|s|$.
Based on the estimates for the functions based on $V_j$ appearing in the expression of $\cS_6$ in Lemma  \ref{le:vv}, and the estimates for the error terms $\Psi_j$ in Lemma \ref{le:35}, we have,
for $-|s|\leq y \leq |s|$, $p=0,1,$
\begin{align*} 
| \partial_y^p(\Lambda V_0 \Psi_1)| +|\partial_y^p(\Lambda V_0 \Psi_4)|+|\partial_y^p(\Lambda_3V_2^*\Psi_8)| & \lesssim C_1^\star |s|^{-4}(\log|s|) (1+|y|^3)\omega \eta ; \\
 \left|\partial_y^p\left(\partial_y^2V_2 \Psi_2\right)\right|+\left|\partial_y^p\left(V_0^3V_1^2 \Psi_2\right)\right|
& \lesssim C_1^\star |s|^{-4}(\log|s|) \left((1+|y|)\omega \eta+|s|^{-1}\eta_L\right) ;\\
\left|\partial_y^p\left(V_0^4V_2 \Psi_2\right)\right|+\left|\partial_y^p\left(\partial_y(V_0^3V_1V_2^*\Psi_2\right)\right|
& \lesssim C_1^\star |s|^{-4}(\log|s|) \left((1+|y|)\omega \eta+|s|^{-1}\eta_L\right)
\\
 |\partial_y^p(\partial_yV_1\Psi_3)| + |\partial_y^p(\partial_y^2V_2^*\Psi_3)|+|\partial_y^p(V_0^4V_2^*\Psi_3)| &\lesssim  C_1^\star |s|^{-4}(\log|s| ) \left((1+|y|)\omega \eta+ \eta_L\right) ; \\
|\partial_y^p(\Lambda_1 V_1\Psi_5)| & \lesssim C_1^\star |s|^{-4}(\log|s|) \left( (1+y^2)\omega \eta+|s|^2\eta_L\right) ; \\ 
|\partial_y^p(\Lambda_1 V_1 \Psi_6 )| & \lesssim   |s|^{-4} \left( (1+|y|)\omega \eta+ |s| \eta_L\right) ; \\ 
|\partial_y^p(\Lambda_2 V_2\Psi_7)| & \lesssim  |s|^{-4} \left((1+|y|) \omega \eta+ (1+|y|) \eta_I\right). 
\end{align*}

For $y>|s|$, we have $\sum_{j=1}^8 |\Psi_j| \lesssim |s|^{-\frac 12}\lesssim 1 $ by \eqref{eq:lt},
but by Lemma  \ref{le:vv}, all the functions based on $V_j$
which appear in $\cS_6$ happen to be exponentially decaying for $y>0$.
Thus, for $y>|s|$,
\[
|\cS_6|+|\partial_y \cS_6| \lesssim \omega \lesssim |s|^{-4} (1+y^4) \omega.
\]
We conclude the proof of \eqref{es:s6.1} gathering these estimates. 

\noindent\emph{Estimate of $\cS_7$.} We claim that 
\begin{equation} \label{es:s7}
|\cS_7|+|\partial_y \cS_7| \lesssim |s|^{-2}\eta_L.
\end{equation}
First, we remark that $\cS_7$ is identically zero for $y<-|s|$.
Second, for $y>-\frac{1}{2}|s|$, we have $\zeta(s,y)=1$ and thus again
$\cS_7$ is identically zero.
Therefore, we only have to consider the region $-|s|\leq y \leq -\frac{1}2|s|$.
Moreover, by the definition of $\zeta$ in \eqref{def:chiL}, one has
\begin{equation}\label{eq:ZT}
|\zeta|\lesssim 1,\quad |\partial_y^p \zeta|\lesssim |s|^{-p}.
\end{equation}

Now, we provide a general computation, for a function $A=A(y)$,
\begin{align*}
-\partial_y(\cLt (A\ZL)- \cL A)
&= 
-(1-\zeta)(A''-A)' - 5(1-\zeta^5) (Q^4 A)'\\
&\quad + \partial_y \zeta (3A''-A) + 25 \partial_y \zeta \zeta^4 Q^4 A 
+ 3\partial_y^2 \zeta A'+ \partial_y^3\zeta A   .
\end{align*}
Thus,
\begin{align*}
|\partial_y(\cLt (A\ZL)- \cL A)|
+|\partial_y^2(\cLt (A\ZL)- \cL A)|
\lesssim \sum_{l=1}^4 |A^{(l)}| +|s|^{-1} |A|.
\end{align*}
It follows from the properties of  $A_j$ (see \eqref{eq:aa}) 
and Lemma \ref{le:th} that
for $-|s|\leq y \leq -\frac 12 |s|$,
\begin{align*}
&\left|\theta  \left(-\partial_y(\cLt (A_1\ZL)- \cL A_1) \right)\right|
+ \left|\theta^2 \left(-\partial_y(\cLt (P\ZL)- \cL P) \right)\right| 
+ \left| \theta^2 \left(-\partial_y(\cLt (A_2\ZL) -\cL A_2) \right)\right| \\
&\quad+ \left| \partial_y\theta  \left(-\partial_y({\cLt} (A_2^*\ZL)-\cL A_2^*)  \right)\right| 
+\left|\theta^3 \left(-\partial_y(\cLt(A_3\ZL)-\cL A_3) \right)\right|
\lesssim |s|^{-2}.
\end{align*}
To estimate the remaining terms, we observe that 
$Q,F_2,F_2^*,F_3\in \cY$ while $\Omega\in\cZ_0$ (see \eqref{eq:FF}).
It follows that
\begin{align*}
\big|\theta  ( \Lambda V_0-\Lambda Q)\big|
+\big|\theta^2 (  \Lambda V_0-\Lambda Q)\big|
+\big|\theta^2 (\widetilde F_2-F_2)\big|
+\big| \partial_y\theta  ( \widetilde F_2^*-F_2)\big|
\lesssim |s|^{-100} \eta_L
\end{align*}
and
\[
\theta^3 |\widetilde F_3 - (c_2\Omega+ F_3)|\lesssim 
|s|^{-3} \eta_L.
\]
The estimates for the derivatives are similar, so that \eqref{es:s7}
is proved.

The proof of \eqref{eq:PW} follows from \eqref{es:s1}-\eqref{es:s7}.

\subsection{Norms of the blowup profile}
To complete the proof of Proposition \ref{pr:bp}, it only remains to prove (iii) and (iv),
concerning the mass, the $H^1$ norm, and the variation of the energy of $W$.
This is the objective of this subsection and the next one.

We have, using the definition of $W$, $V_0$ and $V$,
\begin{align*}
\int W^2-\int Q^2
&= \int (V_0+V)^2 - \int Q^2\\
&= \int V_0^2 -\int Q^2+ 2 \int V_0 V + \int V^2 \\
&= -\int (1-\ZL^2)Q^2 + 2 \int Q \ZL V + \int V^2,
\end{align*}
Thus, using \eqref{eq:vv} and \eqref{est:V:H1}, 
\[
\left| \int W^2-\int Q^2 \right|
\lesssim |s|^{-1} + \int V^2 \lesssim \delta^{2\alpha}.
\]
This is sufficient to prove \eqref{eq:ma}. 

Proceeding similarly as in the proof of \eqref{eq:ma}, and using $(A_1,Q'')=0$ by symmetry, we find
\begin{align*}
\int (\partial_y W)^2-\int (Q')^2
&= \int (\partial_yV_0+\partial_yV)^2 - \int (Q')^2\\
&= \int (\partial_y V_0)^2 -\int (Q')^2 - 2 \int (\partial_y^2 V_0) V + \int (\partial_yV)^2,
\end{align*}
Thus, using \eqref{eq:vv} and \eqref{est:V:H1}, 
\[
\left| \int (\partial_yW)^2-\int (Q')^2 \right|
\lesssim |s|^{-2},
\]
which proves \eqref{eq:h1}.

\subsection{Variation of the energy of the blowup profile}\label{s:3.9}

By time differentiation and integration by parts, we have
\begin{equation*}
\lambda^2 \frac{d}{ds}\left[ \frac {E(W)}{\lambda^2}\right]
=-\int \partial_sW \left( \partial_y^2W+W^5 \right) -2\frac{\lambda_s}{\lambda} E(W).
\end{equation*}
Using~\eqref{eq:EW} and \eqref{eq:ew} and integrating by parts, we get
\begin{align*}
\lambda^2 \frac{d}{ds}\left[ \frac {E(W)}{\lambda^2}\right]
&=-\frac{\lambda_s}{\lambda} \int \Lambda W \left( \partial_y^2W+W^5 \right)-2\frac{\lambda_s}{\lambda} E(W) \\ 
& \quad +\int \left( \frac{\lambda_s}{\lambda}+\beta\right) \left( \Lambda V_0+\Psi_{\lambda} \right) \left( \partial_y^2W+W^5\right) \\
& \quad +\left( \frac{\sigma_s}{\lambda}-1\right)\int \left(\partial_y V_0+\Psi_{\sigma} \right) \left( \partial_y^2W+W^5\right) \\
& \quad -\int \Psi_W \left( \partial_y^2 W+W^5 \right) \\ & 
:= e_1+e_2+e_3+e_4.
\end{align*}
First, integrating by parts,
$-\int \Lambda W \left( \partial_y^2W+W^5 \right) = 2E(W)$ and thus $e_1=0$.
(This cancellation is obviously related to the scaling of the energy.)

\smallskip 

\noindent\textit{Estimate for $e_2$}.
It follows from~\eqref{eq:Q5} that 
\begin{align*}
\partial_y^2 W + W^5 
& = \partial_y^2 V_0 + \partial_y^2 V + (V_0+V)^5\\
& = V_0 + 2 Q'\partial_y\ZL + Q \partial_y^2\ZL
+\partial_y^2 V + (V_0+V)^5 - Q^5\ZL.
\end{align*}
Thus, 
\[
|\partial_y^2 W + W^5 -V_0 |
\lesssim |s|^{-10}\eta_L + |\partial_y^2 V|
+ \omega \eta(|V|+|V|^4)+|V|^5
\]
and it follows from \eqref{eq:vv} that
\begin{equation}\label{eq:23}
\left|\partial_y^2 W + W^5 - V_0\right|\lesssim |s|^{-1} \omega\eta+|s|^{-3}\eta+\lambda^\frac52\eta_R.
\end{equation}
Therefore, by \eqref{eq:ls} and $\int V_0 \Lambda_0 V_0=0$, we obtain
\begin{align*}
\left| e_2\right|
&\lesssim  \left| \frac{\lambda_s}{\lambda}+\beta\right|
 \int (\omega +|s|^{-1}\eta_L+|s|^{-2}\eta_I) \left( |s|^{-1} \omega\eta+|s|^{-3}\eta+\lambda^\frac52\eta_R\right)\\
&\quad +\left| \frac{\lambda_s}{\lambda}+\beta\right| \int (|s|^{-1}\omega\eta+|s|^{-1}\eta_L+|s|^{-2}\eta_I)\omega \\
& \lesssimD |s|^{-1}\left| \frac{\lambda_s}{\lambda}+\beta\right| . 
\end{align*}

 \smallskip

\noindent \textit{Estimate for $e_3$}. 
Using \eqref{eq:ss} and $\int V_0 \partial_y V_0=0$, we prove similarly that
\begin{equation*}
\left| e_3\right| \lesssimD |s|^{-1}\left| \frac{\sigma_s}{\lambda}-1\right| . 
\end{equation*}

\smallskip

\noindent \textit{Estimate for $e_4$}. 
We note from~\eqref{eq:23} that
\[
|\partial_y^2W + W^5|\lesssim \omega\eta+|s|^{-3}\eta+\lambda^\frac52\eta_R.
\]
Therefore, using also~\eqref{eq:PW}
\begin{align*} 
\left| e_4 \right| &\lesssimD 
\int \left(C_1^\star |s|^{-4}(\log |s|)(1+|y|^4)\omega\eta+|s|^{-4}(1+|y|)\eta_I+C_1^\star|s|^{-2}(\log|s|)\eta_L\right) \\
&\qquad \times 
\left(\omega\eta+|s|^{-3}\eta+\lambda^\frac52\eta_R \right)\\
&\lesssimD C_1^\star |s|^{-4}(\log |s|)  .
\end{align*}
The proof of~\eqref{eq:en} follows from gathering these estimates.

\smallskip

The proof of Proposition~\ref{pr:bp} is now complete.

\subsection{From the blowup profile to the blowup residue}
\begin{lemma}\label{le:new}
For all $s\in \cI$,
\begin{equation}\label{eq:30}
\left| \int W^2- M_0\right|
\lesssim |s|^{-1}\quad 
\mbox{where $M_0 = \|Q\|_{L^2}^2 + (2\alpha+1)^2 m_0^2 \|\Theta_0\|_{L^2}^2$.}
\end{equation}
Moreover,
\begin{equation}\label{eq:33}
\int |W - Q - (2\alpha+1)m_0\theta_0|^2 \lesssim |s|^{-1},
\end{equation}
where $\theta_0(s,y) = \lambda^\frac12 \Theta_0(\lambda y + \sigma)$.
\end{lemma}
\begin{remark}\label{rk:39}
Estimate \eqref{eq:30} implies that the mass of $W$ converges to the fixed number $M_0$
for $|s|$ large.
Estimate \eqref{eq:33} means that the blowup profile $W$, once rescaled, is asymptotically in $L^2$
composed of the rescaled bubble plus the fixed chosen residue $\Theta_0$.
\end{remark}
\begin{proof}
We start by proving \eqref{eq:30}.
One has
\begin{align*}
\int W^2 &= \int V_0^2+ \int V_1^2\theta^2 
+2 \int V_0V_1\theta + \int (V_2\theta^2 + V_2^*\partial_y\theta+V_3\theta^3)^2\\
&\quad + 2 \int (V_2\theta^2 + V_2^*\partial_y\theta+V_3\theta^3)(V_0+V_1\theta),
\end{align*}
and thus
using the pointwise bound on $V_1\theta$ in \eqref{eq:mg} of Lemma \ref{le:ww}, 
\[
\left|\int W^2 - \int Q^2 - \int V_1^2 \theta^2 \right| \lesssim |s|^{-2}.
\]
Now, by using the definition of $V_1$ in \eqref{eq:v9} and recalling that $c_1=-(2\alpha+1)$, we have 
\begin{align*}
\int V_1\theta^2-(2\alpha+1)^2m_0^2\int \theta_0^2 =(2\alpha+1)^2\int \left(A_1^2 \zeta^2\theta^2-m_0^2\theta_0^2 \right),
\end{align*}
so that 
\begin{align*}
\left|\int V_1\theta^2-(2\alpha+1)^2m_0^2\int \theta_0^2\right| & \lesssim \int \left|A_1^2 -m_0^2\right| \zeta^2\theta^2+\int \left|\theta^2 -\theta_0^2 \right|\zeta^2+\int \theta_0^2\left| \zeta^2-1\right| .
\end{align*}
We observe from the properties of $A_1$ in Lemma \ref{le:PA} that $A_1^2-m_0^2 \in \mathcal{Y}$, so that we deduce from the bound \eqref{eq:lt} on $\theta$ that 
\begin{align*}
\int \left|A_1^2 -m_0^2\right| \zeta^2\theta^2\lesssim |s|^{-2} .
\end{align*}
In the support of $\zeta$, one has $y\ge -|s|$, so that it follows from \eqref{decomp:dktheta} that 
\begin{equation} \label{decomp:theta}
\theta=\theta_0+\theta_1+\lambda^{\frac12}q(\tau,\lambda y+\sigma), \quad \mbox{where} \quad \theta_1(s,y)=\lambda^{\frac12}\tau \Theta_1(\lambda y+\sigma). 
\end{equation}
 Thus, $\theta^2-\theta_0^2=(\theta_1+\lambda^{\frac12}q(\tau,\lambda y+\sigma))(\theta_0+\theta)$, and we deduce from the Cauchy Schwarz inequality and the $L^2$-bound on $\Theta$ in \eqref{theta:H1} that 
\begin{align*} 
\int \left|\theta^2 -\theta_0^2 \right|\zeta^2 
&\le \left(\|\theta\|_{L^2}+\|\theta_0\|_{L^2}\right)\left\|\left(\theta_1 +\lambda^{\frac12}q(\tau,\lambda \cdot+\sigma)\right)\zeta\right\|_{L^2}\\&
\lesssim \delta^{\alpha} \left(\|\theta_1 \zeta\|_{L^2}+\|\lambda^{\frac12}q(\tau,\lambda \cdot+\sigma) \zeta\|_{L^2}\right) .
\end{align*}
In the support of $\zeta$,we have $\lambda y+\sigma \ge \frac12 \rho \circ \tau$ (see \eqref{est:lambda y}), so that it follows by changing variable $x=\lambda y+\sigma$ and using \eqref{eq:L1} and \eqref{eq:B1}
\begin{equation} \label{est:q:L2}
\|\lambda^{\frac12}q(\tau,\lambda \cdot+\sigma) \zeta\|_{L^2} \le \left( \int_{x \ge \frac12 \rho(\tau)} q^2(\tau,x) dx\right)^{\frac12} \lesssim \tau^{2+\frac{\alpha-6}{4\alpha+3}} \lesssim |s|^{-\frac92} .
\end{equation}
Similarly, recalling that $\tau^{-\frac76\beta}(\lambda y+\sigma)$ on the support of $\zeta$ (see \eqref{est:cutoff}), we have 
using \eqref{eq:Tk}, 
\begin{align}
\|\theta_1 \zeta\|_{L^2}  \le \tau \left( \int_{x \ge \frac12 \rho(\tau)} \Theta_1^2(x) dx\right)^{\frac12} & 
\lesssim \tau \left( \int_{\frac12 \rho(\tau)}^{\delta}x^{2\alpha-7} dx+\delta^{2\alpha-6}\right)^{\frac12} \nonumber\\ & \lesssim \tau\left(\tau^{\frac{\alpha-3}{4\alpha+3}}+\delta^{\alpha-3} \right)\lesssim |s|^{-\frac52} .\label{est:theta1_L2}
\end{align}
To estimate the third integral, we change variable $x=\lambda y+\sigma$. Observe from the bootstrap hypotheses \eqref{eq:B2}-\eqref{eq:B3} that on the support of $\zeta^2-1$, \[x \le \sigma-\frac12\lambda|s| \le (\kappa_\sigma-\frac{\kappa_\lambda}2)|s|^{-\frac1{2\alpha}}+cC_1^\star|s|^{-\frac1{2\alpha}-1} \le c|s|^{-\frac1{2\alpha}} .\]
Thus,
\begin{equation} \label{est:theta0_zeta}
\int \theta_0^2\left| \zeta^2-1\right|\le \int_0^{c|s|^{-\frac1{2\alpha}}}  x^{2\alpha-1} dx \lesssim |s|^{-1} .
\end{equation}
Thus, we conclude gathering these estimates the proof of \eqref{eq:30}. 

To prove \eqref{eq:33}, we start observing by \eqref{eq:mg} and the definition of $V_0$, that
\[
\left\|W - Q - V_1 \theta\right\|_{L^2} \lesssim |s|^{-\frac32}.
\]
Now, using $c_1=-(2\alpha+1)$, we decompose 
\[V_1\theta-(2\alpha+1)m_0\theta_0=-(2\alpha+1)\left( (A_1+m_0)\zeta\theta+m_0\zeta(\theta_0-\theta)+m_0\theta_0(1-\zeta)\right),\]
so that 
\[\|V_1\theta-(2\alpha+1)m_0\theta_0\|_{L^2}\lesssim \|(A_1+m_0)\zeta \theta\|_{L^2}+\|\zeta (\theta-\theta_0)\|_{L^2}+\|\theta_0(1-\zeta)\|_{L^2}.\]
We estimate each term on the right-hand side separately. First, using $(A_1+m_0) \zeta \in \mathcal{Z}_0^-$ and \eqref{eq:lt}, we get
\[\|(A_1+m_0)\zeta \theta\|_{L^2}^2\lesssim |s|^{-2}\int \left(\omega \eta+\ONE_{[-|s| ,0]}\right)\lesssim |s|^{-1} .\]
Second, the decomposition \eqref{decomp:theta}, and then \eqref{est:q:L2}, \eqref{est:theta1_L2} yield
\[ \|\zeta (\theta-\theta_0)\|_{L^2} \le \|\zeta \theta_1\|_{L^2}+\|\lambda^{\frac12}q(\tau,\lambda \cdot+\sigma) \zeta\|_{L^2}\lesssim |s|^{-\frac52} .\]
Third, arguing as in the proof of \eqref{est:theta0_zeta}, we have 
\[\|\theta_0(1-\zeta)\|_{L^2}^2 \le \int_0^{c|s|^{-\frac1{2\alpha}}}  x^{2\alpha-1} dx \lesssim |s|^{-1} .\]
Therefore, we conclude the proof of \eqref{eq:33} gathering these estimates.

\subsection{Formal derivation of the blowup rate}\label{S:310}
Proposition \ref{pr:bp} (ii) and the definition of $\beta$ in \eqref{eq:bt} suggest to consider a blowup profile $W$
with parameters $(\lambda,\sigma)$ satisfying formally the coupled nonlinear system
\[
\frac{\lambda_s}{\lambda} - (2\alpha+1) \lambda^\frac12\sigma^{\alpha-\frac12}  + c_2 \lambda \sigma^{2\alpha-1} = 0 , \quad
\frac{\sigma_s}{\lambda} - 1 = 0.
\]
It turns out that we can approximately solve this system. Neglecting the term in $c_2$ in the equation of $\lambda_s$, which we anticipate to be of lower order,
we reduce ourselves to
\[
\lambda_s = (2\alpha+1) \lambda^\frac32\sigma^{\alpha-\frac12},\quad \sigma_s = \lambda.
\]
Set
$
g = \lambda^\frac12 - \sigma^{\alpha+\frac12} 
$
and compute
\begin{equation*}
g_s  = \frac 12 \lambda_s \lambda^{-\frac12} - \left( \alpha+\frac12\right) \sigma_s \sigma^{\alpha-\frac12} = 0 .
\end{equation*}
By integration, taking the integration constant to be $0$ as $|s|\to +\infty$, we find at the main order
$g(s)=0$ and so
$
\lambda^\frac12= \sigma^{\alpha+\frac12}.
$
This implies
$
\sigma_s = \lambda  = \sigma^{2\alpha+1}
$
and thus $\sigma^{-2\alpha}(s)=2\alpha |s|$ by integrating again.
This justifies the main order of $\sigma$ in \eqref{eq:B2} and then the main order of $\lambda$ in \eqref{eq:B3}
using $\lambda = \sigma^{2\alpha+1}$.

Now, we formally go back to the original time variable, writing (see \eqref{eq:ta})
\[
t= \int_{-\infty}^s \lambda^3(s') ds'
= \kappa_\lambda^3 \int_{-\infty}^s |s'|^{-\frac{6\alpha+3}{2\alpha}} ds'
= \frac {2\alpha}{4\alpha+3}\kappa_\lambda^3 |s|^{-\frac{4\alpha+3}{2\alpha}}.
\]
Thus, 
\begin{equation}\label{eq:ts}
(4\alpha+3) t = (2\alpha |s|)^{-\frac{4\alpha+3}{2\alpha}}
\end{equation}
and
$
\lambda(t) = \kappa_\lambda |s|^{-\frac{2\alpha+1}{2\alpha}}
= (2\alpha |s|)^{-\frac{2\alpha+1}{2\alpha}}
=(4\alpha+3)^\nu t^\nu.
$
\end{proof}

\section{Modulation close to the blowup profile}\label{S:4}

\subsection{Refined blow-up profile} 
We introduce a refined approximate blow-up profile $W_b$ close to $W$ by involving a additional small parameter $b$.
For a $\cC^1$ real-valued function $b:\cI\to \RR$ such that
\begin{equation}\label{wBTb}
|b| < \delta^{2\alpha},
\end{equation}
we define the function $W_b(s,y)=W_b(y;\lambda(s),\sigma(s),b(s))$ by
\begin{equation} \label{def:W:b}
W_b(s,y) = W(s,y) + b(s) P_b(s,y)
\end{equation}
where
\begin{equation} \label{def:P:b}
P_b(s,y) = \chi_b(y) P(y),\quad \chi_b(s,y) = \chi_L(|b(s)|^\gamma y),\quad \gamma = \frac 34, 
\end{equation}
and where the function $\chi_L$ is defined in \eqref{eq:CL}.
The next lemma provides estimates for the refined profile $W_b$.
Let
\begin{equation} \label{def:Psi:b}
\Psi_b= \partial_y \left( b\partial_y^2 P_b - b P_b + W_b^5-W^5 \right) + b( \Lambda V_0 +   \beta \Lambda P_b).
\end{equation}
The motivation for introducing $\Psi_b$ is to complement the error term $\Psi_W$
in the rescaled equation when passing from $W$ to $W_b$.

\begin{lemma}  For all~$s\in\cI$.
\begin{enumerate}
\item[(i)]  \emph{Pointwise estimates for $W_b$.} For any $p=1,2,3,4$, for any $y \in \RR$,
\begin{align} 
|W_b| &\lesssim \omega\eta+|s|^{-1}\eta + \lambda^\frac12\eta_R
+|b| \ONE_{[-1,0]}(|b|^{\gamma}y), \label{est:W0:b} \\ 
|\partial_y^pW_b|
& \lesssim\omega\eta+|s|^{-1-p}\eta + \lambda^{\frac12+p}\eta_R+|b|^{1+\gamma p}\ONE_{[-1,-\frac12]}(|b|^{\gamma}y).\label{est:W12:b}
\end{align}
\item[(ii)]  \emph{Error term in the equation of $W_b$.}
For $p=0,1$, for all $y \in \RR$,
\begin{equation} \label{est:Psi0:b}
|\partial_y^p \Psi_b| \lesssimD (b^2+|b||s|^{-1})\omega \eta
+|b||s|^{-1} \ONE_{[-1,0]}(|b|^{\gamma}y)+ |b|^{1+\gamma(1+p)}\ONE_{[-1,-\frac12]}(|b|^{\gamma}y).
\end{equation}
Moreover, 
\begin{equation} \label{est:Psib:Q}
\left|(\Psi_b,Q)-2c_1(2\alpha)^{-1}m_0^2b|s|^{-1} \right|
\lesssimD b^2+|b| |s|^{-2} .
\end{equation}
\item[(iii)]  \emph{Mass and Energy of $W_b$.}
\begin{align}
 \left| \int W_b^2-\int W^2 \right| &\lesssimD |b|+|b|^{1-\gamma}|s|^{-1}, \label{mass:W:b}\\ 
 \left| \int (\partial_yW_b)^2-\int (\partial_yW)^2 \right|&\lesssimD |b|, \label{h1:W:b}\\
 \left| {E(W_b)}- {E(W)}\right| &\lesssimD |b|+|b|^{1-\gamma}|s|^{-5} . \label{energy:W:b}
\end{align}
\end{enumerate}
\end{lemma}
\begin{proof}
For the reader's convenience, we repeat the proof of \cite[Lemma 3.1]{MP24}.

(i) The estimates \eqref{est:W0:b}-\eqref{est:W12:b} follow directly from the estimates \eqref{eq:w0} for $W$ and the definitions of $\chi_L$, $W_b$ and $P_b$ in \eqref{eq:cL}, \eqref{def:W:b} and \eqref{def:P:b}. 

(ii) Expanding $W_b=W+bP_b$ and using $(\cL P)' = \Lambda Q$ in the definition of $\Psi_b$, we find 
\begin{equation}\label{eq:Psib}
\begin{aligned}
\Psi_b & = b (1-\chi_b) \Lambda Q +b(\Lambda V_0-\Lambda Q)
+ b \beta \Lambda P_b\\
&\quad + 3 b (\partial_y \chi_b) P'' + 3 b (\partial_y^2 \chi_b) P'
+b (\partial_y^3 \chi_b) P - b (\partial_y \chi_b) P + 5 b(\partial_y \chi_b) Q^4 P
\\&\quad + b\partial_y \left(5(W^4-V_0^4)P_b+5(V_0^4-Q^4)P_b  \right) \\ & \quad + b\partial_y \left(10 b W^3 P_b^2 +10 b^2 W^2 P_b^3 + 5 b^3 W P_b^4 + b^4 P_b^5 \right).
\end{aligned}
\end{equation}
We estimate separately each term on the right-hand side of \eqref{eq:Psib}. First, from the properties of $\chi_b$ and $\Lambda Q \in \cY$, for $p=0,1$,
\begin{equation*} 
\left|b\partial_y^p \left((1-\chi_b) \Lambda Q\right) \right| \lesssim |b|e^{-\frac 34|y|} \ONE_{(-\infty,-1]}(|b|^{\gamma}y) \lesssim |b|e^{-\frac 14 |b|^{-\gamma}} \omega 
\end{equation*}
and 
\begin{equation*}
\left|b\partial_y^p\left(\Lambda V_0-\Lambda Q \right) \right|=
\left|b\partial_y^p\left(\Lambda Q(\zeta-1)+yQ\partial_y\zeta \right) \right| \lesssim |b||s|^{-100} \omega .
\end{equation*}
Next, we compute 
\begin{equation} \label{Lambda:P:b}
 \Lambda P_b = (\Lambda P)\chi_b+y(\partial_y \chi_b) P,
\end{equation}
and we estimate
\begin{equation} \label{Lambda:P:b.1}
\left|\Lambda P_b \right| \lesssim \omega+\ONE_{[-1,0]}(|b|^{\gamma}y) \quad \text{and} \quad \left|\partial_y\Lambda P_b \right| \lesssim \omega+|b|^{\gamma}\ONE_{[-1,-\frac12]}(|b|^{\gamma}y) .
\end{equation}
Thus, it follows from the definition of $\beta$ in \eqref{eq:bt} and \eqref{eq:lt} that
\begin{align*}
 \left|b \beta \Lambda P_b \right|
 &\lesssim |b| |s|^{-1}\omega +|b| |s|^{-1} \ONE_{[-1,0]}(|b|^{\gamma}y), \\
 \left|b \beta \partial_y(\Lambda P_b) \right|
 &\lesssim |b| |s|^{-1}\omega+|b|^{1+\gamma} |s|^{-1} \ONE_{[-1,-\frac12]}(|b|^{\gamma}y).
 \end{align*}
Next, we deduce from \eqref{def:P:b} and $P \in \cZ_0$ that, for $p=0,1$, 
\begin{align*}
&\left| b \partial_y^p\left((\partial_y \chi_b) P''\right)\right| +\left| b \partial_y^p\left((\partial_y^2 \chi_b) P'\right) \right|
+\left| b\partial_y^p\left((\partial_y \chi_b) Q^4 P\right) \right|\lesssim e^{-\frac 14|b|^{-\gamma}} \omega, \\
&\left|b \partial_y^p\left((\partial_y^3 \chi_b) P\right) \right|+\left|b \partial_y^p\left((\partial_y \chi_b) P\right) \right|
\lesssim |b|^{1+(1+p)\gamma}\ONE_{[-1,-\frac12]}(|b|^{\gamma}y) .
\end{align*}
From the definition of $V_0=Q\zeta$, we have, for $p=0,1,2$, 
\begin{equation*}
\left|\partial_y^p(V_0^4-Q^4)\right|=\left|\partial_y^p(Q^4(\zeta^4-1))\right| \lesssim \omega^2 \ONE_{(-\infty,-\frac12]}(|s|^{-1}y) ,
\end{equation*}
so that 
\begin{equation*}
\left|b\partial_y^{1+p}\left((V_0^4-Q^4)P_b\right)\right| \lesssim |b||s|^{-100}\omega \ONE_{(-1,0]}(|b|^{\gamma}y) .
\end{equation*}
Expanding $W=V_0+V$, we obtain
$W^4-V_0^4=4V_0^3V+6V_0^2V^2+4V_0V^3+V^4$.
By~\eqref{eq:vv}, for $p=0,1,2$,
\begin{align*}
& |\partial_y^p (V_0^3 V)|+ |\partial_y^p (V_0^2 V^2)| + |\partial_y^p (V_0 V^3)
|\lesssim |s|^{-1} \omega,\\
& |\partial_y^p (V^4)|\lesssim |s|^{-4} \eta +\lambda^2 \eta_R. 
\end{align*}
Thus, for $p=0,1$,
\begin{equation*}
\big|b \partial_y^{1+p} \big((W^4-V_0^4)P_b\big)\big|
 \lesssimD |b||s|^{-1} \omega +|b||s|^{-4}\ONE_{[-1,0]}(|b|^{\gamma}y) 
  + |b|^{1+\gamma}|s|^{-4}\ONE_{[-1,-\frac12]}(|b|^{\gamma}y).
\end{equation*}
Last, by using \eqref{eq:w0} and \eqref{def:P:b}, we have, for $p=0,1$,
\begin{align*} 
\left| b^2\partial_y^{1+p}\left(W^3 P_b^2\right) \right|
&\lesssimD |b|^2\omega\eta+|b|^2|s|^{-5-p} \ONE_{[-1,0]}(|b|^{\gamma}y)\\
&\quad +|b|^{2+\gamma}|s|^{-3}(|s|^{-p}+|b|^{p\gamma}) \ONE_{[-1,-\frac12]}(|b|^{\gamma}y), \\
\left| b^3\partial_y^{1+p}\left(W^2 P_b^3\right) \right|
&\lesssimD |b|^3\omega\eta+|b|^3|s|^{-4-p} \ONE_{[-1,0]}(|b|^{\gamma}y)\\
&\quad 
+|b|^{3+\gamma}|s|^{-2}(|s|^{-p}+|b|^{p\gamma}) \ONE_{[-1,-\frac12]}(|b|^{\gamma}y), \\
\left| b^4\partial_y^{1+p}\left(W P_b^4\right) \right|
&\lesssimD |b|^4\omega\eta+|b|^4|s|^{-3-p} \ONE_{[-1,0]}(|b|^{\gamma}y)\\
&\quad 
+|b|^{4+\gamma}|s|^{-1}(|s|^{-p}+|b|^{p\gamma})\ONE_{[-1,-\frac12]}(|b|^{\gamma}y),\\
\left| b^5 \partial_y^{1+p}(P_b^5) \right|
&\lesssimD |b|^5\omega\eta+|b|^{5+(1+p)\gamma}\ONE_{[-1,-\frac12]}(|b|^{\gamma}y).
\end{align*}
We deduce~\eqref{est:Psi0:b} by gathering these estimates. 

Proof of~\eqref{est:Psib:Q}. 
Taking the scalar product of the expression of $\Psi_b$ in \eqref{eq:Psib} with $Q$
and using the same estimates as before,
we observe that the two terms $b \beta \Lambda P_b$ and $20b\partial_y (V_0^3V_1\theta P_b)$ give contributions of size $b|s|^{-1}$,
while all the other terms give contributions of size at most $b^2+|b||s|^{-2}$.
First, we observe
\begin{align*}
\left|\left(\beta \Lambda P_b-c_1(2\alpha)^{-1}|s|^{-1}\Lambda P, Q\right) \right|
&\lesssim |\beta| \left|\left(P_b- P,\Lambda Q\right) \right|
+| \beta - c_1(2\alpha)^{-1}|s|^{-1}| |( P, \Lambda Q)|\\
&\lesssim |s|^{-1} e^{-\frac 14|b|^{-\gamma}} + |s|^{-1}e^{-\frac{\delta}{4\lambda}}+ |s|^{-2}
\lesssimD |s|^{-1} b^2 + |s|^{-2} .
\end{align*}
Second, we compute from \eqref{eq:v9}, \eqref{eq:lt} and \eqref{eq:0T} that 
\begin{align*}
 & \left|\left(\partial_y (V_0^3V_1\theta P_b),Q\right)
+c_1(2\alpha)^{-1} |s|^{-1}\big(Q^3A_1P,Q'\big)\right| 
\\ & \quad \lesssim |s|^{-1} e^{-\frac 14|b|^{-\gamma}}+|s|^{-1}e^{-\frac{|s|}2}+ |s|^{-1} e^{-\frac\delta{4\lambda}}\lesssimD |s|^{-1} b^2 + |s|^{-2} .
\end{align*}
Thus, \eqref{est:Psib:Q} follows from the identities \eqref{eq:PQ} and \eqref{eq:id}.

(iii) Expanding $W_b=W+bP_b$, we obtain
\begin{equation*}
 \int W_b^2=\int W^2+2b \int WP_b+b^2\int P_b^2 .
\end{equation*}
Observe from \eqref{eq:w0} and \eqref{def:P:b} that 
\begin{equation*}
\left| b \int WP_b \right| \lesssim |b|+|b|^{1-\gamma}|s|^{-1} ,\quad
b^2 \int P_b^2 \lesssim |b|^{2-\gamma}.
\end{equation*}
Hence, \eqref{mass:W:b} is proved.
Expanding $W_b=W+bP_b$ in the definition of the energy, we obtain 
\begin{equation*} 
E(W_b)=E(W)+b\int (\partial_yW) (\partial_yP_b)+\frac12b^2\int (\partial_yP_b)^2 -\frac16 \int \left( (W+bP_b)^6-W^6 \right).
\end{equation*}
Observe from \eqref{eq:w0} and \eqref{def:P:b} that 
\begin{gather*}
\left| b\int (\partial_yW) (\partial_yP_b) \right| \lesssim |b| ,\quad
b^2 \int (\partial_yP_b)^2 \lesssim b^2, \\
\left|\int \left( (W+bP_b)^6-W^6 \right) \right| \lesssim 
|b|+|b|^{1-\gamma}|s|^{-5} .
\end{gather*}
Hence, \eqref{energy:W:b} is proved. 
Estimate \eqref{h1:W:b} is proved similarly.
\end{proof}

\subsection{Choice of the modulation parameters}
Let 
\[
\mbox{$\cI=[S,s_0]$ where $S<s_0<0$ and $|s_0| \gg 1$.}
\]
We look for $(w(s,y),\lambda(s),\sigma(s))$ solution of
the rescaled equation~\eqref{rescaled} on $\cI$ 
such that $w$ has the form
\[
w(s,y) = W_{b}(s,y) + \varepsilon(s,y), 
\]
where the parameters $\tau$, $\lambda$, $\sigma$, $b$ satisfy \eqref{eq:B1}--\eqref{eq:B3}, \eqref{wBTb}
and the function $\varepsilon$ satisfies
\begin{equation}\label{eq:H1b}
\|\varepsilon(s)\|_{H^1} \leq \delta^\frac 14
\end{equation}
and the orthogonality relations
\begin{equation}\label{eq:or}
(\varepsilon(s),\Lambda Q)=(\varepsilon(s),y\Lambda Q)=(\varepsilon(s),Q)=0.
\end{equation}
Moreover, we assume that 
\begin{equation} \label{w:0}
w(S) \equiv W(S) \iff b(S)=0, \, \varepsilon(S)\equiv 0.
\end{equation}
In this context, the existence and uniqueness of $\cC^1$ functions $\lambda$, $\sigma$, $b$ ensuring the orthogonality relations~\eqref{eq:or} follow from standard arguments based on the implicit function theorem.
In the case where the blow-up profile is $Q$, this is justified in~\cite[Lemma~2.5]{MMR1} and the references therein.
The same proof applies to the refined ansatz $W$.
The orthogonality relations \eqref{eq:or} are taken from~\cite{MM1}  in order to ensure the positivity of a quadratic form
related to a virial argument (see Lemma~\ref{le:virial}).

We derive the equation of $\varepsilon$. First, we check from \eqref{rescaled} that
\begin{equation}\label{eq:eps} 
\partial_s \varepsilon + \partial_y \left[\partial_y^2\varepsilon-\varepsilon+\left((W_b+\varepsilon)^5-W_b^5\right)\right]
-\frac{\lambda_s}{\lambda} \Lambda\varepsilon - \left( \frac{\sigma_s}{\lambda}-1\right) \partial_y \varepsilon
+ \cE_b(W) = 0
\end{equation}
where (see \eqref{eq:EW} and \eqref{def:Psi:b} for the definitions of $\cE(W)$ and $\Psi_b$)
\begin{equation*}
\cE_b(W) = \cE(W) - b \Lambda V_0 
-\left(\frac{\lambda_s}{\lambda}+\beta\right) \Lambda (bP_b)
- \left( \frac{\sigma_s}{\lambda} - 1 \right) (b\partial_y P_b)
+ \partial_s (bP_b)+ \Psi_b.
\end{equation*}
Next, using \eqref{eq:ew}, we rewrite
\begin{equation}\label{def:EbW}
\cE_b(W) = - \vec{m} \cdot \vec{M}V_0+ \Psi_M + \Psi_{W} + \Psi_b, 
\end{equation}
where
\begin{equation} \label{def:mM}
 \vec m =\begin{pmatrix} \frac{\lambda_s}{\lambda}+\beta+b \\ \frac{\sigma_s}{\lambda}-1 \end{pmatrix},
 \quad \vec M =\begin{pmatrix} \Lambda \\ \partial_y \end{pmatrix},
\end{equation}
and 
\begin{equation}\label{def:PsiM}
\begin{split}
\Psi_M &= -\left(\frac{\lambda_s}{\lambda}+\beta+b\right) \left( \Psi_{\lambda}+b\Lambda P_b\right)
- \left( \frac{\sigma_s}{\lambda} - 1 \right) (\Psi_{\sigma}+b \partial_y P_b) \\ & \quad
+b\left(\Psi_{\lambda}+b\Lambda P_b\right)+ b_s \left(\chi_b + \gamma y \partial_y \chi_b \right) P.
\end{split}
\end{equation}

\begin{lemma}[Modulation estimates] 
For any $s \in \cI$,
\begin{align}
|\vec m|:=\left| \frac{\lambda_s}{\lambda} +\beta +b\right|+
\left| \frac{\sigma_s}{\lambda} - 1 \right|& \lesssimD \|\varepsilon\|_{L^2_\loc}+C_1^\star |s|^{-4} \log|s|+|b||s|^{-1}+b^2,
 \label{est:lambda:sigma}\\
\left| b_s+2c_1b(2\alpha)^{-1} |s|^{-1}\right|& \lesssimD \|\varepsilon\|_{L^2_\loc}^2+\|\varepsilon\|_{L^2_\loc} (|s|^{-1}+|b|)\nonumber\\
&\quad +C_1^\star |s|^{-4} \log|s|+C_1^\star|b||s|^{-2}\log |s| +b^2 .\label{est:b_s}
\end{align}
\end{lemma}
\begin{proof}
By using the definition of $\cL$ in \eqref{eq:cL}, we rewrite the equation of $\varepsilon$ as
\begin{equation}\label{eq:eps.2} 
\partial_s \varepsilon - \partial_y\cL \varepsilon
-\vec{m} \cdot \vec{M}(V_0+\varepsilon)+(\beta+b)\Lambda \varepsilon+\partial_y\cR_1+\partial_y\cR_2+\Psi_M+\Psi_W+\Psi_b=0
\end{equation}
where 
\begin{equation*}
\cR_1=(W_b+\varepsilon)^5-W_b^5-5W_b^4\varepsilon ,\quad 
\cR_2=5(W_b^4-Q^4)\varepsilon .
\end{equation*}
We start by general estimates on the error terms in \eqref{eq:eps.2}.
We claim that, for any $f \in \cY$, 
\begin{align} 
\left|\left(\partial_y \cR_1,f\right) \right|
&\lesssim \|\varepsilon\|_{L^2_\loc}^2 , \label{est:mod.1}\\
\left|\left(\partial_y \cR_2,f\right) \right|
&\lesssim \left(|s|^{-1}+|b|\right)\|\varepsilon\|_{L^2_\loc}, \label{est:mod.2}\\
\left|\left(\Psi_{M},f\right) \right|
& \lesssimD |b_s|+(|\vec{m}|+b) (|s|^{-1}+|b|)+|s|^{-100}, \label{est:mod.3}\\
\left|\left(\Psi_{W},f\right) \right| & \lesssimD C_1^\star |s|^{-4} \log |s|, \label{est:mod.4}\\ 
\left|\left(\Psi_{b},f\right) \right| &\lesssimD |b|(|s|^{-1}+|b|). \label{est:mod.5}
\end{align}
where the implicit constants depend on $f$.
Indeed, the estimate~\eqref{est:mod.1} is a consequence of 
$|\cR_1|\lesssim |W_b|^3 \varepsilon^2 + |\varepsilon|^5\lesssim |\varepsilon|^2$ by~\eqref{est:W0:b} and
$\|\varepsilon\|_{L^\infty}^2\lesssim \|\partial_y \varepsilon\|_{L^2}\|\varepsilon\|_{L^2}\lesssim 1$.
The estimate~\eqref{est:mod.2} follows from integration by parts $\left(\partial_y \cR_2,f\right)= -\left((W_b^4-Q^4)\varepsilon, f'\right)$
and $|W_b^4-Q^4| \lesssim |s|^{-1}+|b|$
using \eqref{eq:vv} and \eqref{def:W:b}.
The estimates~\eqref{est:mod.3}-\eqref{est:mod.5} follow from \eqref{eq:lt}, \eqref{eq:ls}, \eqref{eq:ss}, \eqref{eq:PW}, \eqref{est:Psi0:b}, \eqref{Lambda:P:b} and \eqref{def:PsiM}. 

Now, using the first orthogonality relation in \eqref{eq:or}, \eqref{eq:eps.2} and $(Q',\Lambda Q)=0$, we obtain 
\begin{align*}
\left|\left( \frac{\lambda_s}{\lambda}+\beta+b\right)-
\frac{(\varepsilon,\cL(\Lambda Q)')}{\|\Lambda Q\|_{L^2}^2} \right|
&\lesssimD (|\vec{m}|+ \|\varepsilon\|_{L^2_\loc}) \left(|s|^{-1}+|b|+ \|\varepsilon\|_{L^2_\loc}\right)+\left|\left(\partial_y \cR_1,\Lambda Q\right) \right|\\ 
&\quad +\left|\left(\partial_y \cR_2,\Lambda Q\right) \right|
 +\left|\left(\Psi_M,\Lambda Q\right) \right|
+\left|\left(\Psi_{W},\Lambda Q\right) \right|+\left|\left(\Psi_b,\Lambda Q\right) \right| .
\end{align*}
Hence, using \eqref{est:mod.1}-\eqref{est:mod.5} with $f=\Lambda Q$,
\begin{equation} \label{est:lambda.2}
\left| \frac{\lambda_s}{\lambda}+\beta+b \right| \lesssimD
\|\varepsilon\|_{L^2_\loc}+ |\vec{m}| \left(|s|^{-1}+|b|+ \|\varepsilon\|_{L^2_\loc}\right)+|b_s| +C_1^\star |s|^{-4} \log|s|+|b||s|^{-1}+b^2 .
\end{equation}
Similarly, by the second orthogonality relation in \eqref{eq:or} and $(Q',y\Lambda Q)=\|\Lambda Q\|_{L^2}^2$, we obtain
\begin{equation} \label{est:sigma.2}
\left| \frac{\sigma_s}{\lambda}-1\right|
\lesssimD
\|\varepsilon\|_{L^2_\loc}+ |\vec{m}| \left(|s|^{-1}+|b|+ \|\varepsilon\|_{L^2_\loc}\right)+|b_s| +C_1^\star |s|^{-4} \log|s|+|b||s|^{-1}+b^2 .
\end{equation}

Last, using the third orthogonality relation in \eqref{eq:or} and \eqref{eq:eps.2}, we have
\begin{align*}
0=(\partial_s\varepsilon,Q )&=-\left( \partial_y\cL\varepsilon,Q\right)-\left(\vec{m} \cdot\vec M (Q+\varepsilon) , Q \right)+(\beta+b) (\Lambda \varepsilon,Q)\\ 
&\quad + \left(\partial_y \cR_1, Q\right) +\left(\partial_y \cR_2, Q\right)
+(\Psi_M,Q)+(\Psi_W,Q)+(\Psi_b,Q) .
\end{align*}
We observe the special cancellations
$\left(\partial_y\cL\varepsilon,Q\right)=-\left( \varepsilon,\cL Q'\right)=0$ and $(\Lambda \varepsilon,Q)=-(\varepsilon,\Lambda Q)=0$.
Moreover, it follows from the identities $(\Lambda Q, Q)=(Q',Q)=0$ that 
\begin{equation*} 
\left| \left(\vec{m} \cdot\vec M (Q+\varepsilon) , Q \right) \right| \lesssimD |\vec{m}| \|\varepsilon\|_{L^2_\loc}.
\end{equation*}
Recalling $m_0^2=(P,Q)>0$ (see \eqref{eq:PQ}) and arguing as for the proof of \eqref{est:mod.3} but using \eqref{est:Psi_lambda:Q} instead of \eqref{eq:ls}, we obtain
\begin{equation*} 
\left|(\Psi_M,Q)-m_0^2 b_s \right| \lesssimD |\vec{m}|(|s|^{-1}+|b|)+C_1^*|b||s|^{-2}\log|s|+b^2+|b_s||b|^{100}. 
\end{equation*}
Then, we deduce combining these estimates with \eqref{est:Psib:Q} and \eqref{est:mod.1}-\eqref{est:mod.4} with $f=Q$ that 
\begin{equation} \label{est:b_s.2}
\begin{aligned}
\left| b_s+2c_1b(2\alpha)^{-1} |s|^{-1}\right| 
&\lesssimD (|\vec{m}|+\|\varepsilon\|_{L^2_\loc}) \left(|s|^{-1}+|b|+ \|\varepsilon\|_{L^2_\loc}\right)
\\&\quad +C_1^\star |s|^{-4} \log|s|+C_1^\star|b| |s|^{-2} \log|s|+b^2 ++b^2+|b_s||b|^{100}.
\end{aligned}
\end{equation}
Finally, we conclude the proof of \eqref{est:lambda:sigma} combining \eqref{est:lambda.2}, \eqref{est:sigma.2} and \eqref{est:b_s.2},
and taking~$\delta$ small enough. The proof of \eqref{est:b_s} then follows from \eqref{est:lambda:sigma} and \eqref{est:b_s.2}. 
\end{proof}

\section{The bootstrap setting}\label{S:5}

\subsection{Definition of a sequence of solutions}\label{S:5.1}

The time variables $t$ and $s$ being related approximately by \eqref{eq:ts}, it is natural
to introduce, for $n$ large
\begin{equation}\label{eq:TS}
T_n = \frac 1{4\alpha+3} (2\alpha n)^{-\frac{4\alpha+3}{2\alpha}},\quad
S_n = -n.
\end{equation}
We denote by $W_{b,n}(t,y)$ the function $W_b(y;\tau_n(t),\lambda_n(t),\sigma_n(t),b_n(t))$ defined in Section~\ref{S:4}
for $\tau_n(t)$, $\lambda_n(t)$, $\sigma_n(t)$ and~$b_n(t)$ to be chosen.
We define the solution $u_n$ of \eqref{eq:KV} with initial data at $T_n$
\begin{equation}\label{eq:IN}
u_n(T_n) = \frac 1{\lambda_n^\frac 12(T_n)} W_{b,n}\left(T_n, \frac{x-\sigma_n(T_n)}{\lambda_n(T_n)}\right)
\end{equation}
where $b_n(T_n)= 0$, $\tau_n(T_n)=T_n=\kappa_\tau |S_n|^{-\frac{4\alpha+3}{2\alpha}}$ and where $\sigma_n(T_N)$ and $\lambda_n(T_n)$ are uniquely chosen so that
\[
\frac 1{2\alpha} \sigma_n^{-2\alpha}(T_n) - c_2 \log \sigma_n (T_n) = n
\]
and
\[
\lambda_n^\frac12(T_n) - \sigma_n^{\alpha+\frac 12}(T_n)  + \frac{c_2}{2} \frac1{3\alpha+\frac12} \sigma_n^{3\alpha+\frac12}(T_n)=0.
\]
(See below \eqref{eq:dg} and \eqref{eq:cs}  where these choices are used.)
In particular, it is easy to check that
\begin{equation} \label{at_Tn}
\lambda_n(T_n)\approx \kappa_\lambda n^{-\frac{2\alpha+1}{2\alpha}} ,\quad
\sigma_n(T_n) \approx \kappa_\sigma n^{-\frac1{2\alpha}}.
\end{equation}
We consider the solution $u_n$ for times $t\geq T_n$. As long as it exists and remains close to $Q$ up to rescaling and translation,
we can decompose it as in Section \ref{S:4}
\begin{equation}\label{eq:SD}
u_n(t,x) =\frac 1{\lambda_n^{\frac 12}(t)}w_n\left(t, \frac{x-\sigma_n(t)}{\lambda_n(t)}\right)= \frac 1{\lambda_n^{\frac 12}(t)} \left( W_{b,n}\left(t, \frac{x-\sigma_n(t)}{\lambda_n(t)}\right)
+ \varepsilon_n\left(t,\frac{x-\sigma_n(t)}{\lambda_n(t)}\right)\right),
\end{equation}
where $\varepsilon_n$ satisfies \eqref{eq:or}.
At $t=T_n$, this decomposition satisfies  $\varepsilon_n(T_n)\equiv 0$.

We define the rescaled time variable $s$
\begin{equation}\label{defts}
s=s(t) = S_n + \int_{T_n}^t \frac{dt'}{\lambda_n^3(t')},
\end{equation}
which is equivalent to \eqref{eq:ta} with $T=T_n$ and $S=S_n$.
From now on, any time-dependent function will be seen either as a function of $t$ or as a function of $s$.

\subsection{Notation for the energy-virial functional}

Let $\psi,\varphi\in\cC^\infty$ be non decreasing functions such that
\begin{equation}\label{eq:5P}
\psi(y) = \left\{ \begin{aligned} e^y \quad & \mbox{for } y < -1, \\ 1 \quad & \mbox{for } y > -\frac 12, \end{aligned} \right.
\qquad \mbox{and} \qquad
\varphi(y) = \left\{ \begin{aligned} e^y \quad & \mbox{for } y < -1, \\ y + 1 \quad 
& \mbox{for } y>-\frac 12.\end{aligned} \right.
\end{equation}
We note that such functions satisfy $\frac 12 e^y \leq \psi(y) \leq 3e^y$
and $\frac 13 e^y \leq \varphi(y) \leq 3e^y$ for $y<0$,
and $\psi(y) \leq 3\varphi(y)$ for all $y\in\RR$.

In previous articles using a similar virial-energy functional such as the one
introduced in Section \ref{S:6} (see \cite{CoM1,CoM2,MMR1,MMR2,MMR3,MP20,MP24}), various functions 
$\psi$ and $\varphi$ have been used, depending on the context.
In particular, the choice here slightly differs
from the one in \cite{MP24}. We justify the particular choice \eqref{eq:5P} in Remark \ref{rk:pb} below.

For $B > 100$ large to be chosen later we define
\[
\psi_B(y) = \psi\left( \frac yB \right) \quad\mbox{and}\quad
\varphi_B(y) = \varphi\left( \frac yB\right).
\]
From the properties of $\psi$ and $\varphi$, we have for all $y<0$,
\begin{equation} \label{cut:left}
\frac 12 e^{\frac yB} \leq \psi_B(y) \leq 3 e^{\frac yB}, \quad
\frac 13 e^{\frac yB} \leq \varphi_B(y) \leq 3 e^{\frac yB},
\end{equation}
and for all $y\in\RR$,
\begin{equation} \label{cut:onR}
\psi_B'(y) + B^2|\psi_B'''(y)| + B^2|\varphi_B'''(y)| \lesssim \varphi_B'(y).
\end{equation}
We set
\[
\cN_B(\varepsilon_n) = \left(\int \left( (\partial_y \varepsilon_n)^2 + \varepsilon_n^2 \right) \varphi_B \right)^{\frac 12}
\]
and we observe 
\begin{equation} \label{eL2B}
\|\varepsilon_n\|_{L^2_\loc} \lesssim \left( B \int \varepsilon_n^2\varphi_B' \right)^{\frac 12} \lesssim \cN_B(\varepsilon_n).
\end{equation}

\subsection{Bootstrap estimates}
For $C_1^\star$, $C_2^\star$ and $C_3^\star$ positive to be chosen, we introduce
the following bootstrap estimates
\begin{align}
\left|\tau_n(s)-\kappa_\tau |s|^{-\frac{4\alpha+3}{2\alpha}}\right|
&\leq 
C_1^\star |s|^{-\frac{4\alpha+3}{2\alpha}-1}\log|s|,\label{eq:1n}\\
\left|\sigma_n(s)-\kappa_\sigma  |s|^{-\frac1{2\alpha}}\right|
&\leq C_1^\star |s|^{-\frac1{2\alpha}-1}\log|s|,\label{eq:2n}\\
\left|\lambda_n(s)-\kappa_\lambda  |s|^{-\frac{2\alpha+1}{2\alpha}}\right|
&\leq C_1^\star |s|^{-\frac{2\alpha+1}{2\alpha}-1}\log|s|\label{eq:3n}\\
|b_n(s)|
&\leq C_2^\star |s|^{-3} \log|s|,\label{eq:4n}\\
\mathcal N_B(\varepsilon_n(s))
+\left(\int_{S_n}^s \left( \frac{\tilde{s}}{s} \right)^{\frac{2\alpha+1}{\alpha}} \|\varepsilon_n(\tilde{s})\|_{L^2_\loc}^2 
d\tilde{s} \right)^{\frac 12} 
&\leq C_3^\star  |s|^{-3} \log|s|,\label{eq:5n}
\end{align}
where $\kappa_\tau$, $\kappa_\sigma$ and $\kappa_\lambda$ are defined in \eqref{eq:ka}. 

For $s_0<-1$, $|s_0|$ large enough to be chosen later, we define $S_n^\star \in (S_n,s_0]$ by 
\begin{equation*}
 S_n^{\star}=\sup \left\{\mbox{$S_n<\tilde{s}<s_0$ : \eqref{eq:H1b}, \eqref{eq:1n}--\eqref{eq:5n} are satisfied for all $s\in [S_n,\tilde{s}]$}\right\}.
\end{equation*}
\begin{proposition} \label{prop:bootstrap}
There exist $C_1^\star>1$, $C_2^\star>1$, $C_3^\star>1$, $B>1$ and $s_0<-1$, independent of $n$, such that for $n$ large enough, $S_n^{\star}=s_0$. 
\end{proposition}

Proposition~\ref{prop:bootstrap} is proved in the rest of Section~\ref{S:5} and in Section~\ref{S:6}.
We work on the time interval $\cI=[S_n,S_n^{\star}]$ where the bootstrap estimates~\eqref{eq:H1b}, \eqref{eq:1n} and~\eqref{eq:5n} hold.
In every step of the proof, the value of $|s_0|$ will be taken large enough, depending on the constants $\delta$, $C_1^\star$, $C_2^\star$ $C_3^\star$
and~$B$, but independent of $n$.
For simplicity of notation, we drop the $n$ in index of the functions $(W_{b,n},\varepsilon_n,\lambda_n,\sigma_n,b_n)$.

\subsection{Closing the parameter estimates}
The next lemma follows from inserting \eqref{eq:1n}--\eqref{eq:5n}
into the estimates \eqref{est:lambda:sigma} and \eqref{est:b_s}, using also \eqref{eL2B}.
\begin{lemma}\label{le:cB}
For all $s\in\cI$,
\begin{align}
 |\vec m |= \left| \frac{\lambda_s}{\lambda} + \beta + b \right|+\left| \frac{\sigma_s}{\lambda} - 1 \right|
 &\lesssimD \|\varepsilon\|_{L^2_\loc} +(C_1^\star+C_2^\star)|s|^{-4}\log |s| , \label{eq:1c}\\
 \left| b_s + \frac{2\alpha+1}{\alpha} s^{-1} b \right| & \lesssimD |s|^{-1}\|\varepsilon\|_{L^2_\loc} + C_1^\star|s|^{-4}\log |s|. \label{eq:2c} 
\end{align}
\end{lemma}
We strictly improve the bootstrap estimates \eqref{eq:1n}-\eqref{eq:4n} on the parameters. 
\begin{lemma} \label{le:pB}
There exists $C_1^\star>1$, independent of $C_3^\star$
such that for all $s \in \cI$,
\begin{align}
\left|\tau(s)-\kappa_\tau |s|^{-\frac{4\alpha+3}{2\alpha}}\right|
&\leq 
\frac{C_1^\star}2  |s|^{-\frac{4\alpha+3}{2\alpha}-1}\log|s|,\label{eq:1o}
\\
\left|\sigma(s)-\kappa_\sigma  |s|^{-\frac1{2\alpha}}\right|
&\leq \frac{C_1^\star}2 |s|^{-\frac1{2\alpha}-1}\log|s|,\label{eq:2o}\\
\left|\lambda(s)-\kappa_\lambda  |s|^{-\frac{2\alpha+1}{2\alpha}}\right|
&\leq \frac{C_1^\star}2 |s|^{-\frac{2\alpha+1}{2\alpha}-1}\log|s|\label{eq:3o}\\
|b(s)|
&\leq \frac{C_2^\star}2 |s|^{-3} \log|s|.\label{eq:4o}
\end{align}
\end{lemma} 
\begin{proof}
\emph{Closing the estimates for $\sigma$ and $\lambda$.}
Note that \eqref{eq:1c} and \eqref{eq:4n}--\eqref{eq:5n} implies the rougher estimates
\begin{equation}\label{eq:1d}
\left| \frac{\lambda_s}{\lambda} + c_1 \lambda^\frac12 \sigma^{\alpha-\frac12}
+ c_2 \lambda \sigma^{2\alpha-1}\right|
+ \left| \frac{\sigma_s}{\lambda} -1 \right|\lesssim |s|^{-\frac 52},
\end{equation}
where the bootstrap constants $C_2^\star$ and $C_3^\star$ do not appear anymore.
(This is obtained by taking $S_0$ large depending on the bootstrap constants.)
We set 
\[
g(s) = \lambda^\frac12 (s) - \sigma^{\alpha+\frac 12}(s) + \frac{c_2}{2} \frac1{3\alpha+\frac12} \sigma^{3\alpha+\frac12}.
\]
Then, using \eqref{eq:1d}, the expression of $c_1=-(2\alpha+1)$
and $\lambda^\frac12 = \sigma^{\alpha+\frac 12} + O(|s|^{-\frac 12-\frac{2\alpha+1}{4\alpha}})$
(from \eqref{eq:2n}-\eqref{eq:3n}),
\begin{align*}
g_s & = \frac 12 \lambda_s \lambda^{-\frac12} - \left( \alpha+\frac12\right) \sigma_s \sigma^{\alpha-\frac12}
+\frac{c_2}{2} \sigma_s \sigma^{3\alpha-\frac 12}\\
& = -\frac {c_1}2\lambda \sigma^{\alpha-\frac 12}- \left(\alpha+\frac 12\right) \lambda \sigma^{\alpha-\frac12}  -\frac {c_2}2(\lambda^\frac32 \sigma^{2\alpha-1} -\lambda \sigma^{3\alpha-\frac12})
+O(|s|^{-\frac52 -\frac{2\alpha+1}{4\alpha}})
\\
& = O(|s|^{-\frac 52 -\frac{2\alpha+1}{4\alpha}}).
\end{align*}
Thus, integrating over $(S,s)$, using (see beginning of \S \ref{S:5.1}
for the choice of $\lambda(S)$ and $\sigma(S)$)
\begin{equation}\label{eq:dg}
g(S)=0,
\end{equation}
we find $|g(s)|\lesssim |s|^{-\frac32 -\frac{2\alpha+1}{4\alpha}}$.
In particular,
\begin{equation}\label{eq:sl}
\lambda = \sigma^{2\alpha+1} -\frac{c_2}{3\alpha+\frac12} \sigma^{4\alpha+1}+ O (|s|^{-\frac 32-\frac{2\alpha+1}{2\alpha}}) ,\quad
\frac1\lambda= \sigma^{-2\alpha-1} +\frac{c_2}{3\alpha+\frac12} \sigma^{-1}+ O(|s|^{-\frac 32+\frac{2\alpha+1}{2\alpha}}).
\end{equation}
Inserting this into \eqref{eq:1d}, we find
\[
\left|  \sigma_s(\sigma^{-2\alpha-1}+\frac{c_2}{3\alpha+\frac12} \sigma^{-1}) -1 \right | 
=\left| \frac1{2\alpha} (\sigma^{-2\alpha})_s - \frac{c_2}{3\alpha+\frac12} (\log \sigma)_s +1 \right|
\lesssim |s|^{-\frac 32}.
\]
Integrating over $(S,s)$ and recalling the choice $\sigma(S)$ such that 
\begin{equation}\label{eq:cs}
\frac 1{2\alpha} \sigma^{-2\alpha}(S) - \frac{c_2}{3\alpha+\frac12} \log \sigma (S) +S=0 ,
\end{equation}
we get
\[
\left|\frac 1{2\alpha} \sigma^{-2\alpha} - \frac{c_2}{3\alpha+\frac12} \log \sigma  + s  \right| \lesssim  |s|^{-\frac12}.
\]
Thus,
$\left|\frac 1{2\alpha} \sigma^{-2\alpha}  + s  \right| \lesssim  \log|s|$, and so
\[
\left|\sigma  - \left(2 \alpha |s|\right)^{-\frac 1{2\alpha}}\right|
\lesssim |s|^{-\frac1{2\alpha}-1} \log|s|,
\]
which implies \eqref{eq:2o} for $C_1^\star$ sufficiently large.
Moreover, \eqref{eq:sl} now implies 
\begin{equation}\label{eq:99}
\left|\lambda -\kappa_\lambda  |s|^{-\frac{2\alpha+1}{2\alpha}}\right|
\leq |s|^{-\frac{2\alpha+1}{2\alpha}-1}\log|s|
\end{equation}
and thus \eqref{eq:3o} holds for $C_1^\star$ sufficiently large.
Note that from the above estimates, one could derive a higher order expansion of the parameters.

\noindent \emph{Closing the estimate for $\tau$.}
We use 
\[
\frac{d\tau}{ds} = \lambda^3
\]
and \eqref{eq:99}, to obtain
\[
\frac{d\tau}{ds} = \kappa_\lambda^3 |s|^{-\frac{6\alpha+3}{2\alpha}}
+O\big(|s|^{-\frac{6\alpha+3}{2\alpha}-1} \log|s|\big).
\]
By integration on $(S,s)$, using $\tau(S)=T=\kappa_\tau |S|^{-\frac{4\alpha+3}{2\alpha}}$,
we obtain
\[
\tau(s) = \kappa_\tau |s|^{-\frac{4\alpha+3}{2\alpha}}
+O\big(|s|^{-\frac{4\alpha+3}{2\alpha}-1} \log|s|\big).
\]

Now, the constant $C_1^\star$ is fixed sufficiently large.

\noindent \emph{Closing the estimate for $b$.}
From \eqref{eq:2c} (recall that the bootstrap constant $C_1^\star$ is fixed, and we do not mention it
anymore), it holds
\[
|(s^{\frac{2\alpha+1}{\alpha}}b)_s| \lesssim |s|^{-1+\frac{2\alpha+1}{\alpha}} \|\varepsilon\|_{L^2_\loc}
+|s|^{-4+\frac{2\alpha+1}{\alpha}} \log |s|.
\]
Thus, integrating on $(S,s)$, using $b(S)=0$ and multiplying by $|s|^{-\frac{2\alpha+1}{\alpha}}$,
we obtain
\[
|b(s)|\lesssimD
|s|^{-\frac{2\alpha+1}{\alpha}}\int_{S}^s |\tilde{s}|^{-1+\frac{2\alpha+1}{\alpha}} \|\varepsilon(\tilde{s})\|_{L^2_\loc} d\tilde{s}
+|s|^{-3} \log |s|.
\]
We set $d=\frac{\alpha-1}{3\alpha}$. We have, using \eqref{eq:5n},
\begin{align*}
|\tilde{s}|^{-1+\frac{2\alpha+1}{\alpha}} \|\varepsilon(\tilde{s})\|_{L^2_\loc}
&\lesssim |\tilde{s}|^{\frac{\alpha+1}{\alpha}} 
\cN_B(\varepsilon(\tilde{s}))^{1-2d}\|\varepsilon(\tilde{s})\|_{L^2_\loc}^{2d}\\
&\lesssim (C_3^\star)^{1-2d} |\tilde{s}|^{-3(1-2d)} (\log|\tilde{s}|)^{1-2d} 
|\tilde{s}|^{\frac{\alpha+1}{\alpha}}  \|\varepsilon(\tilde{s})\|_{L^2_\loc}^{2d}\\
&\lesssim (C_3^\star)^{1-2d}  (\log|\tilde{s}|)^{1-2d} 
|\tilde{s}|^{\frac{1-2\alpha-d +4\alpha d}{\alpha}}
\left( |\tilde{s}|^{\frac{2\alpha+1}{\alpha}}\|\varepsilon(\tilde{s})\|_{L^2_\loc}^2\right)^{d}
\end{align*}
Thus, by the Hölder inequality,
\begin{align*}
&\int_{S}^s |\tilde{s}|^{-1+\frac{2\alpha+1}{\alpha}} \|\varepsilon(\tilde{s})\|_{L^2_\loc} d\tau\\
&\quad \lesssim 
(C_3^\star)^{1-2d}   \left(\int_{S}^s (\log|\tilde{s}|)^{\frac{1-2d}{1-d}}|\tilde{s}|^{\frac{1-2\alpha-d +4\alpha d}{\alpha(1-d)}} d\tilde{s} \right)^{1-d}
\left( \int_{S}^s  |\tilde{s}|^{\frac{2\alpha+1}{\alpha}}\|\varepsilon(\tilde{s})\|_{L^2_\loc}^2 d\tilde{s} \right)^d.
\end{align*}
We compute
\[
\frac{1-2\alpha-d +4\alpha d}{\alpha(1-d)}
=-1-\frac{\alpha-1}{\alpha(2\alpha+1)}.
\]
Therefore, using \eqref{eq:5n},
\begin{align*}
 \int_{S}^s |\tilde{s}|^{-1+\frac{2\alpha+1}{\alpha}} \|\varepsilon(\tilde{s})\|_{L^2_\loc} d\tau 
&\  \lesssim 
(C_3^\star)^{1-2d} (\log|s|)^{\frac{\alpha+2}{3\alpha}}|s|^{-\frac{\alpha-1}{3\alpha^2}}  
\left( \int_{S}^s |\tilde{s}|^{\frac{2\alpha+1}{\alpha}}\|\varepsilon(\tilde{s})\|_{L^2_\loc}^2 d\tau \right)^{\frac{\alpha-1}{3\alpha}} \\
& \lesssim 
C_3^\star(\log|s|)^{\frac{3\alpha+1}{3\alpha}}
|s|^{\frac{-4(\alpha-1)}{3\alpha}}  .
\end{align*}
Finally, since $\alpha>1$,
\[
|b(s)|\lesssimD C_3^\star (\log|s|)^{\frac{3\alpha+1}{3\alpha}}
|s|^{-3-\frac13(1-\frac1{\alpha}))} 
+|s|^{-3} \log |s| \lesssimD |s|^{-3} \log |s|.
\]
Thus, we have proved \eqref{eq:4o} by taking the bootstrap constant $C_2^\star$ sufficiently large. 
\end{proof}

\subsection{Estimate of global norms from conservation of mass and energy}

\begin{lemma}\label{le:mass-ener}
For any $s \in \cI$,
\begin{equation}\label{GN:BS}
\|\varepsilon(s)\|_{L^2} \lesssim \delta^\alpha ,\quad
\|\partial_y \varepsilon(s)\|_{L^2}\lesssim \delta^{\alpha-1} \lambda(s), \quad
\|\varepsilon(s)\|_{L^{\infty}}\lesssim \delta^{\alpha-\frac12}\lambda^\frac12(s) .
\end{equation}
\end{lemma}
\begin{proof} 
From $\varepsilon(S)\equiv 0$, $b(S)=0$ and $\int w^2(S)=\int w^2(s)$, it follows that
\begin{equation*}
 \int W^2(S)= \int w^2(s)=\int (W_b+\varepsilon)^2(s)=\int W_b^2(s)+2\int (W+bP_b)(s) \varepsilon(s)+\int \varepsilon^2(s) .
\end{equation*}
Using \eqref{eq:ma}, we have $\left| \int W^2(s)-\int W^2(S) \right| \lesssim \delta^{2\alpha}$.
Moreover, we deduce from the orthogonality relation $\int \varepsilon Q =0$,
 the Cauchy-Schwarz inequality and the estimates \eqref{eq:mg} and \eqref{def:P:b} that
\begin{equation*}
 \left| \int (W+bP_b) \varepsilon \right| \le \left( \|V\|_{L^2}+|b|\|P_b\|_{L^2}\right) \|\varepsilon \|_{L^2} \le \frac12 \int \varepsilon^2+C(\delta^{2\alpha}+|b|^{2-\gamma}) .
\end{equation*}
Combining these estimates with \eqref{mass:W:b} comparing $W$ and $W_b$, yields 
the control of the $L^2$ norm of $\varepsilon$ (recall that \eqref{wBTb}
control the size of $b$). 

\smallskip

The conservation of energy implies that 
$\lambda^{-2}(S)E(w(S))=\lambda^{-2}(s)E(w(s)$. Thus, by~\eqref{w:0}
\begin{equation*}
 \frac{E(W(S))}{\lambda^2(S)} = \frac{E(w(s))}{\lambda^2(s)}=
 \frac{E(W_b(s)+\varepsilon(s))}{\lambda^2(s)} .
\end{equation*}
Using the identity $\int(Q''+Q^5)\varepsilon=\int Q\varepsilon=0$,
\begin{equation*}
\begin{split}
E(W_b+\varepsilon)&=E(W_b)+\int (\partial_yV+b \partial_y P_b)(\partial_y\varepsilon) +\frac12 \int (\partial_y\varepsilon)^2 \\ & \quad -\int (W_b^5-Q^5)\varepsilon -\frac16 \int \left((W_b+\varepsilon)^6-W_b^6-6W_b^5\varepsilon \right) .
\end{split}
\end{equation*}
Using the Cauchy-Schwarz inequality and then \eqref{eq:mg.1},
\begin{equation*}
 \left|\int (\partial_y(V-V_1\theta))(\partial_y\varepsilon) \right|
 \leq \|\partial_y (V-V_1\theta)\|_{L^2}\|\partial_y\varepsilon\|_{L^2}
 \leq \frac18 \|\partial_y\varepsilon\|_{L^2}^2 + C|s|^{-3}.
\end{equation*}
Now, we estimate $\int \partial_y(V_1\theta)\partial_y\varepsilon$ as follows,
using $A_1\in \cZ_0$ and then \eqref{eq:5n} and \eqref{theta:H1} 
\begin{align*}
\left|\int \partial_y(V_1\theta)\partial_y\varepsilon\right|
& \lesssim \int \left| A_1'\theta\partial_y\varepsilon\right|
+\int \left| \partial_y(\theta)\partial_y\varepsilon\right|\\
& \lesssim |s|^{-1} \cN_B(\varepsilon) + \|\partial_y\theta\|_{L^2}  \|\partial_y\varepsilon\|_{L^2}\\
& \lesssim C_3^\star|s|^{-4} \log|s| + \delta^{\alpha-1} \lambda  \|\partial_y\varepsilon\|_{L^2}
\leq \frac 18 \int (\partial_y\varepsilon)^2 + \delta^{2\alpha-2}\lambda^2.
\end{align*}
Next, the Cauchy-Schwarz inequality and \eqref{def:P:b} yield 
\begin{equation*}
 \left|\int b(\partial_y P_b)(\partial_y\varepsilon) \right| \lesssim |b|\|\partial_y\varepsilon\|_{L^2}
 \le \frac18 \|\partial_y\varepsilon\|_{L^2}^2+C|b|^2 .
\end{equation*}
We observe from \eqref{eq:vv}, \eqref{def:P:b},
\begin{align*}
\left| \int (W_b^5-Q^5) \varepsilon \right|
&\lesssim \int Q^5 |\zeta^5-1||\varepsilon|+\int V_0^4 (|V|+|b||P_b|)|\varepsilon|+\int (|V|^5+|b|^5|P_b|^5)|\varepsilon|\\
& \lesssim (|s|^{-1}+|b|)\|\varepsilon\|_{L^2_\loc}
+(\|V\|_{L^\infty}^4\|V\|_{L^2} +|b|^{5-\frac{\gamma}2}) \|\varepsilon\|_{L^2}\\
& \lesssim |s|^{-1}\|\varepsilon\|_{L^2_\loc}
+((|s|^{-4}+\lambda^2)\delta^{\alpha} +|b|^{5-\frac{\gamma}2})\|\varepsilon\|_{L^2}\\
&
\lesssim |s|^{-1}\|\varepsilon\|_{L^2_\loc} + \delta^{2\alpha}\lambda^2 + b^2.
\end{align*}
Moreover, from \eqref{est:W0:b} and the Gagliardo-Nirenberg inequality \eqref{sharpGN},
\begin{align*}
\left| \int \left((W_b+\varepsilon)^6-W_b^6-6W_b^5\varepsilon \right)\right|
&\lesssim \int \left( W_b^4\varepsilon^2+\varepsilon^6 \right)\\
&\lesssim \|\varepsilon\|_{L^2_\loc}^2+(|s|^{-4}+\lambda^2+b^4)\|\varepsilon\|_{L^2}^2+\|\varepsilon\|_{L^2}^4\|\partial_y \varepsilon\|_{L^2}^2\\
&\lesssim \|\varepsilon\|_{L^2_\loc}^2+ \delta^{2\alpha} \lambda^2 +b^4+ \delta^{4\alpha} \|\partial_y \varepsilon\|_{L^2}^2.
\end{align*}
Gathering these estimates, for $\delta$ small enough, we deduce
\begin{align*}
 \|\partial_y \varepsilon\|_{L^2}^2 &\lesssim 
 \|\varepsilon\|_{L^2_\loc}^2+|s|^{-1}\|\varepsilon\|_{L^2_\loc}
 +\delta^{2\alpha-2}\lambda^2+b^2 +\lambda^2 \left| \frac{E(W_b)}{\lambda^2}-\frac{E(W(S))}{\lambda^2(S)}\right| .
\end{align*}
By \eqref{energy:W:b} and then \eqref{eq:B1}--\eqref{eq:B3} we have (recall $\gamma = 3/4$)
\[
|E(W_b)-E(W)|\lesssim |b| + |b|^{\frac 14}|s|^{-5}
\lesssim |b| + |s|^{-\frac {11}2}
\lesssim |b| +\delta^{2\alpha-2} \lambda^2.
\]
From \eqref{eq:en}, \eqref{eq:1c}, \eqref{eL2B} and \eqref{eq:2n}--\eqref{eq:5n}, we have
\[
\left| \frac{d}{ds} \left[ \frac{E(W)}{\lambda^2}\right]\right| \lesssim \frac 1{\lambda^2}
C_3^\star |s|^{-4}\log|s|
\lesssim C_3^\star |s|^{-2+\frac1\alpha}\log|s|.
\]
Thus, by integration over $(S,s)$,
\begin{equation}\label{est:asymp:energy:W}
\left| \frac{E(W(s))}{\lambda^2(s)}-\frac{E(W(S))}{\lambda^2(S)}\right| \lesssim  C_3^\star |s|^{-1+\frac 1\alpha}\log|s| \lesssim \delta^{2\alpha-2}.
\end{equation}
This yields, using also \eqref{eL2B} and \eqref{eq:2n}--\eqref{eq:5n},
\[
\|\partial_y\varepsilon\|_{L^2}^2 
\lesssim 
\|\varepsilon\|_{L^2_\loc}^2+|s|^{-1}\|\varepsilon\|_{L^2_\loc} 
+ \delta^{2\alpha-2 }\lambda^2 +|b|
\lesssim \delta^{2\alpha-2 }\lambda^2.
\]
The estimate $\|\varepsilon\|_{L^{\infty}}^2\lesssim \|\varepsilon\|_{L^2}\|\partial_y\varepsilon\|_{L^2}\lesssim \delta^{2\alpha-1} \lambda$
completes the proof of~\eqref{GN:BS}.
\end{proof}

Now, we use \eqref{eq:30} to improve the estimate of the $L^2$ norm of $\varepsilon$.
\begin{lemma} For any $s\in\cI$,
\begin{equation}\label{eq:EU}
\|\varepsilon\|_{L^2} \lesssim |s|^{-\frac12}.
\end{equation}
\end{lemma}
\begin{proof}
By \eqref{eq:IN} and \eqref{eq:30}, we have
\[
\int w^2(S) = \int W^2(S) = M_0 + O(|S|^{-1}).
\]
By the conservation of mass $\int w^2(S)=\int w^2(s)$, we have
\[
\int w^2(s)=
M_0+O(|S|^{-1})
= \int (W_b+\varepsilon)^2(s)=\int W_b^2(s)+2\int (W+bP_b)(s) \varepsilon(s)+\int \varepsilon^2(s)
\]
Using \eqref{eq:w0} and then \eqref{eq:5n},
\begin{align*}
\int (|W|+|bP_b|)|\varepsilon| & \lesssim 
  \int (\omega + |s|^{-1}\eta + \lambda^\frac 12 \eta_R + |b||P_b|) |\varepsilon|\\
&\lesssim \cN_B(\varepsilon) + |s|^{-1}  \int_{y<0} \eta |\varepsilon|
\lesssim |s|^{-2}+|s|^{-\frac12}\|\varepsilon\|_{L^2}.
\end{align*}
Using \eqref{eq:30}, \eqref{mass:W:b} and \eqref{eq:4n} we have
\[
\int W_b^2(s) = M_0 + O(|s|^{-1}).
\]
The above estimates imply that
$\|\varepsilon\|_{L^2} \lesssim |s|^{-\frac12}$, which is \eqref{eq:EU}.
\end{proof}

\section{Local energy estimates}\label{S:6}
To complete the proof of Proposition \ref{prop:bootstrap},
we improve the bootstrap estimate~\eqref{eq:5n} using the following virial-energy functional 
\[
\cF = \frac{1}{\lambda^2}
\left[ \int (\partial_y\varepsilon)^2 \psi_B + \int \varepsilon^2 \varphi_B
- \frac13 \int \left( (W_b+\varepsilon)^6 - W_b^6 - 6 W_b ^5 \varepsilon\right) \psi_B\right].
\]

\subsection{Virial-energy estimates}

\begin{proposition}\label{pr:local-energy}
There exist $\mu_1,\mu_2>0$ and $B>100$ such that the following hold on $\cI$.
\begin{enumerate}
\item[{(i)}] \emph{Time variation of $\cF$.} 
\begin{equation}\label{time:local}
\frac{d\cF}{ds} + \frac{\mu_1} {\lambda^2} \int \left( (\partial_y\varepsilon)^2 + \varepsilon^2\right)\varphi_B'
\lesssimD C_3^\star|s|^{-7+\frac{2\alpha+1}{\alpha}}(\log|s|)^2 .
\end{equation}
\item[{(ii)}]  \emph{Coercivity of $\cF$.} 
\begin{equation}\label{coer:local}
\cF \geq \frac{\mu_2}{\lambda^2} \left[{\cN}_B(\varepsilon)\right]^2 .
\end{equation}
\end{enumerate}
\end{proposition}

\begin{proof}
We start by the proof of \eqref{time:local}. Integrating by parts and then using \eqref{eq:eps}
\begin{align*}
\lambda^2 \frac{d\cF}{ds}& = 
2 \int \partial_s\varepsilon \left(-\psi_B' \partial_y\varepsilon - \psi_B \partial_y^2\varepsilon + \varepsilon \varphi_B
- \psi_B \left[ (W_b+\varepsilon)^5 - W_b^5\right] \right) \\
&\quad - 2 \frac{\lambda_s}{\lambda}\lambda^2 \cF
- 2 \int (\partial_s W_b) \left[ (W_b+\varepsilon)^5 - W_b^5 - 5 W_b^4 \varepsilon \right] \psi_B\\
& = f_1+f_2+f_3+f_4+f_5
\end{align*}
where
\begin{align*}
f_1 & = 2 \int \partial_y\left(-\partial_y^2\varepsilon + \varepsilon - \left[ (W_b+\varepsilon)^5 - W_b^5\right]\right) G_B(\varepsilon)\\
f_2 & = - 2 \int \cE_b(W) G_B(\varepsilon) \\
f_3 & = 2 \left(\frac{\sigma_s}{\lambda} - 1\right) \int (\partial_y\varepsilon) G_B(\varepsilon) \\
f_4 & = 2 \frac{\lambda_s}{\lambda} \int \left((\Lambda \varepsilon) G_B(\varepsilon)
-(\partial_y\varepsilon)^2 \psi_B - \varepsilon^2 \varphi_B
+ \frac 13 \left[ (W_b+\varepsilon)^6 - W_b^6 - 6 W_b^5 \varepsilon\right]\psi_B \right)\\
f_5 & = - 2 \int (\partial_s W_b) \left[ (W_b+\varepsilon)^5 - W_b^5 - 5 W_b^4 \varepsilon \right] \psi_B
\end{align*}
and
\[
G_B(\varepsilon) = - \partial_y(\psi_B \partial_y\varepsilon) + \varepsilon \varphi_B
- \psi_B \left[ (W_b+\varepsilon)^5 - W_b^5\right].
\]

\smallskip

\noindent \emph{Estimate for $f_1$.} We claim that there exist $\mu_1>0$, $B_0 \ge 100$, such that for all $B \ge B_0$,
\begin{equation} \label{est:f1}
 f_{1} +2\mu_1 \int \left(\varepsilon^2+(\partial_y\varepsilon)^2\right)
 \varphi_B' \lesssimD |s|^{-8} .
\end{equation}
Following the computations page 89 in \cite{MMR1},
\begin{align*}
f_1 & = - \int \left[ 3 (\partial_y^2 \varepsilon)^2 \psi_B'
+(3 \varphi_B'+\psi_B'-\psi_B''') (\partial_y \varepsilon)^2
+ (\varphi_B'-\varphi_B''') \varepsilon^2 \right]\\
&\quad - \frac 13 \int \left[ (W_b+\varepsilon)^6 - W_b^6 - 6 (W_b+\varepsilon)^5 \varepsilon \right] (\varphi_B'-\psi_B')\\
&\quad + 2 \int\left[ (W_b+\varepsilon)^5 - W_b^5 - 5 W_b^4 \varepsilon\right] (\partial_yW_b) (\psi_B-\varphi_B)\\
&\quad + 10 \int \psi_B' (\partial_y \varepsilon) \left[ (\partial_y W_b)\left( (W_b+\varepsilon)^4 - W_b^4\right)
+ (W_b+\varepsilon)^4 (\partial_y \varepsilon)\right]\\
&\quad - \int \psi_B' \left[\left(-\partial_y^2\varepsilon+ \varepsilon- \left( (W_b+\varepsilon)^5 -W_b^5\right)\right)^2
- (-\partial_y^2 \varepsilon + \varepsilon)^2 \right].
\end{align*}
We decompose $f_1  =:f_1^{<}+f_1^{\sim}+f_1^{>}$
where $f_1^{<}$, $f_1^{\sim}$ and $f_1^{>}$ correspond to the integration over the regions $y<-\frac{B}2$, $|y|\le \frac{B}2$ and $y>\frac{B}2$ respectively. 

We first estimate $f_1^{<}$. By using the properties of $\varphi_B$ and $\psi_B$ in \eqref{cut:left}-\eqref{cut:onR} and choosing $B \ge 100$ large enough, we have 
\begin{align*}
 f_1^{<}+3\int_{y<-\frac{B}2} (\partial_y^2\varepsilon)^2& \psi_B'+\frac12 \int_{y<-\frac{B}2} \left( \varepsilon^2+(\partial_y\varepsilon)^2 \right) \varphi_B'
 \\ 
 & 
 \lesssim \int_{y<-\frac{B}2} |\partial_y\varepsilon| \left(|\partial_yW_b| \left(|W_b|^3|\varepsilon|+\varepsilon^4\right)+|\partial_y\varepsilon|\left( |W_b|^4+\varepsilon^4\right)\right)\varphi_B' \\ 
 &\quad +\int_{y<-\frac{B}2} \left| -2\left(-\partial_y^2\varepsilon+\varepsilon \right)+ \left( (W_b+\varepsilon)^5-W_b^5\right) \right| \left|(W_b+\varepsilon)^5-W_b^5 \right|\psi_B' \\ 
 &=:f_{1,1}^{<}+f_{1,2}^{<} .
\end{align*}
It follows from \eqref{est:W0:b}, \eqref{est:W12:b}, then \eqref{cut:left}-\eqref{cut:onR} and \eqref{GN:BS} that
\begin{align*}
 f_{1,1}^{<} &\lesssim \int_{y<-\frac{B}2} |\partial_yW_b| \left(|W_b|^3+|\varepsilon|^3\right) \left(\varepsilon^2+(\partial_y\varepsilon)^2\right)\varphi_B'+\int_{y<-\frac{B}2} (\partial_y\varepsilon)^2\left(|W_b|^4+\varepsilon^4 \right)\varphi_B' \\ 
 & \lesssim \left(e^{-\frac{B}{4}}+\lambda^2\right) \int_{y<-\frac{B}2}\left(\varepsilon^2+(\partial_y\varepsilon)^2\right)\varphi_B'
\end{align*}
and
\begin{align*}
 f_{1,2}^{<} 
 & \lesssim \int_{y<-\frac{B}2} \left(|W_b|^4+\varepsilon^4\right) \left(\varepsilon^2+(\partial_y^2\varepsilon)^2\right) \psi_B'+
 \int_{y<-\frac{B}2}\left(|W_b|^8+\varepsilon^8\right)\varepsilon^2\varphi_B'
 \\
 & \lesssim \left(e^{-\frac{B}{4}}+\lambda^{2}\right) \int_{y<-\frac{B}2} \left((\partial_y^2\varepsilon)^2\psi_B'+
\varepsilon^2\varphi_B' \right).
\end{align*}
Gathering these estimates and taking $B>100$ and $|s|$ large enough, we deduce that 
\begin{equation} \label{est:f1<}
f_1^{<}+\int_{y<-\frac{B}2} (\partial_y^2\varepsilon)^2 \psi_B'+\frac14 \int_{y<-\frac{B}2} \left( \varepsilon^2+(\partial_y\varepsilon)^2 \right) \varphi_B' \le 0 .\end{equation}
 
Since $\psi_B' \equiv 0$ on $y>\frac{B}2$, we observe using also \eqref{cut:onR} that, for $B>100$ large enough, 
\begin{align*}
f_1^> +\frac12 \int_{y>\frac{B}2} \left( \varepsilon^2+(\partial_y\varepsilon)^2\right) \varphi_B' 
\lesssim \int_{y>\frac{B}2}\left(|W_b|^4+\varepsilon^4 \right) \varepsilon^2 \varphi_B' + \int_{y>\frac{B}2}|\partial_yW_b|\left( |W_b|^3+|\varepsilon|^3 \right)\varepsilon^2 \varphi_B.
\end{align*}

Using \eqref{est:W0:b}, \eqref{est:W12:b} and \eqref{GN:BS}, it follows that
\begin{equation*}
 f_{1}^{>} +\frac12 \int_{y>\frac{B}2} \left( \varepsilon^2+(\partial_y\varepsilon)^2\right) \varphi_B' \lesssim \int_{y>\frac{B}2} \varepsilon^2 \omega 
 +\lambda^2\int_{y>\frac{B}2} \varepsilon^2 \varphi_B
 \lesssim e^{-\frac{B}8} \int_{y>\frac{B}2} \varepsilon^2 \varphi_B'+\lambda^2[\cN_B(\varepsilon)]^2.
\end{equation*}
Thus, using \eqref{eq:3n} and \eqref{eq:5n}, for $B$ and $|s|$ large enough,
we obtain
\begin{equation} \label{est:f1>}
f_1^{>}+\frac14 \int_{y>\frac{B}2} \left( \varepsilon^2+(\partial_y\varepsilon)^2\right) \varphi_B' \lesssim (C_3^\star)^2\lambda^2|s|^{-6}(\log|s|)^2 \lesssim |s|^{-8} .
\end{equation}

In the region $|y|<\frac{B}2$, using $\varphi_B(y)=1+\frac{y}B$ and $\psi_B(y)=1$,
we obtain the following expression of $f_1^{\sim}$
\begin{align*} 
f_1^{\sim}&=-\frac1B \int_{|y|<\frac{B}2} \left(3(\partial_y\varepsilon)^2+\varepsilon^2+\frac13\left[ (W_b+\varepsilon)^6 - W_b^6 - 6 (W_b+\varepsilon)^5 \varepsilon \right]\right) \\
&\quad -\frac2B \int_{|y|<\frac{B}2}\left[ (W_b+\varepsilon)^5 - W_b^5 - 5 W_b^4 \varepsilon\right] y(\partial_yW_b).
\end{align*}
We decompose $f_1^{\sim}$ as
\begin{align*} 
f_1^{\sim}
&=-\frac1B \int_{|y|<\frac{B}2}\left(3(\partial_y\varepsilon)^2+\varepsilon^2-5Q^4\varepsilon^2+20yQ^3Q'\varepsilon^2\right) \\ 
&\quad -\frac1{3B}\int_{|y|<\frac{B}2}\left[ (W_b+\varepsilon)^6 - W_b^6 - 6 (W_b+\varepsilon)^5 \varepsilon + 15Q^4\varepsilon^2\right] \\ & 
\quad -\frac2B\int_{|y|<\frac{B}2}\left[ (W_b+\varepsilon)^5 - W_b^5 - 5 W_b^4 \varepsilon-10W_b^3\varepsilon^2\right] y(\partial_yW_b) \\ 
& \quad -\frac{20}B \int_{|y|<\frac{B}2} y\left( W_b^3\partial_yW_b-Q^3Q'\right) \varepsilon^2 \\ 
&=:f_{1,1}^{\sim}+f_{1,2}^{\sim}+f_{1,3}^{\sim}+f_{1,4}^{\sim}.
\end{align*}

To control the main quadratic term $f_{1,1}^{\sim}$, we rely on the virial-type estimate proved in \cite{MMR1}.
\begin{lemma}[{\cite[Lemma 3.4]{MMR1}}]\label{le:virial}
There exists $B_0>100$ and $\mu_0>0$ such that, for all $B\ge B_0$,
\[
\int_{|y|<\frac B2} \left[ 3 (\partial_y\varepsilon)^2 + \varepsilon^2 - 5 Q^4 \varepsilon^2 + 20 y Q'Q^3 \varepsilon^2\right]
\geq \mu_0 \int_{|y|< \frac B2} \left(\varepsilon^2+(\partial_y\varepsilon)^2\right) - \frac 1B \int \varepsilon^2 e^{-\frac {|y|}2}.
\]
\end{lemma}
Hence, 
\begin{equation*}
 f_{1,1}^{\sim} +\mu_0 \int_{|y|<\frac{B}2}\left(\varepsilon^2+(\partial_y\varepsilon)^2\right)
 \varphi_B' \lesssim \frac1{B} \int \varepsilon^2\varphi_B' .
\end{equation*}
 
To estimate $f_{1,2}^{\sim}$ and $f_{1,4}^\sim$, we need the following estimate for $k=0,1$, $y \in \mathbb R$,
\begin{equation} \label{cons:BS:Wb3Wb'}
 \left|\partial_y^k (W_b^4-Q^4)\right|
 \lesssimD |s|^{-1}\omega +
 |s|^{-4}\eta + \lambda^2 \eta_R
 +b^4\ONE_{[-1,0]}(|b|^{\gamma}y)\lesssim |s|^{-1}.
\end{equation}
Indeed, for $k=0$, we observe that
\begin{equation*}
 \left| W_b^4-Q^4\right|\lesssim \left| W^4- Q^4 \right| +\left|W_b^4- W^4 \right| 
 \lesssim Q^3|V|+V^4+|b||P_b||W|^3+|b|^4|P_b|^4.
\end{equation*}
For $k=1$, we expand 
\begin{equation*}
 W_b^3\partial_yW_b-Q^3Q'=(W^3-Q^3+3bW^2P_b+3b^2WP_b^2+b^3P_b^3)\partial_y W_b+Q^3\partial_y(V+bP_b)
\end{equation*}
which implies that
\begin{equation*}
 \left|W_b^3\partial_yW_b-Q^3Q'\right| \lesssim \left(Q^2|V|+|V|^3+|b|W^2|P_b|+|b|^3|P_b|^3 \right)|\partial_yW_b|+Q^3\left(|\partial_yV|+|b||\partial_yP_b|\right).
\end{equation*}
Thus, \eqref{cons:BS:Wb3Wb'} follows from \eqref{eq:w0}, \eqref{eq:vv}, \eqref{def:P:b} and \eqref{eq:2n}--\eqref{eq:4n}.

\smallskip

We decompose 
\begin{align*}
(&W_b+\varepsilon)^6 - W_b^6 - 6 (W_b+\varepsilon)^5 \varepsilon +15Q^4\varepsilon^2 \\ &=(W_b+\varepsilon)^6 - W_b^6-6W_b^5\varepsilon-15W_b^4\varepsilon^2
-6\left[(W_b+\varepsilon)^5-W_b^5-5W_b^4\varepsilon \right] \varepsilon
-15\left[W_b^4-Q^4\right]\varepsilon^2 ,
\end{align*}
so that 
\begin{align*}
\left| (W_b+\varepsilon)^6 - W_b^6 - 6 (W_b+\varepsilon)^5 \varepsilon +15Q^4\varepsilon^2 \right| \lesssim \left(|W_b|^3|\varepsilon|+\varepsilon^4 \right)\varepsilon^2+\left|W_b^4-Q^4\right|\varepsilon^2 .
\end{align*}
Thus, it follows from \eqref{est:W0:b}, \eqref{GN:BS} and \eqref{cons:BS:Wb3Wb'} that 
\begin{equation*}
 |f_{1,2}^{\sim}| \lesssim \left( \|W_b\|_{L^{\infty}}^3\|\varepsilon\|_{L^{\infty}}+\|\varepsilon\|_{L^{\infty}}^4+\|W_b^4-Q^4\|_{L^{\infty}} \right) \int_{|y|<\frac{B}2}\varepsilon^2\lesssim \lambda^\frac12\int_{|y|<\frac{B}2} \varepsilon^2 \varphi_B' .
\end{equation*}
Similarly, \eqref{est:W0:b}, \eqref{est:W12:b} and \eqref{GN:BS} yield 
\begin{align*}
|f_{1,3}^{\sim}| &\lesssim \frac1B \int_{|y|<\frac{B}2} \left(W_b^2|\varepsilon|^3+|\varepsilon|^5 \right) |y||\partial_yW_b|\\ &
\lesssim \left( \|W_b\|_{L^{\infty}}^2\|\varepsilon\|_{L^{\infty}}+\|\varepsilon\|_{L^{\infty}}^3 \right)\|\partial_yW_b\|_{L^{\infty}} \int_{|y|<\frac{B}2}\varepsilon^2 
\lesssimD \lambda^\frac12 B \int_{|y|<\frac{B}2} \varepsilon^2\varphi_B'.
\end{align*}
Finally, \eqref{cons:BS:Wb3Wb'} implies that
\begin{equation*}
 |f_{1,4}^{\sim}| \lesssim \left\| W_b^3\partial_yW_b-Q^3Q'\right\|_{L^{\infty}} \int_{|y|<\frac{B}2} \varepsilon^2
 \lesssimD |s|^{-1}B\int_{|y|<\frac{B}2} \varepsilon^2 \varphi_B'.
\end{equation*}
Hence, by gathering these estimates and taking $|s|$ large enough,
\begin{equation} \label{est:f1sim}
 f_{1}^{\sim} +\frac{\mu_0}2 \int_{|y|<\frac{B}2}\left(\varepsilon^2+(\partial_y\varepsilon)^2\right)\varphi_B'
 \lesssimD \frac1{B} \int \varepsilon^2\varphi_B' .
\end{equation}

The proof of \eqref{est:f1} follows by combining \eqref{est:f1<}, \eqref{est:f1>}, \eqref{est:f1sim} and by fixing $2^{-4}\min\{1,\mu_0\}$, $B\ge B_0$ large enough and taking $|s|$ large enough possibly depending on $B$.

\smallskip

\noindent \emph{Estimate for $f_2$.} We claim
\begin{equation} \label{est:f2}
|f_2| \le \frac{\mu_1}2 \int \varepsilon^2\varphi_B'+ C_\delta C_3^\star|s|^{-7}(\log|s|)^2 .
\end{equation}

Using the expression of $\cE_b(W)$ in \eqref{def:EbW},
\begin{align*}
f_2 &= 
2 \int \left(\vec{m} \cdot \vec{M}V_0 \right)G_B(\varepsilon) -2 \int \left( \Psi_M + \Psi_{W} + \Psi_b\right) G_B(\varepsilon)=: f_{2,1}+f_{2,2}.
\end{align*}

We deal first with $f_{2,1}$.
Integrating by parts and using $\cL\Lambda Q = -2Q$, $\cL Q'=0$, we compute
\begin{align*}
\int \Lambda V_0 \, G_B(\varepsilon)
&= \int \left( \Lambda V_0-\Lambda Q\right) G_B(\varepsilon)- 2 \int \psi_B Q\varepsilon -\int \psi_B'(\Lambda Q)' \varepsilon +\int (\varphi_B-\psi_B)\Lambda Q \, \varepsilon
\\ & \quad -\int \Lambda Q \psi_B \left[ (W_b+\varepsilon)^5-W_b^5-5W_b^4\varepsilon\right] -5\int \Lambda Q \psi_B \left[W_b^4-Q^4\right]\varepsilon
\end{align*}
and
\begin{align*}
\int \partial_y V_0 \, G_B(\varepsilon)
&= \int \left(\partial_y V_0-Q'\right) \, G_B(\varepsilon)-\int \psi_B'Q'' \varepsilon +\int (\varphi_B-\psi_B)Q' \varepsilon
\\ & \quad -\int Q' \psi_B \left[ (W_b+\varepsilon)^5-W_b^5-5W_b^4\varepsilon\right]-5\int Q' \psi_B \left[ W_b^4-Q^4\right]\varepsilon.
\end{align*}
First, using the identities $\Lambda V_0-\Lambda Q=\Lambda Q (\zeta-1)+yQ\partial_y\zeta$ and $\partial_y V_0-Q'=Q'(\zeta-1)+Q\partial_y\zeta$, we see that 
\[\left| \int \left( \Lambda V_0-\Lambda Q\right) G_B(\varepsilon)\right|+ \left| \int \left(\partial_y V_0-Q'\right) \, G_B(\varepsilon)\right| \lesssim |s|^{-100}\|\varepsilon\|_{L^2_\loc} . \]
Next, by the orthogonality relations \eqref{eq:or}, we also have $\int yQ' \varepsilon=\int ( \Lambda Q-\frac12Q) \varepsilon=0$.
Using the definitions of $\psi_B$, $\psi_B'$ and $\varphi_B$, it follows that 
\begin{align*}
&\left|\int \psi_B Q\varepsilon \right|=\left| \int (\psi_B-1) Q \varepsilon\right| \lesssim \int_{y<-\frac{B}2} e^{-\frac{|y|}2} |\varepsilon| \lesssim e^{-\frac{B}8} \|\varepsilon\|_{L^2_\loc} , \\
&\left|\int \psi_B' Q'' \varepsilon \right|+\left|\int \psi_B'(\Lambda Q)' \varepsilon \right| \lesssim \int_{y<-\frac{B}2} e^{-\frac{|y|}2} |\varepsilon| \lesssim e^{-\frac{B}8} \|\varepsilon\|_{L^2_\loc} ,
\end{align*}
and
\begin{align*}
&\left|\int (\varphi_B-\psi_B)\Lambda Q \, \varepsilon \right|=\left| \int(\varphi_B-\psi_B-\frac{y}B)\Lambda Q \, \varepsilon\right| \lesssim \int_{|y|>\frac{B}2} e^{-\frac{|y|}2}|\varepsilon| \lesssim e^{-\frac{B}8} \|\varepsilon\|_{L^2_\loc} , \\ 
&\left|\int (\varphi_B-\psi_B) Q' \varepsilon \right|=\left|\int (\varphi_B-\psi_B-\frac{y}B)Q' \varepsilon \right|
\lesssim \int_{|y|>\frac{B}2} e^{-\frac{|y|}2}|\varepsilon| \lesssim e^{-\frac{B}8} \|\varepsilon\|_{L^2_\loc} .
\end{align*}

Next, we get from \eqref{cons:BS:Wb3Wb'} that 
\begin{equation*}
\int \psi_B\left(|\Lambda Q|+|Q'|\right) \left|W_b^4-Q^4\right||\varepsilon|
\lesssimD |s|^{-1}\int e^{-\frac{|y|}2}|\varepsilon| \lesssimD |s|^{-1}\|\varepsilon\|_{L^2_\loc} .
\end{equation*}
Finally, we deduce from \eqref{est:W0:b} and \eqref{GN:BS}
that 
\begin{align*}
&\int \psi_B\left(|\Lambda Q|+|Q'|\right)\left|(W_b+\varepsilon)^5-W_b^5-5W_b^4\varepsilon\right| 
\\ & \lesssim \int \psi_B e^{-\frac{|y|}2} \left(|W_b|^3\varepsilon^2+|\varepsilon|^5\right)
\lesssim \left(\|W_b\|_{L^{\infty}}^3+\|\varepsilon\|_{L^{\infty}}^3\right)\|\varepsilon\|_{L^2_\loc}^2
\lesssim \|\varepsilon\|_{L^2_\loc}^2.
\end{align*}

Combining these estimates with \eqref{eL2B}, \eqref{eq:5n}, \eqref{eq:1c} and taking $B$ large enough, we obtain
\begin{equation} \label{est:f21}
|f_{2,1}| \le \frac{\mu_1}4 \int \varepsilon^2\varphi_B'+C_\delta |s|^{-8}(\log|s|)^2 .
\end{equation}
We fix $B$ to such a value (independent of $C_3^\star$ and $\delta$) and we do not track anymore this constant.

To estimate $\mathrm{f}_{2,2}$, we need an estimate on the error term $\Psi = \Psi_M + \Psi_{W} + \Psi_b$.
\begin{lemma}\label{BS:lemma}
For all $s\in \cI$, 
$\cN_B(\Psi) \lesssimD |s|^{-1} \|\varepsilon\|_{L^2_\loc} + s^{-4} \log |s|$.
\end{lemma}
\begin{proof}[Proof of Lemma~\ref{BS:lemma}]
First, we observe that \eqref{eq:PW} and \eqref{eq:2n}--\eqref{eq:4n} imply 
\[
\cN_B(\Psi_W)\lesssimD |s|^{-4} \log |s|.
\]
(Recall that both the constant $C_1^\star$ and the constant $B$ have been fixed.)
Second, we also have 
$\cN_B(\Psi_b)\lesssimD |s|^{-4} \log |s|$ by \eqref{est:Psi0:b} and \eqref{eq:2n}--\eqref{eq:4n}.
(Again, we recall that the constant $C_2^\star$ have already been fixed.)

Last, we claim that $\cN_B(\Psi_M) \lesssimD |s|^{-1} \|\varepsilon\|_{L^2_\loc} + s^{-4} \log |s|$.
By the definition of $\Psi_M$ in \eqref{def:PsiM},
\begin{align*}
\cN_B(\Psi_M)
&\lesssim \left|\vec m\right| \left( \cN_B(\Psi_{\lambda})+\cN_B(\Psi_{\sigma})+|b|\cN_B(\Lambda P_b) + |b| \cN_B(\partial_y P_b)\right)\\
&\quad + |b_s| \cN_B\left((\chi_b+\gamma y\partial_y\chi_b)P\right)
+|b| \left( \cN_B( \Psi_{\lambda})+|b|\cN_B(\Lambda P_b)\right).
\end{align*}
From \eqref{eq:1c}, \eqref{eq:2c}, \eqref{eq:2n}--\eqref{eq:4n}, we have
$\left|\vec m \right|\lesssimD \|\varepsilon\|_{L^2_\loc} + |s|^{-4} \log |s|$,
$|b|\lesssimD |s|^{-3}\log|s|$ and
$|b_s|\lesssimD |s|^{-1}\|\varepsilon\|_{L^2_\loc} + |s|^{-4} \log |s|$.
Next, from \eqref{eq:ls}-\eqref{eq:ss}, 
$\cN_B(\Psi_\lambda)+\cN_B(\Psi_\sigma)\lesssim |s|^{-1}$
and from the definition of $P$ and \eqref{Lambda:P:b.1}, 
\[
\cN_B(\Lambda P_b)+\cN_B(\partial_y P_b)+\cN_B\left((\chi_b+\gamma y\partial_y\chi_b)P\right)\lesssim 1.
\]
These estimates are sufficient to prove the claim on $\Psi_M$.
\end{proof}

We compute from the definition of $G_B(\varepsilon)$
\begin{equation*} 
\mathrm{f}_{2,2}=2\int \psi_B (\partial_y\Psi) (\partial_y \varepsilon)+2\int \varphi_B \Psi \, \varepsilon
-2\int \psi_B \Psi \left[(W_b+\varepsilon)^5-W_b^5 \right] .
\end{equation*}
We deduce from the Cauchy-Schwarz inequality, \eqref{cut:onR}, Lemma~\ref{BS:lemma} and then~\eqref{eq:5n} and~\eqref{eL2B} that
\begin{align*}
\left|\int \psi_B (\partial_y\Psi) (\partial_y\varepsilon) \right|
&\leq \cN_B(\Psi) \cN_B(\varepsilon)
\lesssimD \left(|s|^{-1} \|\varepsilon\|_{L^2_\loc}+ |s|^{-4} \log |s| \right) \cN_B(\varepsilon)\\
&\lesssimD 
|s|^{-\frac 12} \|\varepsilon\|_{L^2_\loc}^2
+|s|^{-\frac 32}[\cN_B(\varepsilon)]^2
+C_3^\star |s|^{-7} (\log |s|)^2 \\
&\lesssimD |s|^{-\frac 12} \int \varepsilon^2 \varphi_B'+C_3^\star |s|^{-7} (\log |s|)^2 .
\end{align*}
Similarly,
\begin{equation*}
\left|\int \varphi_B \Psi \varepsilon \right|
\le \cN_B(\Psi) \cN_B(\varepsilon)
\lesssimD |s|^{-\frac 12} \int \varepsilon^2 \varphi_B'+C_3^\star |s|^{-7} (\log |s|)^2.
\end{equation*}
Moreover, since $\psi_B\leq 3 \varphi_B$ and
$\|W_b\|_{L^{\infty}}+\|\varepsilon\|_{L^{\infty}}\lesssim 1$ by
\eqref{est:W0:b} and \eqref{GN:BS}, we check
\begin{equation*}
\left|\int \psi_B \Psi \left[(W_b+\varepsilon)^5-W_b^5 \right] \right|
\lesssim \int \varphi_B |\Psi|\left(|W_b|^4 |\varepsilon|+|\varepsilon|^5 \right)
\lesssim \cN_B(\Psi)\cN_B(\varepsilon).
\end{equation*}
Thus,
\begin{equation} \label{est:f22}
|f_{2,2}| \leq \frac{\mu_1}4 \int \varepsilon^2\varphi_B'+ C_\delta C_3^\star|s|^{-7}(\log|s|)^2 ,
\end{equation}
which together with \eqref{est:f21} yields \eqref{est:f2}. 

\smallskip 

\noindent \emph{Estimate for $f_3$.}
Integrating by parts, we compute
\begin{align*}
f_3 =\left(\frac{\sigma_s}{\lambda} - 1\right) 
\int & \biggl[ - \psi_B' (\partial_y\varepsilon)^2 -\varphi_B' \varepsilon^2 
+\frac 13 \psi_B' \left( (W_b+\varepsilon)^6- W_b^6 - 6 W_b^5 \varepsilon\right) \\
&+ 2 \psi_B (\partial_y W_b) \left( (W_b+\varepsilon)^5 - W_b^5 - 5 W_b^4 \varepsilon\right)\biggr]
\end{align*}
By \eqref{eq:1c}, \eqref{eL2B} and \eqref{eq:5n}, we have $|\vec m|\lesssimD C_3^\star |s|^{-3}\log|s|$.
From \eqref{est:W0:b}-\eqref{est:W12:b} and~\eqref{GN:BS}, $\|W_b\|_{L^\infty}+\|\partial_y W_b\|_{L^\infty}+\|\varepsilon\|_{L^\infty}
\lesssim 1$.
Thus, by \eqref{eq:5n} and the properties of $\varphi_B$, $\psi_B$,
\begin{equation}\label{est:f3}
 |f_3| \lesssim |\vec{m}| [\cN_b(\varepsilon)]^2 \lesssimD (C_3^\star)^3|s|^{-9}(\log|s|)^3 \lesssim |s|^{-8}.
\end{equation}

\smallskip

\noindent \emph{Estimate for $f_4$.} We claim 
\begin{equation} \label{est:f4}
f_4 
\leq \frac{\mu_1}8\int \left( \varepsilon^2+(\partial_y\varepsilon)^2\right) \varphi_B'+C_\delta |s|^{-8} .
\end{equation}

By integration by parts (see also~\cite[page 97]{MMR1}), we have the identities
\begin{align*} 
\int (\Lambda \varepsilon) \partial_y(\psi_B\partial_y\varepsilon)
&=-\int \psi_B(\partial_y\varepsilon)^2+\frac12\int y\psi_B'(\partial_y\varepsilon)^2 , \\ 
\int (\Lambda \varepsilon) \varepsilon \varphi_B
&=-\frac12 \int y \varphi_B' \varepsilon^2\\
\int (\Lambda \varepsilon) \psi_B \left( (W_b+\varepsilon)^5-W_b^5\right)
&=\frac16 \int (2\psi_B-y\psi_B')\left( (W_b+\varepsilon)^6-W_b^6-6W_b^5\varepsilon \right)\\
&\quad 
-\int \psi_B \Lambda W_b\left( (W_b+\varepsilon)^5-W_b^5-5W_b^4\varepsilon \right).
\end{align*}
Hence,
\begin{align*}
f_4
&=-\frac{\lambda_s}{\lambda}\left(\int y\psi_B'(\partial_y\varepsilon)^2+2\int \varphi_B \varepsilon^2+\int y\varphi_B' \varepsilon^2\right)\\
& \quad +\frac13 \frac{\lambda_s}{\lambda}\int y\psi_B' \left( (W_b+\varepsilon)^6-W_b^6-6W_b^5\varepsilon \right)\\
& \quad +2 \frac{\lambda_s}{\lambda}\int \psi_B \Lambda W_b \left( (W_b+\varepsilon)^5-W_b^5-5W_b^4\varepsilon \right)\\
&=:f_{4,1}+f_{4,2}+f_{4,3} .
\end{align*}
First, from \eqref{eq:1c}, \eqref{eq:2n}--\eqref{eq:4n}, \eqref{eq:5n} and then the definition of $\beta$ in \eqref{eq:bt}
(recall that $c_1=-(2\alpha+1)$ and thus $\beta<0$ for large $|s|$), we have
\[
\frac {\lambda_s}{\lambda} \geq - \beta - C_3^\star |s|^{-3}\log|s|\geq (2\alpha+1) |s|^{-1} -C|s|^{-2} \geq |s|^{-1}>0,
\]
so that using $\psi_B'\geq 0$, $\varphi_B\geq 0$ and $\varphi_B'\geq 0$,
\[
-\frac{\lambda_s}{\lambda}\left(\int_{y>0} y\psi_B'(\partial_y\varepsilon)^2+2\int \varphi_B \varepsilon^2
+\int_{y>0} y\varphi_B' \varepsilon^2\right) \leq 0
\]
Thus, by $\big|\frac {\lambda_s}{\lambda}\big|\lesssim |s|^{-1}$,
the properties of $\psi_B$, $\varphi_B$ and the H\"older and Young inequalities
\begin{align*}
f_{4,1}
&\lesssim |s|^{-1} \int_{y<0} |y| e^{\frac{y}B} \left(\varepsilon^2+(\partial_y\varepsilon)^2\right)\\
&\lesssim |s|^{-1} \left(\int_{y<0} |y|^{100} e^{\frac{y}B}\left( \varepsilon^2+(\partial_y\varepsilon)^2\right)\right)^{\frac 1{100}}
\left( \int_{y<0} \left(\varepsilon^2+(\partial_y\varepsilon)^2\right) \varphi_B'\right)^{\frac{99}{100}}\\
&\lesssimD 
|s|^{-\frac{101}2}
+|s|^{-\frac 12} \int_{y<0} \left( \varepsilon^2+(\partial_y\varepsilon)^2\right) \varphi_B' 
\end{align*}
where we have also used $\int_{y<0} |y|^{100} e^{\frac{y}B}\left( \varepsilon^2+(\partial_y\varepsilon)^2\right)\lesssim \|\varepsilon\|_{H^1}^2 \lesssimD 1$ (see \eqref{eq:2n}--\eqref{eq:4n}).

To deal with $f_{4,2}$, we first recall the observation that $\psi_B'\equiv 0$ for $y>-\frac B2$. Moreover, note that using \eqref{est:W0:b} and \eqref{GN:BS} 
\[
\left| (W_b+\varepsilon)^6-W_b^6-6W_b^5\varepsilon \right| \lesssim |W_b|^4 \varepsilon^2 + \varepsilon^6
\lesssimD \omega \varepsilon^2 + \lambda^2\varepsilon^2.
\]
Then, it follows arguing as for $f_{4,1}$ that 
\begin{align*}
 |f_{4,2}| &\lesssimD |s|^{-1}\|\varepsilon\|_{L^2_{\loc}}^2
 +|s|^{-3}\int_{y<0} |y|e^{\frac{y}B} \varepsilon^2 \\ &
 \lesssimD |s|^{-1} \int \varepsilon^2 \varphi_B'
 +|s|^{-3}\left(\int_{y<0}|y|^{100}e^{\frac{y}B}\varepsilon^2 \right)^{\frac1{100}}\left(\int \varepsilon^2\varphi_B' \right)^{\frac{99}{100}}\\ 
 & \lesssimD |s|^{-201}+|s|^{-1} \int \varepsilon^2 \varphi_B'.
\end{align*}

To deal with $f_{4,3}$, we observe from \eqref{est:W0:b}--\eqref{est:W12:b},
and \eqref{GN:BS} that
\begin{align*}
\left|\Lambda W_b \left[ (W_b+\varepsilon)^5-W_b^5-5W_b^4\varepsilon \right] \right| 
& \lesssim (|W_b|+|y|\partial_yW_b|) \left(|W_b|^3\varepsilon^2 +|\varepsilon|^5\right)\\
& \lesssimD \omega \varepsilon^2 + \lambda^2 (1+ \lambda|y|\eta)\varepsilon^2\\
& \lesssimD \omega \varepsilon^2 + |s|^{-2} (1+ |y|\eta) \varepsilon^2.
\end{align*}
By the choice of the function $\varphi$ in \eqref{eq:5P}, we have
\[
\int (1+|y|\eta)\varepsilon^2 \psi_B \lesssim \int \varepsilon^2 \varphi_B \lesssim [\cN_B(\varepsilon)]^2.
\]
Hence, it follows from \eqref{cut:onR}, \eqref{eL2B}, and then \eqref{eq:5n} that
\begin{equation*}
 |f_{4,3}| \lesssimD |s|^{-1} \|\varepsilon\|_{L^2_\loc}^2+|s|^{-3}[\cN_B(\varepsilon)]^2 
 \lesssimD |s|^{-1} \int \varepsilon^2 \varphi_B'+(C_3^\star)^2|s|^{-9}\log|s|.
\end{equation*}
\begin{remark}\label{rk:pb}
Note that the term $f_{4,3}$ was the motivation in the present paper for the particular
choice of $\varphi$ in \eqref{eq:5P}.
Indeed, it allows to absorb in the term $\cN_B(\varepsilon)$ the function $|y|$ that appears in the definition of the operator $\Lambda$.
\end{remark}

We conclude the proof of \eqref{est:f4} by combining the estimates for $f_{4,1}$, $f_{4,2}$ and $f_{4,3}$. 

\smallskip

\noindent \emph{Estimate for $f_5$.} We claim that for $|s|$ large enough possibly depending  $C_3^\star$, 
\begin{equation} \label{est:f5}
|f_5| \leq \frac{\mu_1}{8}\int \varepsilon^2 \varphi_B'+C_\delta |s|^{-8}. 
\end{equation}

Using \eqref{est:W0:b} and \eqref{GN:BS}, we have 
\[
\left| (W_b+\varepsilon)^5 - W_b^5 - 5 W_b^4 \varepsilon \right|
\lesssim (|W_b|^3  +|\varepsilon|^3) \varepsilon^2
\lesssim (\omega + \lambda^\frac32) \varepsilon^2.
\]
and so
\begin{equation}\label{eq:mm}
|f_5|\lesssim \int |\partial_s W_b| (\omega + \lambda^\frac32) \varepsilon^2\psi_B.
\end{equation}
We claim
\begin{equation}\label{on:Wb}
|\partial_s W_b|\lesssim |s|^{-1}\eta_L + |s|^{-\frac32} (1+ |y| \eta).
\end{equation}
To prove \eqref{on:Wb}, we use the definition of $W$ in \eqref{eq:WW} and the definition of $W_b$ in \eqref{def:W:b}, and we compute
\begin{align*}
\partial_s W_b & = \partial_s W + \partial_s (bP_b) \\
& = |s|^{-1}y \chi_L'(|s|^{-1}y) \left( Q+c_1A_1\theta+(c_2P+A_2)\theta^2
+A_2^*\partial_y \theta + A_3 \theta^3\right)\\
& \quad + V_1 \partial_s\theta + V_2 \partial_s (\theta^2) + V_2^* \partial_s\partial_y\theta
+V_3 \partial_s (\theta^3)\\
& \quad +b_s\left( \chi_b + \gamma y \partial_y \chi_b\right) P.
\end{align*}
Inserting (recall \eqref{eq:TA})
\[
\partial_s \theta = \frac{\lambda_s}{\lambda}\Lambda\theta +   \frac{\sigma_s}{\lambda} 
\partial_y \theta -\partial_y^3 \theta - m_0^4 \partial_y (\theta^5)
\]
we obtain
\begin{align*}
\partial_s W_b 
& = |s|^{-1}y \chi_L'(|s|^{-1}y) \left( Q+c_1A_1\theta+(c_2P+A_2)\theta^2
+A_2^*\partial_y \theta + A_3 \theta^3\right)\\
&\quad + \frac{\lambda_s}\lambda \Lambda\theta \left(V_1  + 2V_2 \theta  +3 V_3 \theta^2\right)
+ \frac{\lambda_s}\lambda \left(\partial_y\Lambda\theta\right)  V_2^* 
\\
&\quad + \frac{\sigma_s}\lambda \partial_y\theta \left(V_1  + 2V_2 \theta  +3 V_3 \theta^2\right)
+ \frac{\sigma_s}\lambda \left(\partial_y^2\theta\right)  V_2^*\\
& \quad -\left(\partial_y^3 \theta + m_0^4 \partial_y (\theta^5)\right)
\left(V_1  + 2V_2 \theta  +3 V_3 \theta^2\right)
- \left(\partial_y^4 \theta + m_0^4 \partial_y^2 (\theta^5)\right) V_2^*
\\
& \quad +b_s\left( \chi_b + \gamma y \partial_y \chi_b\right) P.
\end{align*}
First, we have 
\[
|s|^{-1} |y \chi_L'(|s|^{-1}y)| \left| Q+c_1A_1\theta+(c_2P+A_2)\theta^2
+A_2^*\partial_y \theta + A_3 \theta^3\right|
\lesssim |s|^{-1} \eta_L.
\]
Second, using \eqref{eq:lt}, we have, for $y\ge -|s|$,
\begin{align*}|
\theta|+|\Lambda\theta| &\lesssim   (|s|^{-1}+|s|^{-2} |y|)\eta +\lambda^\frac 12 \eta_R ; \\ 
|\partial_y \theta| +|\partial_y \Lambda\theta|& \lesssim (|s|^{-2}+|s|^{-3} |y|)\eta+\lambda \eta_R, \\
|\partial_y^2\theta|  & \lesssim |s|^{-3}\eta+\lambda^\frac32\eta_R. 
\end{align*}
This yields, using $\big|\frac{\lambda_s}{\lambda}\big|\lesssim |s|^{-1}$ and Lemma \ref{le:vv},
\[
\left|  \frac{\lambda_s}\lambda \right| \left| \Lambda\theta \left(V_1  + 2V_2 \theta  +3 V_3 \theta^2\right)
+  \left(\partial_y\Lambda \theta \right)  V_2^* \right|
\lesssim |s|^{-1}(\lambda^\frac12 +|s|^{-2} |y|)\eta
\lesssim |s|^{-\frac32} (1+|y)\eta,
\]
and, using $\big|\frac{\sigma_s}{\lambda}\big|\lesssim 1$ and Lemma \ref{le:vv},
\[
\left| \frac{\sigma_s}\lambda\right| \left| \partial_y\theta \left(V_1  + 2V_2 \theta  +3 V_3 \theta^2\right)
+  \left(\partial_y^2\theta\right)  V_2^* \right|\lesssim |s|^{-1}\lambda^\frac12 \eta
\lesssim |s|^{-\frac32} \eta.
\]
Third, using $|b_s|\lesssimD C_3^\star |s|^{-4}\log|s|$, we have
\[
\left|b_s\left( \chi_b + \gamma y \partial_y \chi_b\right) P \right| 
\lesssimD C_3^\star |s|^{-4}\log|s|.
\]
Thus, \eqref{on:Wb} is proved.

It  follows from \eqref{eq:mm} and \eqref{on:Wb} that
\begin{align*}
|f_5| & \lesssim |s|^{-\frac32} \|\varepsilon\|_{L^2_\loc}^2 + 
|s|^{-3} \int (1+|y|\eta) \varepsilon^2 \psi_B\\
&\lesssim |s|^{-\frac32} \|\varepsilon\|_{L^2_\loc}^2 +|s|^{-3} [\cN_B(\varepsilon)]^2 \\
&\lesssim |s|^{-\frac32} \|\varepsilon\|_{L^2_\loc}^2 
+(C_3^\star)^2 |s|^{-9}
(\log|s|)^2  ,
\end{align*}
which yields \eqref{est:f5} for $|s|$ large enough.

Finally, we conclude the proof of \eqref{time:local} by gathering \eqref{est:f1}, \eqref{est:f2}, \eqref{est:f3}, \eqref{est:f4} and \eqref{est:f5}. 

\smallskip

We turn to the proof of \eqref{coer:local}. We decompose $\cF$ as follows: 
\begin{align*}
 \lambda^2\cF &=\int \left[(\partial_y\varepsilon)^2 \psi_B + \varepsilon^2 \varphi_B-5Q^4\varepsilon^2 \right]
- \frac13 \int \left[ (W_b+\varepsilon)^6 - W_b^6 - 6 W_b ^5 \varepsilon-15Q^4\varepsilon^2\right] \psi_B \\ 
&=: \cF_1+\cF_2 .
\end{align*}
To bound $\cF_1$ from below, we rely on the coercivity of the linearized operator $\cL$ around the ground state under the orthogonality conditions \eqref{eq:or} and standard localisation arguments. Proceeding for instance as \cite[Appendix A]{MM1}, we deduce that there exists $\tilde{\mu}_2>0$ such that, for $B$ large enough, 
$\cF_1 \ge \tilde{\mu}_2 [\cN_B(\varepsilon)]^2 $.

To estimate $\cF_2$, we write 
\begin{equation*}
 (W_b+\varepsilon)^6 - W_b^6 - 6 W_b ^5 \varepsilon-15Q^4\varepsilon^2
 =(W_b+\varepsilon)^6 - W_b^6 - 6 W_b ^5 \varepsilon-15W_b^4\varepsilon^2-15\left( W_b^4-Q^4\right)\varepsilon^2 ,
\end{equation*}
so that 
\begin{equation*}
 \left|(W_b+\varepsilon)^6 - W_b^6 - 6 W_b ^5 \varepsilon-15Q^4\varepsilon^2 \right|
 \lesssim \left(|W_b|^3|\varepsilon|+\varepsilon^4\right)\varepsilon^2+\left| W_b^4-Q^4\right|\varepsilon^2.
\end{equation*}
Thus, it follows from \eqref{cut:onR}, and then \eqref{est:W0:b}, \eqref{cons:BS:Wb3Wb'} and \eqref{GN:BS} that
\begin{align*}
 \left|\cF_2 \right| \lesssim \left( \|W_b\|_{L^{\infty}}^3\|\varepsilon\|_{L^{\infty}}+\|\varepsilon\|_{L^{\infty}}^4+\left\| W_b^4-Q^4\right\|_{L^{\infty}}\right) \int \varepsilon^2\varphi_B 
 \lesssimD \lambda^\frac12[\cN_B(\varepsilon)]^2 .
\end{align*}
The proof of \eqref{coer:local} follows from these estimates taking $|s|$ large enough
\end{proof}

\subsection{Closing the energy estimates}
For $C_3^\star$ sufficiently large, we improve the bootstrap estimate 
\eqref{eq:5n} using Proposition \ref{pr:local-energy}.

\begin{lemma} \label{lemma:eps_loc:BS}
There exists $C_3^\star>1$ such that on $\cI$,
\begin{equation}\label{eq:9n}
\cN_B(\varepsilon(s)) 
+\left(\int_{S_n}^s \left(\frac{\tilde{s}} s\right)^{\frac{2\alpha+1}{\alpha}} \|\varepsilon(\tau)\|_{L^2_\loc}^2\, d\tilde{s} \right)^{\frac 12} 
\leq \frac{C_3^\star}2 |s|^{-3} \log |s|. 
\end{equation}
\end{lemma} 

\begin{proof}

Let $s\in\cI$.
Integrating~\eqref{time:local} on $[S_n,s]$ and using $\cF(S_n)=0$, we find
\[
\cF(s) + \int_{S_n}^s \tilde{s}^{\frac{2\alpha+1}{\alpha}}  \int \left( (\partial_y \varepsilon)^2 + \varepsilon^2 \right) \varphi_B' d\tilde{s}
\lesssimD C_3^\star |s|^{-6+\frac{2\alpha+1}{\alpha}} (\log|s|)^2 .
\]
Thus, by \eqref{eL2B} and~\eqref{coer:local}, we obtain
\[
[\cN_B(\varepsilon(s))]^2 
+s^{-{\frac{2\alpha+1}{\alpha}} } \int_{S_n}^s \tilde{s}^{\frac{2\alpha+1}{\alpha}}  \|\varepsilon(\tau)\|_{L^2_\loc}^2 d\tilde{s} 
\lesssimD C_3^\star s^{-6} (\log |s|)^2.
\]
By choosing $C_3^\star>1$ large enough, this implies \eqref{eq:9n}.
\end{proof}

We complete the proof of Proposition~\ref{prop:bootstrap}. The constants $C_1^\star>1$
and $C_2^\star$ have been fixed as in Lemma~\ref{le:pB}, and the constant $C_3^\star>1$ is now chosen as in Lemma~\ref{lemma:eps_loc:BS}.
Assuming by contradiction that $S_n^{\star}<s_0$, by continuity, it follows that at least one of the estimates in \eqref{eq:H1b}, \eqref{eq:2n}--\eqref{eq:4n} and \eqref{eq:5n} is reached, which is absurd by Lemmas~\ref{le:pB}, \ref{le:mass-ener} and \ref{lemma:eps_loc:BS}.

\section{Proof of the main result}\label{S:7}

\subsection{Uniform estimates in the original variables}

By Proposition~\ref{prop:bootstrap}, the sequence $\{u_n\}$ of solutions of \eqref{eq:KV}
defined in Section~\ref{S:5.1} satisfy the uniform estimates~\eqref{eq:1n}--\eqref{eq:5n} on $[S_n,s_0]$ where $s_0<-1$ is independent of $n$. We rewrite these estimates in the time variable $t$.
The variable $t$ is exactly $\tau_n(s)$ and thus
\[
t = \kappa_\tau |s|^{-\frac{4\alpha+3}{2\alpha}} + O(|s|^{-\frac{4\alpha+3}{2\alpha}-1}\log|s|).
\]
Thus, estimates~\eqref{eq:2n}--\eqref{eq:4n}, \eqref{GN:BS} and \eqref{eq:EU}
on $[S_n,s_0]$ imply the existence of $t_0>0$ independent of $n$ such that on $[T_n,t_0]$, the following estimates in the $t$ variable
hold
\begin{equation}\label{BS:t}\begin{aligned}
|b_n(t)|&\lesssimD t^{\frac{6\alpha}{4\alpha+3}} |\log t|,\\
\left| \lambda_n(t) - (4\alpha+3)^\nu t^\nu \right| 
&\lesssimD t^{\nu+\frac{2\alpha}{4\alpha+3}}|\log t|,\\
\left| \sigma_n(t) - (4\alpha+3)^\frac{1}{4\alpha+3} t^\frac1{4\alpha+3} \right|
&\lesssimD t^{\frac1{4\alpha+3}+\frac{2\alpha}{4\alpha+3}}|\log t|,\\
\|\varepsilon_n(t)\|_{L^2} \lesssim t^{\frac\alpha{4\alpha+3}},\quad  
\|\partial_y \varepsilon_n(t)\|_{L^2} &\lesssim \delta^{\alpha-1} t^{\nu},
\end{aligned}
\end{equation}
where we recall that $\nu=\frac{2\alpha+1}{4\alpha+3}$ (see \eqref{eq:ab}). Now, for any $t\in [T_n,t_0]$, $x\in \RR$, we set
\begin{equation}\label{eq:rn}
r_n(t,x)= u_n(t,x) - \frac 1{((4\alpha+3)t)^\frac\nu2} Q\left( \frac {x-\sigma_n(t)}{((4\alpha+3)t)^{\nu}}\right),
\end{equation}
and we decompose it as
\begin{align*}
r_n(t,x)
&= \frac 1{\lambda_n^\frac 12(t)} Q\left( \frac {x-\sigma_n(t)}{\lambda_n(t)}\right)
- \frac 1{((4\alpha+3)t)^\frac\nu2} Q\left( \frac {x-\sigma_n(t)}{((4\alpha+3)t)^{\nu}}\right)\\
&\quad 
+u_n(t,x)-\frac 1{\lambda_n^\frac 12(t)} Q\left( \frac {x-\sigma_n(t)}{\lambda_n(t)}\right)
= r_{1,n}(t,x)+r_{2,n}(t,x).
\end{align*}
Introducing 
$\bar \lambda_n(t) = {((4\alpha+3)t)^\nu}/\lambda_n(t)$ and using \eqref{BS:t}, we have
\begin{align*}
\|r_{1,n}(t)\|_{L^2} & =\left\| \bar \lambda_n^\frac 12(t) Q (\bar\lambda_n(t)\, \cdot \, ) - Q\right\|_{L^2}
\lesssim |\bar \lambda_n(t) -1| \lesssim t^{\frac{2\alpha}{4\alpha+3}} |\log t|,\\
t^{\nu} \|\partial_x r_{1,n}(t)\|_{L^2} & = \left\| \bar \lambda_n^\frac 32(t) Q' (\bar\lambda_n(t) \,\cdot \, ) - Q'\right\|_{L^2}
\lesssim |\bar \lambda_n(t) -1| \lesssim t^{\frac{2\alpha}{4\alpha+3}} |\log t|.
\end{align*}
Moreover, by the triangle inequality
\begin{align*}
\|r_{2,n}-(2\alpha+1) m_0 \Theta_0\|_{L^2}
&  = \|W_{b,n}+\varepsilon_n-Q-(2\alpha+1) m_0 \Theta_0\|_{L^2}\\
& \lesssim \|W_{b,n}-W_n\|_{L^2}
+ \|W_n-Q-(2\alpha+1)m_0 \theta_0\|_{L^2} + \|\varepsilon_n\|_{L^2}
\end{align*}
and so using \eqref{BS:t} for $b_n$ and $\|\varepsilon_n\|_{L^2}$, and \eqref{eq:33}, we obtain
\[
\|r_{2,n}-(2\alpha+1) m_0 \Theta_0\|_{L^2}\lesssim t^{\frac\alpha{4\alpha+3}}.
\]
Finally,  by \eqref{est:V:H1}, \eqref{def:P:b} and \eqref{BS:t},
\begin{align*}
t^{\nu}\|\partial_x r_{2,n}\|_{L^2} 
&\leq \|\partial_y V\|_{L^2}+\|b \partial_y P_{b}\|_{L^2}+\|\partial_y \varepsilon\|_{L^2}\lesssim t^{\frac{2\alpha}{4\alpha+3}} +|b|+ \delta^{\alpha-1} t^{\nu}
\lesssimD t^{\frac{2\alpha}{4\alpha+3}}.
\end{align*}
In conclusion on all the above estimates, we have constructed a sequence $\{u_n\}_n$
of $H^1$ solutions of the equation \eqref{eq:KV} each existing on an interval
$[T_n,t_0]$, where $T_n\to 0$ as $n\to +\infty$, and
a sequence of functions $\sigma_n$ on $[T_n,t_0]$ satisfying the following uniform
estimates on $[T_n,t_0]$,
\begin{equation}\label{eq:fin}
\begin{aligned}
|\sigma_n(t)|&\lesssim t^{\frac1{4\alpha+3}},\\ 
\|r_n(t)-(2\alpha+1)m_0 \Theta_0\|_{L^2} & \lesssim t^{\frac\alpha{4\alpha+3}},\\
t^\nu\|\partial_x r_n(t)\|_{L^2} & \lesssim t^{\frac{2\alpha}{4\alpha+3}},
\end{aligned}
\end{equation}
where
\[
u_n(t,x)=  \frac 1{((4\alpha+3)t)^\frac\nu2} Q\left( \frac {x-\sigma_n(t)}{((4\alpha+3)t)^{\nu}}\right) +r_n(t,x).
\]

\subsection{Passing to the limit}
First, observe that \eqref{eq:fin} taken at the time $t=t_0$
implies that the sequence  $\{\sigma_n(t_0)\}_n$ is bounded in $\RR$
and that the sequence $\{u_n(t_0)\}_n$ is bounded in $H^1$.

Therefore, there exists a subsequence  $\{ u_{n_k}\}_k$ and $\tilde u_0\in H^1$ 
such that $ u_{n_k} \rightharpoonup \tilde u_0$ weakly in $H^1$
  as $k\to \infty$.
 
Applying \cite[Lemma 2.10]{CoM1},
it follows that the solution~$\tilde u$ of~\eqref{eq:KV} with initial data $\tilde u(t_0)=\tilde u_0$ exists on the time interval $(0,t_0]$ and that for all $t\in (0,t_0]$,
$u_{n_k}(t)\rightharpoonup \tilde u(t)$ in $H^1$ weak as $k\to +\infty$
(weak continuity of the gKdV flow).

Moreover,
\[
\tilde u(t,x)=  \frac 1{((4\alpha+3)t)^\frac\nu2} Q\left( \frac {x-\tilde\sigma(t)}{((4\alpha+3)t)^{\nu}}\right) +\tilde r(t,x),
\]
where  on $(0,t_0]$, the function $\tilde \sigma(t)$ and $\tilde r(t,x)$ satisfy
\begin{equation}\label{eq:fi}
\begin{aligned}
|\tilde\sigma(t)|&\lesssim t^{\frac1{4\alpha+3}},\\ 
\|\tilde r(t)-(2\alpha+1)m_0 \Theta_0\|_{L^2} & \lesssim t^{\frac\alpha{4\alpha+3}},\\
t^\nu\|\partial_x \tilde r(t)\|_{L^2} & \lesssim t^{\frac{2\alpha}{4\alpha+3}}.
\end{aligned}
\end{equation}
In particular,
$\tilde u$ is an $H^1$ solution of \eqref{eq:KV} blowing up at time $t=0$ at the point $0$
and after the simple rescaling
\[
 u(t,x) = \lambda_0^\frac12 \tilde u(\lambda_0^3 t,\lambda_0 x),
\quad \lambda_0 = (4\alpha+3)^{\frac\nu{1-3\nu}},
\]
the solution  $ u(t)$ satisfies the conclusions of Theorem~\ref{th:01}.

\section{Statements}

Part of this work was done while the first author (Yvan Martel) was visiting the Department of Mathematics, University of Bergen.
The second author (Didier Pilod)  was partially supported by a grant from the Trond Mohn Foundation. He also would like to thank the Laboratoire de Mathématiques de Versailles for the kind hospitality during part of the writing of this work.

\end{document}